\documentclass[reqno,10pt,a4paper,final]{amsart}
\usepackage[a4paper,left=20mm,right=20mm,top=30mm,bottom=30mm,marginpar=25mm]{geometry} 
\usepackage{amsmath}

\usepackage[citecolor=red]{hyperref}
\usepackage{amssymb}
\usepackage{amsthm}
\usepackage{amscd}
\usepackage[ansinew]{inputenc}
\usepackage{cite}
\usepackage{bbm}
\usepackage{xcolor}
\usepackage[english=american]{csquotes}
\usepackage[final]{graphicx}
\graphicspath{{Pictures/}}
\numberwithin{equation}{section}
\hypersetup{
colorlinks,
linkcolor={cyan!90!black},
citecolor={magenta},
urlcolor={green!40!black}
}

\newtheoremstyle{thmlemcorr}{10pt}{10pt}{\itshape}{}{\bfseries}{.}{10pt}{{\thmname{#1}\thmnumber{ #2}\thmnote{ (#3)}}}
\newtheoremstyle{thmlemcorr*}{10pt}{10pt}{\itshape}{}{\bfseries}{.}\newline{{\thmname{#1}\thmnumber{ #2}\thmnote{ (#3)}}}
\newtheoremstyle{defi}{10pt}{10pt}{\itshape}{}{\bfseries}{.}{10pt}{{\thmname{#1}\thmnumber{ #2}\thmnote{ (#3)}}}
\newtheoremstyle{remexample}{10pt}{10pt}{}{}{\bfseries}{.}{10pt}{{\thmname{#1}\thmnumber{ #2}\thmnote{ (#3)}}}
\newtheoremstyle{ass}{10pt}{10pt}{}{}{\bfseries}{.}{10pt}{{\thmname{#1}\thmnumber{ A#2}\thmnote{ (#3)}}}
\theoremstyle{thmlemcorr}
\newtheorem{theorem}{Theorem}
\numberwithin{theorem}{section}
\newtheorem{lemma}[theorem]{Lemma}
\newtheorem{corollary}[theorem]{Corollary}
\newtheorem{proposition}[theorem]{Proposition}

\theoremstyle{thmlemcorr*}
\newtheorem{theorem*}{Theorem}
\newtheorem{lemma*}[theorem]{Lemma}
\newtheorem{corollary*}[theorem]{Corollary}
\newtheorem{proposition*}[theorem]{Proposition}
\newtheorem{problem*}[theorem]{Problem}
\newtheorem{conjecture*}[theorem]{Conjecture}
\theoremstyle{defi}
\newtheorem{definition}[theorem]{Definition}
\theoremstyle{remexample}
\newtheorem{remark}[theorem]{Remark}

\theoremstyle{ass}

\newcommand{\Fcal}{\mathcal{F}}
\newcommand{\Gcal}{\mathcal{G}}

\newcommand{\Ical}{\mathcal{I}}
\newcommand{\Jcal}{\mathcal{J}}

\newcommand{\Mcal}{\mathcal{M}}
\newcommand{\Ncal}{\mathcal{N}}

\DeclareMathOperator{\supp}{supp}

\newcommand{\norm}[1]{\|#1\|}

\newcommand{\normb}[1]{\bigl\|#1\bigr\|}
\newcommand{\normB}[1]{\Bigl\|#1\Bigr\|}

\newcommand{\abs}[1]{|#1|}

\newcommand{\absb}[1]{\bigl|#1\bigr|}
\newcommand{\absB}[1]{\Bigl|#1\Bigr|}
\newcommand{\absBB}[1]{\biggl|#1\biggr|}

\newcommand{\floor}[1]{\lfloor #1 \rfloor}

\newcommand{\dpr}[1]{\langle #1 \rangle}

\newcommand{\dd}{\;\mathrm{d}}

\newcommand{\R}{\mathbb{R}}

\newcommand{\loc}{\mathrm{loc}}

\newcommand{\toweak}{\rightharpoonup}

\newcommand{\sbullet}{\begin{picture}(1,1)(-0.5,-2)\circle*{2}\end{picture}}
\newcommand{\frarg}{\,\sbullet\,}

\def\Xint#1{\mathchoice 
{\XXint\displaystyle\textstyle{#1}}% 
{\XXint\textstyle\scriptstyle{#1}}% 
{\XXint\scriptstyle\scriptscriptstyle{#1}}% 
{\XXint\scriptscriptstyle\scriptscriptstyle{#1}}% 
\!\int} 
\def\XXint#1#2#3{{\setbox0=\hbox{$#1{#2#3}{\int}$} 
\vcenter{\hbox{$#2#3$}}\kern-.5\wd0}} 
\def\dashint{\,\Xint-}
\renewcommand{\epsilon}{\varepsilon}
\renewcommand{\phi}{\varphi}

\title[H\"older continuity]{A perturbative approach to H\"older continuity of solutions to a nonlocal $p$-parabolic equation }
\author{Alireza Tavakoli}
\address{Department of Mathematics, KTH Royal Institute of Technology, Stockholm, Sweden}
\email{alirezat@kth.se}
\begin{document}
\begin{abstract}
We study local boundedness and H\"older continuity of a parabolic equation involving the fractional $p$-Laplacian of order $s$, with $0<s<1$, $2\leq p < \infty$, with a general right hand side. We focus on obtaining precise H\"older continuity estimates. The proof is based on a perturbative argument using the already known H\"older continuity estimate for solutions to the equation with zero right hand side.

\vspace{8pt}
\noindent\textsc{Date:} \today.
\end{abstract}
\maketitle
%\hrule\vspace{2pt}
%\begin{center}
%\textbf{\large
%DRAFT (version of \today)}
%\end{center}
%\hrule
%\vspace{10mm}Introduction
%\setcounter{tocdepth}{1} 
%\tableofcontents
%\newpage
%\section*{Check List}
% 
%\subsection{introduction }
%\textbf{done}
%\subsection{chapter 2}
% Lemmas \ref{lm:interpol-critic} and \ref{lm:interpol-morr} are rewritten.
%\subsection{chapter 3}
%\textbf{done} 
%
%\subsection{Appendix A}
%to do : complete the proof of \ref{thm:thm1tail}
%
%
%\subsection{Appendix B}
%to do: write a bit more!
\section{Introduction}
%\subsection{The problem}
In this paper, we study the local boundedness and H\"older regularity of solutions to the inhomogeneous equation
\begin{equation}\label{eq:main-nonlocal-diffution}
  u_t+ (-\Delta_p)^s u=f(x,t),
\end{equation}
where $f \in L^r_{\loc}(I; L^q_{\loc}(\Omega))$ with $q\geq 1$, $r\geq 1$, $p\geq 2$ and $s \in (0,1)$. Here, $ (- \Delta_p)^s $ is the fractional $p$-Laplacian, arising as the first variation of the Sobolev-Slobodecki\u{\i} seminorm
$$(-\Delta_p)^s u (x) :=2\, \mathrm{P.V.} \int_{\R^n}\frac{|u(x)-u(y)|^{p-2}(u(x)-u(y))}{|x-y|^{n+s\,p}} \dd y. $$
Nonlocal equations involving operators of the type above, with a singular kernel, were first considered in \cite{Ishi} to the best of our knowledge.

In this study, continuing the work in \cite{BLS}, we perform a perturbative argument to obtain H\"older continuity estimates, with explicit exponents for the equations with a right hand side. Our approach closely follows the arguments in \cite{Teix} and \cite{BLSc}. In such perturbative arguments it is often possible to establish H\"older regularity results for bounded solutions using only $L^\infty$ estimates for the equations with zero right hand side. This is not the case here. Due to the presence of a suprimim in time in the tail (see section \ref{sec3}), we are led to proving a $L^\infty$ bound for equations with right hand sides, this is Theorem \ref{thm:loc-bound}. The proof is inspired by the work \cite{Arons}.

%
% Our approach is of perturbative nature both for the local boundedness and H\"older regularity. We consider a solution to \eqref{eq:main-nonlocal-diffution}, say $u$ as boundary (and initial) condition for the same equation with zero right hand side. Denoting the new solution, the $(s,p)$-caloric replacement of $u$,  say $v$, we shall transport the existing estimates for the equation with zero source term through $v$ to $u$.
% 
%  Regarding local boundedness in Proposition \ref{thm:sub-moser-bound}, we perform a Moser iteration inspired by the boundedness result of \cite{Arons} for the heat equation and also \cite{Brasc} for the elliptic fractional $p$-Laplace equation, to get an estimate of $\norm{u-v}_{L^\infty}$ where $v$ is the $(s,p)$-caloric replacement of $u$.

Below, we state the main results. For the definition of the tail and relevant function spaces, see Section \ref{sec2}.
We use the following notation of parabolic cylinders
$$Q_{R,r}(x,T):= B_R(x_0) \times (T-r,T]\,.$$
The exponent $p_s^\star= \frac{np}{n-sp}$ is the critical exponent for the Sobolev embedding therem, see Proposition \ref{SObolev-morrey}. We denote by $p^\prime$, the H\"older conjugate of $p$, that is $p^\prime= \frac{p}{p-1}$.
\begin{theorem}\label{thm:loc-bound}
Let $\Omega\subset\R^n$ be a bounded and open set, $I=(t_0,t_1]$, $p\geq 2$, $0<s<1$. Consider $q$ and $r$ such that $r \geq p^\prime$,
$$\frac{1}{r} + \frac{n}{spq} < 1 \quad\text{ and } q\geq (p_s^\star)^\prime \quad \text{in the case } sp < n,$$ 
and
$$\frac{1}{r} + \frac{1}{q} < 1 \quad\text{ and  } q > 1\quad \text{in the case } sp\geq n.$$
Suppose $u$
is a local weak solution of
\[
u_t+(-\Delta_p)^s u=f \qquad \mbox{ in }\Omega\times I,
\]
such that 
\begin{equation*}
u\in L^\infty_{\loc} (I; L_{sp}^{p-1}(\R^n))  \quad \text{and}  \quad f \in L^{r}_{\loc}(I ; L^{q}_{\loc}(\Omega)).
\end{equation*}
then $u$ is locally bounded in $\Omega$. More specifically, if $Q_{2R,(2R)^{sp}(x_0,T_0)} \Subset \Omega \times I$, u is bounded in $Q_{R/2,(R/2)^{sp}}(x_0,T_0)$ and in the case $sp \neq n $, the estimate reads
\[
 \begin{aligned}
 \norm{u}_{L^\infty(Q_{\frac{R}{2},(\frac{R}{2})^{sp}})}&\leq 2 \sup_{T_0-R^{sp}<t\leq T_0} \mathrm{Tail}_{p-1,sp} \bigl(u(\frarg ,t);x_0,\frac{R}{2} \bigr) \\
 &+ 
 C \left[ 1 + \left( \dashint_{Q_{R,R^{sp}}(x_0,T_0)} \abs{u}^{p} \dd x \dd t \right)^{\frac{1}{p}} 
 + \vartheta^{\frac{(p-1)\vartheta}{(\vartheta -1 )^2}} \Bigl( 1 + R^{sp\nu +\frac{(p-2)sp\nu}{(p-1)(\vartheta-1)}}\norm{f}_{L^{q,r}(Q_{R,R^{sp}})}^{1+\frac{1}{\vartheta-1}\frac{p-2}{p-1} } \Bigr) \right],
 \end{aligned}
 \]
where, $C=C(n,s,p)$, $\nu = 1-\frac{1}{r} - \frac{n}{spq}$,
 $$ \vartheta = 1+\frac{sp\nu}{n} \quad \text{if}\quad sp<n, \quad \text{and} \quad \vartheta= 2-\frac{1}{r}-\frac{1}{q} \quad \text{if} \quad sp> n. $$
 In the case $sp=n $, given any $l$ such that $ \frac{p}{r^\prime}(1-\frac{1}{r} - \frac{1}{q})^{-1}  <l < \infty$ we get
 \[
 \begin{aligned}
 &\norm{u}_{L^\infty(Q_{\frac{R}{2},(\frac{R}{2})^{sp}})}\leq 2 \sup_{T_0-R^{sp}<t\leq T_0} \mathrm{Tail}_{p-1,sp} \bigl(u(\frarg ,t);x_0,\frac{R}{2} \bigr) \\
 &\qquad + 
C\left[ 1+ \left( \dashint_{Q_{R,R^{sp}}(x_0,T_0)} \abs{u}^{p} \dd x \dd t \right)^{\frac{1}{p}} 
 + \vartheta^{\frac{(p-1)\vartheta}{(\vartheta -1 )^2}} \Bigl( 1 + R^{sp\nu +\frac{(p-2)sp\nu}{(p-1)(\vartheta-1)}}\norm{f}_{L^{q,r}(Q_{R,R^{sp}})}^{1+\frac{1}{\vartheta-1}\frac{p-2}{p-1} } \Bigr)  \right],
 \end{aligned}
 \]
where, $C= C(n,s,p,q,l) $, $\vartheta= 2-\frac{1}{r}- \frac{1}{q} - \frac{p}{lr^\prime}$ and $\nu= 1-\frac{1}{r} - \frac{1}{q}$.
\end{theorem}

\begin{theorem}\label{thm:main-holder}
Let $\Omega\subset\R^n$ be a bounded and open set, $I=(t_0,t_1]$, $p\geq 2$, $0<s<1$. Consider $q$ and $r$ such that $r \geq p^\prime$,
$$\frac{1}{r} + \frac{n}{spq} < 1 \quad\text{ and } q\geq (p_s^\star)^\prime \quad \text{in the case } sp < n,$$ 
and
$$\frac{1}{r} + \frac{1}{q} < 1 \quad\text{ and  } q > 1\quad \text{in the case } sp\geq n.$$
Define the exponent
\begin{equation}\label{eq:exponent-theta}
 \Theta(s,p):=\left\{\begin{array}{rl}
\dfrac{s\,p}{p-1},& \mbox{ if } s<\dfrac{p-1}{p},\\
&\\
1,& \mbox{ if } s\ge \dfrac{p-1}{p},
\end{array}
\right.
\end{equation}
 Suppose $u$
is a local weak solution of
\[
u_t+(-\Delta_p)^s u=f \qquad \mbox{ in }\Omega\times I,
\]
such that 
\begin{equation*}
u\in L^{\infty}_{\loc}(I;L^\infty_{\loc}(\Omega)) \cap L^\infty_{\loc} (I; L_{sp}^{p-1}(\R^n)),  \quad \text{and}  \quad f \in L^{r}_{\loc}(I ; L^{q}_{\loc}(\Omega)).
\end{equation*}
Then 
\[
u\in C^{\alpha}_{x,\loc}(\Omega\times I)\cap C^{\frac{\alpha}{sp- (p-2)\alpha}}_{t,\loc}(\Omega\times I),\qquad \mbox{ for every } 0<\alpha <\min{ \Bigl \lbrace \Theta , \frac{r(spq-n)-spq}{q(r(p-1)-(p-2))} \Bigr \rbrace}  .
\] 
More precisely, for every $0<\alpha<\min{ \left\lbrace \Theta , \frac{r(spq-n)-spq}{q(r(p-1)-(p-2))} \right\rbrace} $, $R>0$, $x_0\in\Omega$ and $T_0$ such that 
\[
Q_{R,2R^{s\,p}}(x_0,T_0)\Subset\Omega\times (t_0,t_1],
\] 
there exists a constant $C=C(n,s,p, q,r,\alpha)>0$ such that
\begin{equation}
\label{apriori_spacetime}
\begin{aligned}
|u(x_1,t_1)- & u(x_2,t_2)| \leq  C \,\left[ \Mcal \Bigl(\frac{\abs{x_2-x_1}}{R} \Bigr)^\alpha +  \Mcal ^{p-1}\Bigl( \frac{\abs{t_2-t_1}}{R^{s\, p}} \Bigr)^{\frac{\alpha}{sp-(p-2)\alpha}}    \right]
\end{aligned}
\end{equation}
for any $(x_1,t_1),\,(x_2,t_2)\in Q_{R/2,(R/2)^{s\,p}}(x_0,T_0)$, with 
$$\Mcal = 1+ \norm{u}_{L^\infty(Q_{R,2R^{sp}}(x_0,T_0))} + \sup_{ T_0- 2R^{sp}\leq t \leq T_0}\mathrm{Tail}_{p-1,sp} (u(\frarg ,t);x_0,R) + \bigl(  R^{sp-\frac{n}{q} - \frac{sp}{r}} \norm{f}_{L^{q,r}(Q_{R,2R^{sp}}(x_0,T_0))} \bigr)^{\frac{1}{1+ \frac{p-2}{r^\prime}}}.$$
\end{theorem}
\subsection{Known results}

Recently, there has been a growing interest in nonlocal problems of both elliptic and parabolic types. For studies of fractional $p$-Laplace operators with different (continuous) kernels see \cite{Mazo}.
Parabolic equations of the type \eqref{eq:main-nonlocal-diffution} were first considered in \cite{Puhst} with a slightly different diffusion operator. See also \cite{abdel}, \cite{Mazo2}, \cite{Vasques} and \cite{Vasq20} for studies of the existence, uniqueness, and long time behavior of solutions.  Here we seize the opportunity to mention \cite{Caf}, \cite{Chang}, \cite{Chang2},  and \cite{Warm} which contain regularity results for parabolic nonlocal equations.

The local boundedness of the solutions to equations modeled on \eqref{eq:main-nonlocal-diffution} with zero right hand side was obtained in \cite{Strqv}. The results concern operators of the form 
$$L_K =P.V. \int_{R^n} K(x,y,t)|u(x)-u(y)|^{p-2}(u(x)-u(y)) \dd y,$$
where $K$ is symmetric in the space variables and satisfies the ellipticity condition 
$$\frac{\Lambda^{-1}}{\abs{x-y}^{n+sp}} \leq K(x,y,t) \leq \frac{\Lambda}{\abs{x-y}^{n+sp}} . $$
Later in \cite{Ding} local boundedness for certain right hand sides of the form $f(x,t,u)$ was established. 
\cite{Strqv2} contains a Harnack inequality for equations with zero right hand side. H\"{o}lder regularity has also been established in \cite{Adimur} and \cite{Liao} for all $1<p<\infty$ for equations with zero right hand sides. In \cite{BLS} they prove H\"older continuity of the solutions with explicit exponents (for $f=0$ and $K= \abs{x-y}^{-n-sp}$). Recently in \cite{Kumar}, the same type of result has been established for nonlocal equations with double phase, that is for diffusion operators involving two different degrees of homogeneity and differentiability. 

In this study, continuing the work in \cite{BLS}, we perform a perturbative argument to obtain H\"older continuity estimates with explicit exponents for equations with a right hand side.
 
  \subsubsection{Comparison of the results to some previous works} \paragraph{ \textit{Local boundedness and continuity}} We compare our boundedness result to \cite{Ding}. Their result concerns more general right hand sides depending on the solution as well. In the limiting case of $s \to 1$, they reproduce the local boundedness result contained in \cite{Dib} for the evolution $p$-Laplacian equation. To compare the results, if we restrict their result to right hand sides that are $u$-independent, their assumption on the integrability becomes $q,r >\frac{n+sp}{sp}(\frac{p(n+2s)}{2sp +(p-1)n}) $. Their analysis is done with the same integrability assumption in time and space; Our local boundedness result, Theorem \ref{thm:loc-bound}, contains this range of exponents.
  
   In the limiting case when $s$ goes to $1$, our assumptions become $q \geq (p^\star)^\prime , r \geq p^\prime$, and $1-\frac{1}{r} - \frac{n}{pq} >0$. 
   This is in accordance with the classical condition for boundedness of the evolution $p$-Laplace equation, see for example Remark 1 in \cite{Liskev}.
  If we assume the same integrability in time and space, the condition $1-\frac{1}{r} - \frac{n}{pq} >0$ reduces to $f\in L^{\hat{q}}$ with $\hat{q}>\frac{n+p}{p}$. This matches the condition in \cite{Ves}.
  
  Now we turn our attention to the nonlocal elliptic (time independent) case. For $r=\infty$ the condition for boundedness and basic H\"older continuity becomes
  $$q > \frac{n}{sp}\, , \qquad \text{if }\; sp<n \qquad \text{and} \qquad q >1\, , \qquad \text{if }\; sp\geq n \, . $$
  In the case $sp<n$, this is the same condition for local boundedness and continuity contained in \cite{Brasc}.
  When $sp>n$ and $q\geq 1$, the boundedness and H\"older continuity for the time independent equation is automatic using Morrey's inequality. Our result does not cover the case of $r=\infty\,  , q=1\, ,$ which one would expect in comparison to the time independent case.

 \textit{H\"older continuity exponent}: 
   In the case $r= \infty$, the critical H\"older continuity exponent 
   \begin{equation}\label{eq:critical-hol-expon}
   \min{\Bigl\lbrace \Theta , \frac{r(spq-n)-spq}{q(r(p-1)-(p-2))} \Bigr\rbrace}
   =\min{\Bigl \lbrace \Theta , sp\frac{1-\frac{1}{r}-\frac{n}{spq}}{p-1-\frac{p-2}{r}} \Bigr\rbrace},
   \end{equation}
   reduces to $\min{ \lbrace \Theta,\frac{sp}{p-1}(1-\frac{n}{spq})\rbrace}  $ which matches the results in \cite{BLSc}.
   
  Let us also compare our results to the local $p$-parabolic equation studied in \cite{Teix} where precise H\"older continuity exponents are obtained. If we send $s$ to 1, \eqref{eq:critical-hol-expon} becomes
  $$\min{\lbrace 1, \frac{r(pq-n)-pq}{q(r(p-1)-(p-2))}  \rbrace} \, ,$$
  which is in accordance with the result in \cite{Teix}.
    
In \cite{Kumar} explicit H\"older continuity exponents for the more general case of double phase nonlocal diffusion operators were obtained. The ideas explored there are similar to the ones in \cite{BLS}, but their result allows for a bounded right hand side instead of just zero. Their result implies the H\"older continuity exponent that we get in the case of $f\in L^\infty$, although with a slightly different estimate of the H\"older constants.

\subsection{Plan of the paper}
 In Section \ref{sec2} we introduce some notations and preliminary lemmas. We also restate and adapt a result on the existence of solutions to our setting. 
 
In Section \ref{sec3}, we establish basic local H\"older regularity and boundedness for local weak solutions.  

Section \ref{sec4} is devoted to proving Theorem \ref{thm:main-holder}. A so called tangential analysis is performed to get specific H\"older continuity exponents in terms of $q,r,s,$ and $p$.

The article is also accompanied by two appendices. In the first one, Appendix A, we work out the details for a modified version of \cite[Theorem 1.1]{BLS}. The aim is to establish a H\"older estimate in terms of the tail quantity.
 
In Appendix B we justify using certain test functions in the weak formulation of \eqref{eq:main-nonlocal-diffution}.

\subsection{Acknowledgements} The author warmly thanks Erik Lindgren for introducing the problem, proofreading this paper, for his helpful comments, and long hours of fruitful discussions.
The author has partially been supported by the Swedish Research Council, grant no. 2017- 03736.

During the development of this paper, I have been a PhD student at Uppsala University. In particular, I wish to express my gratitude to the Department of Mathematics at Uppsala University for its warm and hospitable research environment

This paper was finalized while I was participating in the program geometric aspects of nonlinear partial differential equations at Mittag-Leffler institute in Djursholm, Sweden during the fall of 2022. The research program is supported by Swedish Research Council grant no. 2016-06596   

\section{Preliminaries}\label{sec2}
\subsection{Notation}
We define the monotone function $J_p: \R \to \R$ by 
$$J_p(t)= \abs{t}^{p-2}t \,.$$

We use the notation $B_R(x_0)$ for the open ball of radius $R$ centered at $x_0$. If the center is the origin, we simply write $B_R$. We use the notation of $\omega_n$ for the surface area of the unit $n$-dimensional ball. For parabolic cylinders, we use the notation $Q_{r,r^\theta}(x_0,t_0) := B_r(x_0)\times (t_0-r^\theta , t_0]$. If the center is the origin, we write $Q_{r,r^\theta}$.

We will work with the fractional Sobolev space extensively:
$$W^{\beta, q}(\R^n) := \lbrace \psi \in L^q(\R^n) \, :\, [\psi]_{W^{\beta , q}(\R^n)} < \infty \rbrace , \qquad 0< \beta <1 ,\quad 1\leq q<\infty,$$
where the seminorm $[\psi]_{W^{s,p}(\R^n)}$ is defined as below
$$[\psi]_{W^{\beta,q}(\R^n)}^q = \iint_{\R^n \times \R^n } \frac{\abs{\psi(x)-\psi(y)}^q}{\abs{x-y}^{n+\beta q}} \dd x \dd y .$$
We also need the space $W^{\beta , q }(\Omega)$ for a subset $\Omega \subset\R^n$, defined by
$$W^{\beta, q}(\Omega) := \lbrace \psi \in L^q(\Omega) \, :\, [\psi]_{W^{\beta , q}(\Omega)} < \infty \rbrace , \qquad 0< \beta <1 , \quad 1\leq q<\infty,$$
where 
$$[\psi]_{W^{\beta,q}(\Omega)}^q = \iint_{\Omega \times \Omega } \frac{\abs{\psi(x)-\psi(y)}^q}{\abs{x-y}^{n+\beta q}} \dd x \dd y .$$
In the following, we assume that $\Omega\subset \R^n$ is a bounded open set in $\R^n$.
We define the space of Sobolev functions taking boundary values $g \in L^{p-1}_{sp}(\R^n) $ by
$$
X_g^{\beta,q}(\Omega,\Omega^\prime) = \lbrace \psi \in W^{\alpha,q}(\Omega^\prime)\cap L_{sp}^{p-1}(\R^n) \, : \, \psi=g \; \text{on} \; \R^n \setminus \Omega \, \rbrace ,
$$
where $\Omega^\prime$ is an open set such that $\Omega \Subset \Omega^\prime$.

We recall the definition of {\it tail space}
$$
L^{q}_{\alpha}(\R^n)=\left\{u\in L^{q}_{\loc}(\R^n)\, :\, \int_{\R^n} \frac{|u|^q}{1+|x|^{n+\alpha}}\dd x < +\infty\right\},\qquad q \geq 1 \text{ and } \alpha > 0,
$$
which is endowed with the norm 
$$
\|u\|_{L_\alpha^{q}(\R^n)} = \left(\int_{\R^n} \frac{|u|^q}{1+|x|^{n+\alpha}}\dd x\right)^{\frac{1}{q}}.
$$
For every $x_0\in\R^n$, $R>0$ and $u\in L^q_{\alpha}(\R^n)$, the following quantity
$$
\mathrm{Tail}_{q,\alpha}(u;x_0,R)=\left[R^{\alpha}\,\int_{\R^n\setminus B_R(x_0)} \frac{|u|^q}{|x-x_0|^{n+\alpha}}\dd x\right]^\frac{1}{q},
$$
plays an important role in regularity estimates for solutions to fractional problems.

Let $I \subset \R$ be an interval and let $V$ be a separable, reflexive, Banach space endowed with a norm $\norm{\frarg}_V$. We denote by $V^\star$ its topological dual space. Let us suppose that
$v$ is a mapping such that for almost every $t \in I$, $v(t)$ belongs to $V$. If the function $t \to \norm{v(t)}_V$ is measurable on $I$ and $1 \leq p \leq \infty$, then $v$ is an element of the Banach space $L^p(I; V)$ if and only if
$$\int_I \norm{v(t)}_V \dd t < \infty$$
By \cite[Theorem 1.5]{Show}, the dual space of $L^p(I; V)$ can be characterized according to
$(L^p(I; V ))^\star = L^{p^\prime}(I; V^\star).$
We write $v \in C(I; V )$ if the mapping $t \to v(t)$ is continuous with respect to the norm on $V$.
\subsection{Pointwise inequalities}
We will need the following pointwise inequality:
Let $p \geq 2$, then for every $A ,B \in \R$ we have
\begin{equation}\label{pntwiseineq}
 \abs{A-B}^p \leq C \bigl( J_p(A) - J_p(B)\bigr)(A-B).
\end{equation}
For a proof look at \cite[Remark A.4]{BLS}, a close inspection of the proof reveals that the constant can be taken as $C= 3 \cdot 2^{p-1}$. Before stating the next inequality, we recall \cite[Lemma A.2]{Brasc}.
\begin{lemma}\label{lm:Lemma A.2 in Brasc}
Let $1< p < \infty$ and $g:\R \to \R$ be an increasing function, and define
$$G(t)= \int_0^t g^\prime(\tau)^{\frac{1}{p}} \dd \tau , \quad t \in \R .$$
Then
$$J_p(a-b) \bigl( g(a)-g(b) \bigr) \geq \absb{G(a)-G(b)}^p .$$
\end{lemma}
\begin{lemma}\label{lm:pointwise-ineq}
For $p\geq 2$ and $\beta \geq 1$
\begin{equation}
\begin{aligned}
\bigl( J_p(a-b)- J_p(c-d) \bigr)&\Bigl( ((a-c)^{+}_M +\delta)^\beta -((b-d)^{+}_M+\delta )^\beta \Bigr) \\
&\geq \frac{1}{3 \cdot 2^{p-1}} \frac{\beta p^p}{(\beta+p-1)^p}\absB{((a-c)^{+}_M +\delta)^{\frac{\beta+p-1}{p}} - ((b-d)^{+}_M +\delta)^{\frac{\beta + p-1}{p}}}^p ,
\end{aligned}
\end{equation}
where $(t)^{+}_M:=\min{\lbrace \max{\lbrace t,0\rbrace},M\rbrace}$.
\end{lemma}
\begin{proof}
First notice that using \eqref{pntwiseineq} for $a-b -c+d \neq 0$:
$$3\cdot 2^{p-1} (J_p(a-b)- J_p(c-d)) \geq \frac{\abs{a-b -c+d}^p}{a-b -c+d} = J_p((a-c)- (b-d)), $$
after verifying the trivial case $a-b -c+d = 0$, we get the inequality
\begin{equation}\label{eq:pointwise-rearange}
(J_p(a-b)- J_p(c-d)) \geq  \frac{1}{3 \cdot 2^{p-1}} J_p((a-c)- (b-d)).
\end{equation}
Now we use Lemma \ref{lm:Lemma A.2 in Brasc} with $g(t)= ((t)^{+}_M +\delta)^\beta$. Then with $G= \int_0^t g^\prime(\tau)^{\frac{1}{p}} \dd \tau$, 
$$G(t)=\frac{p\beta^{\frac{1}{p}}}{\beta+p-1}\Bigl((t^{+}_M +\delta)^{\frac{\beta+p-1}{p}} - \delta^{\frac{\beta+p-1}{p}} \Bigr). $$
By Lemma \ref{lm:Lemma A.2 in Brasc}
$$J_p \bigl((a-c)-(b-d) \bigr) \bigl(g(a-c)-g(b-d) \bigr) \geq \abs{G(a-c)-G(b-d)}^p .$$
Hence
\[
\begin{aligned}
J_p \bigl((a-c)-(b-d) \bigr)& \Bigl( ((a-c)^{+}_M +\delta)^\beta -((b-d)^{+}_M +\delta )^\beta \Bigr) \\
&\geq \frac{\beta p^p}{(\beta+p-1)}\absB{((a-c)^{+}_M +\delta)^{\frac{\beta+p-1}{p}} - ((b-d)^{+}_M +\delta)^{\frac{\beta + p-1}{p}}}^p.
\end{aligned}
\]
Using \eqref{eq:pointwise-rearange} in the inequality above concludes the proof.
\end{proof}
\subsection{Functional inequalities}

We need the following basic inequalities for the tail.
\begin{lemma}\label{lm:tail_comparison}
Let $\alpha>0$, $1<q< \infty$, and $u,\,v \in L_\alpha^q(\R^n)$ such that $u=v$ on $\R^n \setminus B_R(x_0)$. Then for any $\sigma <1$,
$$\mathrm{Tail}_{\alpha,q}(v;x_0,\sigma R) \leq 2 \mathrm{Tail}_{\alpha,q}(u;x_0,\sigma R) + 2\sigma^{\frac{-n}{q}} \Bigl( \dashint_{B_R(x_0)} \abs{u-v}^q \dd x \Bigr)^{\frac{1}{q}} .$$
\end{lemma}
\begin{proof}
\[
\begin{aligned}
\mathrm{Tail}_{\alpha,q}(v;x_0,\sigma R)^q & = (\sigma R)^\alpha \int_{\R^n \setminus B_{\sigma R}(x_0)} \frac{\abs{v}^q}{\abs{x-x_0}^{n+\alpha}} \dd x \\
&=  (\sigma R)^\alpha \Bigl( \int_{\R^n \setminus B_{R}(x_0)} \frac{\abs{v}^q}{\abs{x-x_0}^{n+\alpha}} \dd x + \int_{B_{R}(x_0) \setminus B_{\sigma R}(x_0)} \frac{\abs{v}^q}{\abs{x-x_0}^{n+\alpha}} \dd x  \Bigr) \\
&= (\sigma R)^\alpha \Bigl( \int_{\R^n \setminus B_{R}(x_0)} \frac{\abs{u}^q}{\abs{x-x_0}^{n+\alpha}} \dd x + \int_{B_{R}(x_0) \setminus B_{\sigma R}(x_0)} \frac{\abs{v}^q}{\abs{x-x_0}^{n+\alpha}} \dd x  \Bigr) \\
& \leq (\sigma R)^\alpha \Bigl( \int_{\R^n \setminus B_{R}(x_0)} \frac{\abs{u}^q}{\abs{x-x_0}^{n+\alpha}} \dd x + 2^{q-1} \int_{B_{R}(x_0) \setminus B_{\sigma R}(x_0)} \frac{\abs{u}^q + \abs{u-v}^q}{\abs{x-x_0}^{n+\alpha}} \dd x  \Bigr) \\
& \leq 2^{q-1}(\sigma R)^\alpha \Bigl( \int_{\R^n \setminus B_{\sigma R}(x_0)} \frac{\abs{u}^q}{\abs{x-x_0}^{n+\alpha}} \dd x + \int_{B_{R}(x_0) \setminus B_{\sigma R}(x_0)} \frac{\abs{u-v}^q}{\abs{x-x_0}^{n+\alpha}} \dd x  \Bigr) \\
&\leq 2^{q-1}  \mathrm{Tail}_{\alpha,q}(u;x_0,\sigma R)^q + 2^{q-1} \sigma^{-n} \dashint_{B_R(x_0)} \abs{u-v}^q \dd x  
\end{aligned}
\]
\end{proof}

For a proof of the following result, see \cite[Lemma 2.3]{BLSc}.
\begin{lemma}\label{lm:tail-shrink}
Let $\alpha >0$, $0 <q<\infty$. Suppose that $B_r(x_0) \subset B_R(x_1)$. Then for every $u\in L^q_\alpha(\R^n)$ we have
$$\mathrm{Tail}_{q,\alpha}(u;x_0,r)^q \leq \Bigl(\frac{r}{R}\Bigr)^\alpha \Bigl(\frac{R}{R- \abs{x-x_0}} \Bigr)^{n + \alpha} \mathrm{Tail}_{q, \alpha} (u; x_1 , R)^q + r^{-n}\norm{u}_{L^q}^q$$
If in addition $u \in L^m_{\loc}(\R^n)$ for some $q <m \leq \infty$, then 
\[
\mathrm{Tail}_{q,\alpha}(u;x_0,r)^q \leq \Bigl(\frac{r}{R}\Bigr)^\alpha \Bigl(\frac{R}{R- \abs{x-x_0}} \Bigr)^{n + \alpha} \mathrm{Tail}_{q, \alpha} (u; x_1 , R)^q + \Bigl(  \frac{(n \omega_n)m-q}{\alpha m+nq}\Bigr)^{\frac{m-q}{m}} r^{-\frac{qn}{m}} \norm{u}_{L^m(B_R(x_1))},
\]
where $\omega_n$ is the measure of the $n$-dimensional open ball of radius $1$.
\end{lemma}

We also recall the following Sobolev and Morrey type inequalities:
\begin{proposition}\label{SObolev-morrey}
Suppose $1<p<\infty$ and $0<s<1$. Let $\Omega \subset \R^n$ be an open and bounded set.
Define $p_s^\star$ as
\begin{equation}\label{eq:sob-exp}
  p_s^\star:= \frac{np}{n-sp}.
\end{equation}
 For every $u \in W^{s,p}(\R^n)$ vanishing almost everywhere in $\R^n \setminus \Omega$ we have 
\begin{equation}\label{eq:Sobolev}
 \norm{u}_{L^{p_s^\star}(\Omega)}^p \leq C_1(n,s,p)\, [u]_{W^{s,p}(\R^n)}^p, \qquad  if \; sp<n
\end{equation}

\begin{equation}\label{eq:Morrey}
\norm{u}_{L^\infty(\Omega)}^p \leq C_2(n,s,p) \abs{\Omega}^{\frac{sp}{n}-1} [u]_{W^{s,p}(\R^n)}^p , \qquad if \; sp>n
\end{equation}
\begin{equation}\label{eq:critic-sob}
 \norm{u}_{L^l(\Omega)}^p \leq C_3(n,s,p,l) \abs{\Omega}^{\frac{p}{l}} [u]_{W^{s,p}(\R^n)}^p , \qquad \text{ for every }1\leq l < \infty, \; if \; sp=n
\end{equation}
In particular the following Poincar\'{e} inequality holds true
\begin{equation}\label{eq:poincare}
 \norm{u}_{L^p(\Omega)}^p \leq C\; \abs{\Omega}^{\frac{sp}{n}}[u]_{W^{s,p}(\R^n)}, 
\end{equation}
for some $C=C(n,s,p)$.
\end{proposition}
\begin{remark}
The Sobolev type inequalities above are also valid for functions $u \in X_0^{s,p}(\Omega, \Omega^\prime)$, where $\Omega$ is a bounded open set and $\Omega^\prime$ is an open set such that $\Omega \Subset \Omega^\prime$. This can be seen using the fact that there is an extension domain containing $\Omega$ and included in $\Omega^\prime$. 
\end{remark}
We will often use the following special application of H\"older's inequality 
\begin{equation}\label{eq:Holder}
\norm{u(x,t)}_{L^{q_1,r_1}(\Omega\times J)} \leq \norm{\abs{\Omega}^{\frac{1}{q_1}-\frac{1}{q_2}} \norm{u(\frarg,t)}_{L^{q_2}(\Omega)}}_{L^{r_1}(J)} \leq 
\abs{\Omega}^{\frac{1}{q_1}-\frac{1}{q_2}} \abs{J}^{\frac{1}{r_1}-\frac{1}{r_2}} \norm{u}_{L^{q_2,r_2}(\Omega \times J)}
\end{equation}
Where $q_1<q_2 \, , r_1\leq r_2$.
The following interpolation inequality (see e.g. \cite{Arons}) will be useful.
\begin{lemma}\label{lm:holint}
If $w$ is contained in $L^{q_1,r_1}(\Omega \times J) \cap L^{q_2 , r_2}(\Omega \times J)$, then $w$ is contained in $L^{\tilde{q}, \tilde{r}}(\Omega \times J)$, where
$$ \frac{1}{\tilde{r}} = \frac{\lambda}{r_1} + \frac{1-\lambda}{r_2}, \quad \frac{1}{\tilde{q}}= \frac{\lambda}{q_1} + \frac{1- \lambda}{q_2} \; ,   \quad (0 \leq\lambda \leq 1) .$$
Moreover,
$$\norm{w}_{L^{\tilde{q} , \tilde{r}}(\Omega \times J)} \leq \norm{w}_{L^{q_1 ,r_1}(\Omega \times J)}^{\lambda} \norm{w}_{L^{q_2 ,r_2}(\Omega \times J)}^{1-\lambda}. $$
\end{lemma} 
The following three lemmas will be needed in the proof of our local boundedness result, Proposition \ref{thm:sub-moser-bound}.
\begin{lemma}\label{lm:interpol}
Let $sp< n$ and assume that $w$ is in $L^{p_s^\star,p}(Q_{R,R^{sp}}) \cap L^{p,\infty}(Q_{R,R^{sp}})$, with $p_s^\star= \frac{np}{n - sp}$ being the Sobolev exponent. Then $w$ is in $L^{p q^\prime , p r^\prime} (Q_{R,R^{sp}})$ as long as $q ,r$ satisfy
$$ 1-\frac{1}{r} - \frac{n}{spq} \geq 0.$$
Moreover,
$$\norm{w}_{L^{pq^\prime , pr^\prime}(Q_{R,R^{sp}})}^p \leq R^{sp(1-\frac{1}{r} - \frac{n}{spq})} \Bigl( \norm{w}_{L^{p,\infty}(Q_{R,R^{sp}})}^p +  \norm{w}_{L^{p_s^\star , p}(Q_{R,R^{sp}})}^p \Bigr).$$
In particular, in the case of $\frac{1}{r} + \frac{n}{spq}=1$ we have
$$\norm{w}_{L^{pq^\prime , pr^\prime}(Q_{R,R^{sp}})}^p \leq \norm{w}_{L^{p,\infty}(Q_{R,R^{sp}})}^p +  \norm{w}_{L^{p_s^\star , p}(Q_{R,R^{sp}})}^p \,.$$
\end{lemma} 
\begin{proof}
Consider a pair of exponents $\tilde{r} = (\frac{1}{r^\prime} - (1-\frac{1}{r} - \frac{n}{spq}))^{-1}= \frac{spq}{n}$, and $\tilde{q} = q^\prime$ such that $\frac{1}{\tilde{r}^{\, \prime}} + \frac{n}{sp\tilde{q}^{\, \prime}} =1 $. Using H\"older's inequality \eqref{eq:Holder}, we obtain
\[
\begin{aligned}
\norm{w}_{L^{pq^\prime , pr^\prime}(Q_{R,R^{sp}})}^p &\leq  (R^{sp})^{\frac{1}{r^\prime} - \frac{1}{\tilde{r}}} \norm{w}_{L^{p\tilde{q} , p\tilde{r}}(Q_{R,R^{sp}})}^p = R^{sp(1-\frac{1}{r} - \frac{n}{spq})} \norm{w}_{L^{p\tilde{q} , p\tilde{r}}(Q_{R,R^{sp}})}^p .
\end{aligned}
\]
Now we use Lemma \ref{lm:holint} with the choice
$$\frac{1}{p\tilde{r}} = \frac{\lambda}{p} \qquad \text{and} \qquad \frac{1}{p\tilde{ q}} = \frac{\lambda}{p_s^\star} + \frac{1-\lambda}{p} \, ,\quad  (0\leq \lambda \leq 1) .$$
This yields
$$ \norm{w}_{L^{p\tilde{q} , p\tilde{r}}(Q_{R,R^{sp}})} \leq \norm{w}_{L^{p_s^\star , p}(Q_{R,R^{sp}})}^\lambda \norm{w}_{L^{p,\infty}(Q_{R,R^{sp}})}^{1-\lambda}.$$
The relations above hold for $\lambda = \frac{1}{\tilde{r}} = \frac{n}{sp\tilde{q}^{\, \prime}}$ and using Young's inequality we get
$$\norm{w}_{L^{p\tilde{q} , p\tilde{r}}(Q_{R,R^{sp}})}^p \leq \norm{w}_{L^{p_s^\star , p}(Q_{R,R^{sp}})}^{p \lambda} \norm{w}_{L^{p,\infty}(Q_{R,R^{sp}})}^{p(1-\lambda)} \leq \norm{w}_{L^{p,\infty}(Q_{R,R^{sp}})}^{p} +\norm{w}_{L^{p_s^\star , p}(Q_{R,R^{sp}})}^{p}. $$
This concludes the desired result. 
\end{proof}
\begin{lemma}\label{lm:interpol-morr}
Let $sp> n$ and assume that $w \in L^{\infty,p}(Q_{R,R^{sp}})\cap L^{p, \infty}(Q_{R,R^{sp}})$. Consider a pair of exponents $q\geq 1 , r\geq 1$, such that
$$1- \frac{1}{r} - \frac{1}{q} = 0 .$$
Then $w$ belongs to $L^{p q^\prime , p r^\prime}(Q_{R,R^{sp}})$ and
$$\norm{w}_{L^{pq^\prime , pr^\prime}(Q_{R,R^{sp}})}^p \leq R^{\frac{sp-n}{q}} \Bigl( \norm{w}^p_{L^{p , \infty}(Q_{R,R^{sp}})} + R^{n-sp}\norm{w}^p_{L^{\infty , p}(Q_{R,R^{sp}})} \Bigr).$$
%$$\norm{w}_{L^{pq^\prime , pr^\prime}(Q_{R,R^{sp}})}^p \leq R^{sp(1-\frac{1}{r}-\frac{n}{spq})} \Bigl( \norm{w}^p_{L^{p , \infty}(Q_{R,R^{sp}})} + R^{n-sp}\norm{w}^p_{L^{\infty , p}(Q_{R,R^{sp}})} \Bigr).$$
\end{lemma}
\begin{proof}
We use Lemma \ref{lm:holint}, with the choice
$$\frac{1}{p r^\prime} = \frac{\lambda}{p} \quad \text{and} \quad \frac{1}{p q^\prime}= \frac{1-\lambda}{p}, \quad (0 \leq \lambda \leq 1) ,$$
which holds for $\lambda = \frac{1}{r^\prime} = 1- \frac{1}{q^\prime}$. This yields
$$\norm{w}_{L^{p q^\prime, p r^\prime}(Q_{R,R^{sp}})} \leq  \norm{w}_{L^{\infty,p}(Q_{R,R^{sp}})}^\lambda  \norm{w}_{L^{p, \infty}(Q_{R,R^{sp}})}^{1-\lambda} .$$
Therefore, using $\lambda = \frac{1}{r^\prime}= \frac{1}{q}$ we arrive at
$$R^{\frac{n-sp}{q}}\norm{w}_{L^{p q^\prime , p r^\prime}(Q_{R,R^{sp}})}^p = R^{\lambda(n-sp)}\norm{w}_{L^{p q^\prime , p r^\prime}(Q_{R,R^{sp}})}^p \leq  \Bigl(R^{n-sp} \norm{w}_{L^{\infty,p}(Q_{R,R^{sp}})}^p \Bigr)^{ \lambda} \Bigl(\norm{w}_{L^{p, \infty}(Q_{R,R^{sp}})}^p \Bigl)^{(1-\lambda)}.$$
Using Young's inequality, we conclude
$$\norm{w}_{L^{p q^\prime , p r^\prime}(Q_{R,R^{sp}})}^p \leq R^{\frac{sp-n}{q}} \Bigr(\norm{w}^p_{L^{p, \infty}(Q_{R,R^{sp}})} + R^{n-sp}\norm{w}^p_{L^{\infty, p}(Q_{R,R^{sp}})}  \Bigl).$$
\end{proof}
\begin{lemma}\label{lm:interpol-critic}
Let $sp=n,\, q\geq 1 , \text{ and } r\geq 1 $ such that
$$ 1 - \frac{1}{r} - \frac{1}{q} > 0 .$$
Assume that $w \in L^{l,p}(Q_{R,R^{sp}})\cap L^{p,\infty}(Q_{R,R^{sp}})$ for some $l$ such that $$l= \frac{p}{r^\prime}(1-\frac{1}{r} - \frac{1}{q})^{-1} .$$ 
Then $w$ belongs to $L^{pq^\prime , p r^\prime}(Q_{R,R^{sp}})$ and
$$\norm{w}_{L^{pq^\prime , p r^\prime}(Q_{R,R^{sp}})} \leq R^{\frac{np}{lr^\prime}} \Bigl( \norm{w}_{L^{p,\infty}(Q_{R,R^{sp}})} + R^{\frac{-np}{l}}\norm{w}_{L^{l,p}(Q_{R,R^{sp}})}^p  \Bigr).$$

\end{lemma}
\begin{proof}
We use Lemma \ref{lm:holint} with the choice
$$\frac{1}{p r^\prime} = \frac{\lambda}{p} \quad \text{and} \quad \frac{1}{p q^\prime}= \frac{\lambda}{l} + \frac{1-\lambda}{p}, \quad (0\leq \lambda \leq 1).$$
Due to the assumption $\frac{1}{l}= \frac{r^\prime}{p}(1-\frac{1}{r}-\frac{1}{q})$, the above equalities hold for $\lambda= \frac{1}{r^\prime}$. Hence we get
$$\norm{w}_{L^{p q^\prime, p r^\prime}(Q_{R,R^{sp}})} \leq  \norm{w}_{L^{l,p}(Q_{R,R^{sp}})}^\lambda  \norm{w}_{L^{p, \infty}(Q_{R,R^{sp}})}^{1-\lambda} .$$
Therefore, recalling that $\lambda= \frac{1}{r^\prime}$
$$R^{\frac{- n p}{l r^\prime}} \norm{w}_{L^{p q^\prime, p r^\prime}(Q_{R,R^{sp}})}^p  =R^{\frac{-\lambda n p}{l}} \norm{w}_{L^{p q^\prime, p r^\prime}(Q_{R,R^{sp}})}^p \leq  \Bigl( R^{\frac{-np}{l}} \norm{w}_{L^{l,p}(Q_{R,R^{sp}})}^p \bigr)^\lambda  \Bigl(\norm{w}_{L^{p, \infty}(Q_{R,R^{sp}})}^p \Bigr)^{1-\lambda} .$$
Using Young's inequality for the right hand side, we can conclude
$$\norm{w}_{L^{p q^\prime, p r^\prime}(Q_{R,R^{sp}})}^p \leq R^{\frac{np}{l r^\prime}} \Bigl( \norm{w}_{L^{p, \infty}(Q_{R,R^{sp}})}^p + R^{\frac{-np}{l}} \norm{w}_{L^{l,p}(Q_{R,R^{sp}})}^p \Bigr).$$
\end{proof}
 
\subsection{Weak solutions}
\begin{definition}
\label{locweak} For any $t_0,t_1\in\R$ with $t_0<t_1$, we define $I=(t_0,t_1]$.
Let 
\[
f\in L^{p'}(I;(W^{s,p}(\Omega))^*).
\]
We say that $u$ is a \emph{local weak solution} to the equation
$$
\partial_t u + (-\Delta_p)^su = f,\qquad\mbox{ in }\Omega\times I,
$$
if for any closed interval $J=[T_0,T_1]\subset I$, the function $u$ is such that
\[
u\in L^p(J;W_{{\loc}}^{s,p}(\Omega))\cap L^{p-1}(J;L_{s\,p}^{p-1}(\R^n))\cap C(J;L_{{\loc}}^2(\Omega)),
\]
and it satisfies 
\begin{equation}
\begin{split}
\label{locweakeq}
-\int_J\int_\Omega u(x,t)\,\partial_t\phi(x,t)\dd x\dd t&+ \int_J\iint_{\R^n\times\R^n}\frac{J_p(u(x,t)-u(y,t))\,(\phi(x,t)-\phi(y,t))}{|x-y|^{n+s\,p}}\,\dd x  \dd y \dd t  \\   
& = \int_\Omega u(x,T_0)\,\phi(x,T_0)\dd x -\int_\Omega u(x,T_1)\,\phi(x,T_1)\dd x \\
&+ \int_J\langle f(\frarg,t),\phi(\frarg,t)\rangle\dd t,
\end{split}
\end{equation}
for any $\phi\in L^p(J;W^{s,p}(\Omega))\cap C^1(J;L^2(\Omega))$ which has spatial support compactly contained in $\Omega$. 
In equation \eqref{locweakeq}, the symbol $\langle\frarg,\frarg \rangle$  stands for the duality pairing between $W^{s,p}(\Omega)$ and its dual space $(W^{s,p}(\Omega))^*$. 
\end{definition}
Now, we define the notion of a weak solution to an initial boundary value problem.

\begin{definition}
Let $I=[t_0,t_1]$, $p\geq 2$, $0<s<1$, and $\Omega \Subset \Omega^\prime$, where $\Omega^\prime$ is a bounded open set in $\R^n$. Assume that the functions $u_0,f$ and $g$ satisfy
\[
u_0\in L^2(\Omega), 
\]
\[
f\in L^{p'}(I;(X_0^{s,p}(\Omega,\Omega'))^*),
\]
\[
g\in L^p(I;W^{s,p}(\Omega'))\cap L^{p-1}(I;L_{s\,p}^{p-1}(\R^n)). 
\]
We say that $u$ is a \emph{weak solution of the initial boundary value problem} 
\begin{equation}
\label{eq:initial-weaksol}
 \left\{\begin{array}{rcll}
\partial_tu + (-\Delta_p)^su&=&f,&\mbox{ in }\Omega\times I,\\
u&=&g,& \mbox{ on }(\R^n\setminus\Omega)\times I,\\
u(\frarg,t_0) &=& u_0,&\text{ on }\Omega, 
\end{array}\right.
\end{equation}
if the following properties are verified:
\begin{itemize}
\item $u\in L^p(I;W^{s,p}(\Omega'))\cap L^{p-1}(I;L_{sp}^{p-1}(\R^n))\cap C(I;L^2(\Omega))$;
\vskip.2cm
\item $u\in X_{\mathbf{g}(t)}(\Omega,\Omega')$ for almost every $t\in I$, where $(\mathbf{g}(t))(x)=g(x,t)$;
\vskip.2cm
\item $\lim_{t\to t_0}\|u(\frarg,t) - u_0\|_{L^2(\Omega)}=0$;
\vskip.2cm
\item for every $J=[T_0,T_1]\subset I$ and every $\phi\in L^{p}(J;X_0^{s,p}(\Omega,\Omega'))\cap C^1(J;L^2(\Omega))$
\[
\begin{split}
-\int_J\int_\Omega u(x,t)\,\partial_t\phi(x,t)\dd x\dd t&+ \int_J\iint_{\R^n\times\R^n}\frac{J_p(u(x,t)-u(y,t))\,(\phi(x,t)-\phi(y,t))}{|x-y|^{n+sp}}\dd x\dd y\dd t  \\   
& = \int_\Omega u(x,T_0)\,\phi(x,T_0)\dd x -\int_\Omega u(x,T_1)\,\phi(x,T_1)\dd x \\
&+ \int_J\langle f(\frarg,t),\phi(\frarg,t)\rangle\dd t.
\end{split}
\]
\end{itemize}

\end{definition}
\begin{theorem}\label{thm:exunsol}
Let $p\ge 2$, let $I = (T_0,T_1]$ and suppose that $g$ satisfies
\begin{align*}
&g\in L^p(I;W^{s,p}(\Omega'))\cap L^{p}(I;L^{p-1}_{s\,p}(\R^n)) \cap C(I;L^2(\Omega)),\quad\partial_t g\in L^{p'}(I;(X_0^{s,p}(\Omega,\Omega'))^*),\\
&\lim_{t\to t_0}\|g(\frarg,t)-g_0\|_{L^2(\Omega)} = 0,\qquad \text{ for some }g_0\in L^2(\Omega). 
\end{align*}
Suppose also that 
\[
f\in L^{p'}(I;(X_0^{s,p}(\Omega,\Omega'))^*).
\]
Then for any initial datum $g_0\in L^2(\Omega)$, there exists a unique weak solution $u$ to problem 
\begin{equation}\label{Mweaksol}
\begin{cases}
u_t + (-\Delta_p)^s u= f \quad & in \;\; \Omega\times I \\
u=g \quad & in \;\; (\R^n \setminus \Omega) \times I \\
u(x,T_0) = g(x,T_0) \quad & in \;\; \Omega
\end{cases}
\end{equation}
\end{theorem}
\begin{proof}
In \cite[Theorem A.3]{BLS} the same result is proven with a stronger condition $g_t \in L^{p^\prime}(I;W^{s,p}(\Omega^\prime)^\star)$. The stronger condition, is not needed in the proof. This condition can be replaced with $g_t \in L^{p^\prime}(I,X^{s,p}_0 (\Omega; \Omega^\prime)^\star)$ in all of the steps in the proof, except that the construction gives us a $C(I;L^2(\Omega))$ solution. There, the stronger assumption is used only to show that the boundary condition is in $C(I;L^2(\Omega))$, which we assume here.
\end{proof}

\section{Basic H\"older regularity and stability}\label{sec3}
Throughout the rest of the article, we assume $0<s<1$ and $2\leq p < \infty$.

Here, we argue that the norm of the $(s,p)$-caloric replacement of $u$ is close to $u$ if $f$ is small enough. By the $(s,p)$-caloric replacement of $u$ in a cylinder $B_\rho(x_0) \times I$ we mean the solution to the following
\begin{equation}\label{eq:pscalreplace}
\begin{cases}
v_t + (-\Delta_p)^s v=0 \quad & in \;\; B_\rho(x_0) \times I \\
v=u \quad & in \;\; (\R^n \setminus B_\rho(x_0) ) \times I \\
v(x,\tau_0) = u(x,\tau_0) \quad & in \;\; B_\rho(x_0)
\end{cases}
\end{equation}
Here $\tau_0$ is the initial point of the interval $I$.
First we show the existence of a $(s,p)$-caloric replacement using Theorem \ref{thm:exunsol}
\begin{proposition}\label{prop:ps_caloric}
Let $u$ be a local weak solution of $ u_t + (-\Delta_p)^s u=f $ in the cylinder $B_\sigma \times J$, for some interval $J=(t_1,t_2]$ with $f \in L^{q,r}_{\loc}(B_\sigma \times J)$, for $q> (p_s^\star)^\prime$ and $r> p^\prime$. In addition, we assume that $u \in L^p(J;L_{sp}^{p-1}(R^n))$. Then for any $0<\rho < \sigma$, and closed interval $I \Subset J$, the $(s,p)$-caloric replacement of $u$ in $B_\rho(x_0)\times I$ (weak solution to \eqref{eq:pscalreplace}) exists.
\end{proposition}
\begin{proof}
We shall check the conditions in Theorem \ref{thm:exunsol}. If they are satisfied there exists a unique weak solution $v \in L^p(I, W^{s,p}(B_{\sigma})) \cap L^{p-1}(I;L_{sp}^{p-1}(\R^n)) \cap C(I;L^2(B_\rho))$ to the problem \eqref{eq:pscalreplace}.
 The only condition on $u$ that is not immediate from the fact that $u$ is weak solution is $\partial_t u \in L^{p^\prime}(I;X^{s,p}_0 (B_\rho \, , B_{\sigma})^\star) $. We have to show that for every function $\psi \in L^p(I;X^{s,p}_0 (B_\rho \, , B_{\sigma}))$
$$\absB{\int_I \dpr{u_t,\psi} \dd x \dd t} \leq C \int_{I} \norm{\psi}_{W^{s,p}(B_{\sigma})}^p \dd t .$$
We shall verify this for test functions belonging to the dense subspace, $\psi \in L^p(I;X^{s,p}_0 (B_\rho \, , B_{\sigma}))\cap C^1_{0}(I; L^2(B))$. We use the equation to do so. We have
\[
\begin{aligned}
\int_I \dpr{u_t,\psi} \dd x \dd t &= \int_I \int_{B_\rho}u \psi_t \dd x \dd t
\\ &=  -\int_I \iint_{\R^n \times \R^n} \frac{J_p(u(x,t)-u(y,t))(\psi(x,t) -\psi(y,t))}{\abs{x-y}^{n+sp}} \dd x \dd y \dd t + \int_I \int_{B_{r}} f(x,t) \psi(x,t) \dd x \dd t \\
&= - \int_I \iint_{B_{\sigma} \times B_{\sigma}} \frac{J_p(u(x,t)-u(y,t))(\psi(x,t) -\psi(y,t))}{\abs{x-y}^{n+sp}} \dd x \dd y \dd t\\
&- 2 \int_I \int_{\R^n\setminus B_{\sigma}} \int_{B_\rho} \frac{J_p(u(x,t)-u(y,t))\psi(x,t)}{\abs{x-y}^{n+sp}} \dd x \dd y \dd t + \int_I \int_{B_\rho} f(x,t)\psi(x,t) \dd x \dd t.
\end{aligned}
\]
By H\"older's inequality, we have
\begin{equation}\label{eq:exi1}
\begin{aligned}
\int_I& \iint_{B_{\sigma} \times B_{\sigma}} \frac{\abs{J_p(u(x,t)-u(y,t))(\psi(x,t) -\psi(y,t))}}{\abs{x-y}^{n+sp}} \dd x \dd y \dd t \\ 
&\leq \int_I \normB{\frac{J_p(u(x,t)-u(y,t))}{\abs{x-y}^{\frac{n}{p^\prime} +s(p-1)}}}_{L^{p^\prime}(B_{\sigma}\times B_{\sigma})} \normB{\frac{\psi(x,t)-\psi(y,t)}{\abs{x-y}^{\frac{n}{p}+s}}}_{L^p(B_{\sigma}\times B_{\sigma})} \dd t \\
&\leq [u]_{L^p(I;W^{s,p}(B_{\sigma}))}^{p-1} [\psi]_{L^p(I;W^{s,p}(B_{\sigma}))}.
\end{aligned}
\end{equation}
For the other nonlocal term, we note that for every $x \in B_\rho$ and $y \in \R^n\setminus B_{\sigma}$ we have $\abs{y} \leq \frac{\sigma}{\sigma-\rho}  \abs{x-y}$. Hence, 
\[
\begin{aligned}
\int_{\R^n\setminus B_{\sigma}} \frac{\abs{J_p(u(x,t)-u(y,t))}}{\abs{x-y}^{n+sp}} \dd y &\leq 
(\frac{\sigma}{\sigma - \rho})^{n+sp} C(p) \int_{\R^n\setminus B_{\sigma}} \frac{\abs{u(x,t)}^{p-1}+ \abs{u(y,t)}^{p-1}}{\abs{y}^{n+sp}} \dd y \\
 &\leq C(\sigma, \rho ,  s ,p , n ) \Bigl( \abs{u(x,t)}^{p-1} + \norm{u(\frarg ,t)}^{p-1}_{L_{sp}^{p-1}} \Bigr).
 \end{aligned}
\]
Therefore,
\[
\begin{aligned}
\int_I \int_{B_\rho}  \int_{\R^n\setminus B_{\sigma}} \frac{\abs{J_p(u(x,t)-u(y,t))\psi(x,t)}}{\abs{x-y}^{n+sp}} \dd y \dd x \dd t &\leq C(\sigma, \rho ,  s ,p , n ) \Bigl( \int_I \int_{B_\rho} \abs{\psi(x,t)} \abs{u(x,t)}^{p-1} \dd x \dd t \\
&+ \int_I \norm{\psi(\frarg ,t)}_{L^1(B_\rho)} \norm{u(\frarg ,t)}^{p-1}_{L_{sp}^{p-1}(\R^n)} \dd t \Bigr) .
\end{aligned}
\]
Using H\"older's inequality, we have
\begin{equation}\label{eq:exi2}
\int_I \int_{B_\rho} \abs{\phi(x,t)} \abs{u(x,t)}^{p-1} \dd x \dd t \leq \int_I \norm{\psi(\frarg , t)}_{L^p(B_\rho)} \norm{u(\frarg,t)}_{L^{p}(B_\rho)}^{p-1} \leq \norm{\psi}_{L^p(I;L^p(B_\rho))} \norm{u}_{L^p(I;L^p(B_\rho))}^{p-1}.
\end{equation}
For the other term,
\begin{equation}\label{eq:exi3}
\int_I \norm{\psi(\frarg ,t)}_{L^1(B_\rho)} \norm{u(\frarg ,t)}^{p-1}_{L_{sp}^{p-1}(\R^n)} \dd t \leq \norm{\psi}_{L^p(I;L^1(B_\rho))} \norm{u(\frarg ,t)}_{L^p(I;L_{sp}^{p-1}(\R^n))}^{p-1} .
\end{equation}
Since $f \in L^{p^\prime}(I ; L^{(p_s^\star)^\prime}(B_\rho))$, we get by H\"older's inequality
\begin{equation}\label{eq:exi0}
\begin{aligned}
\int_{I} \int_{B_{\rho}} \abs{f \psi} \dd x \dd t &\leq \int_{I} \norm{f}_{L^{(p_s^\star)^\prime}(B_{\rho})} \norm{\psi}_{L^{p_s^\star}(B_{\rho})} \dd t \leq \int_{I} \norm{f}_{L^{(p_s^\star)^\prime}(B_{\rho})} \norm{\psi}_{W^{s ,p}(B_{\sigma})} \dd t \\
& \leq \norm{f}_{L^{((p_s^\star)^\prime,p^\prime )}(B_{\rho} \times I)} \norm{\psi}_{L^p(I;W^{s,p}(B_\sigma))}.
\end{aligned}
\end{equation}
Therefore, combining with \eqref{eq:exi1} , \eqref{eq:exi2}, and \eqref{eq:exi3} we obtain
\[
\absB{\int_I \dpr{v_t,\psi} \dd t} \leq C(\sigma, \rho , s,p,n,u,f) \norm{\psi}_{L^p(I;W^{s,p}(B_\sigma))}.
\]
\end{proof}
\begin{lemma}\label{lm:lpavr}
Assume that $f\in L^{q,r}_{\loc}(Q_{\sigma,\sigma^{sp}}(x_0,T_0))$ with $r\geq p^\prime $ and
$$ q\geq (p_s^\star)^\prime \quad\text{if } sp<n ,\qquad q \geq 1  \quad\text{if } sp>n, \;\text{and}\qquad q>1 \quad \text{if } sp=n.$$
Let $u$ be a local weak solution of $\partial_t u + (-\Delta_p)^s u= f$ in $Q_{\sigma,\sigma^{sp}}(x_0,T_0)$ and for $\rho< \sigma$ consider $v$ to be the $(s,p)$-caloric replacement of $u$ in $Q_{\rho,\rho^{sp}}(x_0,T_0)$. Then we have
\begin{equation}\label{eq:campanato-norm1}
 \dashint_{Q_{\rho , \rho^{s \, p} }(x_0,T_0)} \abs{u-v}^p \dd x \dd t \leq C \rho^{\xi} \, \norm{f}_{L^{q,r}(Q_{\rho,\rho^{s \, p}}(x_0,T_0))}^{p^\prime}
\end{equation}
and 
\begin{equation}\label{eq:pq-dual}
 \norm{u-v}_{L^{q^\prime,r^\prime}(Q_{\rho,\rho^{sp}})} \leq C \rho^{\xi + n} \norm{f}_{L^{q,r}(Q_{\rho,\rho^{sp}}(x_0,T_0))}^{\frac{1}{p-1}},
\end{equation}
with $\xi = spp^\prime(1- \frac{1}{r} - \frac{n}{spq})$ and $C= C(n,s,p)$, in the case $sp\neq n$. In the case $sp=n$, we can take $\xi= spp^\prime(1 - \frac{1}{r} - \frac{1}{q})$ , with $C=C(n,s,p,q)$ also depending on $q$.
\end{lemma}
\begin{proof}
Let $J:=[T_0-\rho^{sp},T_0]$, throughout the proof, we drop the dependence of the balls on the center and write $B_\rho$ instead of $B_\rho(x_0)$, and $Q_{\rho,\rho^{sp}}$ instead of $Q_{\rho,\rho^{sp}}(x_0,T_0)$.

By subtracting the weak formulation of the equations \eqref{locweakeq} for $u$ and $v$ with the same test function $\phi(x,t) \in L^{p}(J;X_0^{s,p}(B_{\rho},B_{\sigma}))\cap C^1(J;L^2(B_\rho))$ we get
\[
\begin{aligned}
- \int_J \int_{B_\rho}& (u(x,t)- v(x,t))\frac{\partial}{\partial t} \phi (x,t) \dd x \dd t \\
&+ \int_J \int_{\R^n} \int_{\R^n} \frac{\bigl[J_p \bigl(u (x,t) - u(y,t)\bigr) - J_p\bigl( v (x,t) - v(y,t)\bigr)\bigr](\phi (x,t) - \phi (y,t))}{\abs{x-y}^{n+sp}} \dd x \dd y \dd t \\
&= \int_{B_\rho} ((u(x,T_0-\rho^{sp}) - v(x,T_0-\rho^{sp}))\phi(x,T_0-\rho^{sp}) \dd x - \int_{B_\rho} ((u(x,T_0) - v(x,T_0))\phi(x,T_0) \dd x \\
& + \int_J \int_{B_\rho} f(x,t)\phi(x,t).
\end{aligned}
\]
Now we take $\phi := u-v$, which belongs to $L^p(J; X_0^{s,p}(B_{\rho};B_\sigma))$ but it may not be in $C^1(J;L^2(B_\rho))$. We justify taking this as a test function in Appendix B. By Proposition \ref{prop:testing} with $F(t)=t$, we get
\begin{equation}\label{eq:testlinear}
\begin{aligned}
\int_J &\iint_{\R^n\times \R^n} \frac{\bigl[J_p \bigl(u (x,t) - v(x,t)\bigr) - J_p\bigl( u (y,t) - v(y,t)\bigr)\bigr]\bigl[\bigl(u (x,t) - u(y,t)\bigr) -\bigl( v (x,t) - v(y,t)\bigr) \bigr]}{\abs{x-y}^{n+sp}} \dd x \dd y \dd t \\
& = \int_J \int_{B_\rho} f(x,t)(u(x,t) -v(x,t)) \dd x\dd t \\
&\qquad - \frac{1}{2}\int_{B_\rho} ((u(x,T_0) - v(x,T_0))^2 -((u(x,T_0- \rho^{sp}) - v(x,T_0 - \rho^{sp}))^2 \dd x \\
&= \int_J \int_{B_\rho} f(x,t)(u(x,t) -v(x,t)) \dd x\dd t - \frac{1}{2}\int_{B_\rho} ((u(x,T_0) - v(x,T_0))^2 \dd x \\
& \leq \int_J \int_{B_\rho} \abs{f(x,t)(u(x,t) -v(x,t))} \dd x\dd t ,
\end{aligned}
\end{equation}
where in the third line we have used $u(x,T_0-\rho^{sp}) = v(x,T_0-\rho^{sp})$.
The left hand side is essentially the $W^{s,\,p}$ seminorm. By the pointwise inequality \eqref{pntwiseineq}
 \[
 \begin{split}
 &\int_J [u-v]_{W^{s,p}(\R^n)}^p \dd t = \int_J \iint_{\R^n\times \R^n} \frac{\abs{u(x,t)-v(x,t) - (u(y,t)-v(y,t))}^p}{\abs{x-y}^{n+sp}} \dd x \dd y \dd t \\
 \leq C(p) &\int_J \iint_{\R^n\times \R^n} \frac{\bigl[J_p \bigl(u (x,t) - u(y,t)\bigr) - J_p\bigl( v (x,t) - v(y,t)\bigr)\bigr]\bigl[u (x,t) - u(y,t) -\bigl( v (x,t) - v(y,t)\bigr) \bigr]}{\abs{x-y}^{n+sp}} \dd x \dd y \dd t.
 \end{split}
 \]
Therefore, by \eqref{eq:testlinear} and H\"older's inequality
\begin{equation}\label{eq:sobcalrepl}
\begin{aligned}
\int_J [u-v]_{W^{s,p}(\R^n)}^p \dd t &\leq C(p) \int_J \int_{B_\rho} \abs{f(x,t)(u(x,t) -v(x,t))} \dd x\dd t \\
&\leq  C(p) \int_J \norm{f(\;\sbullet\;,t)}_{L^q(B_\rho)}\norm{(u-v)(\;\sbullet\;,t)}_{L^{q\prime}(B_\rho)} \dd t \\
& \leq C(p) \norm{f}_{L^{q,r}(Q_{\rho,\rho^{sp}})} \norm{u-v}_{L^{q\prime,r \prime}(Q_{\rho,\rho^{sp}})}.
\end{aligned}
\end{equation}
Now we consider three cases: $sp<n$, $sp>n$,and $sp=n$.

\textbf{Case} $sp < n$.
By H\"older's inequality \eqref{eq:Holder} and Sobolev's inequality \eqref{eq:Sobolev} we have
\begin{equation}\label{eq:pq-dual-sob}
\begin{aligned}
\norm{u-v}_{L^{q\prime,r \prime}}
&\leq  \abs{B_\rho}^{ \frac{1}{q\prime} - \frac{1}{p_s^\star}} \Bigl( \int_J \norm{u-v}_{L^{p_s^\star}(B_\rho)}^{r\prime} \dd t \Bigr)^{\frac{1}{r \prime}}\\
& \leq C(n,s,p) \abs{B_\rho}^{ \frac{1}{q\prime} - \frac{1}{p_s^\star}} \Bigl( \int_J [u-v]_{W^{s,p}(\R^n)}^{r\prime} \dd t \Bigr)^{\frac{1}{r\prime}} \\
& \leq C(n,s,p) \abs{B_\rho}^{ \frac{1}{q\prime} - \frac{1}{p_s^\star}} \abs{J}^{\frac{1}{r\prime} - \frac{1}{p}} \Bigl( \int_J [u-v]_{W^{s,p}(\R^n)}^p \dd t\Bigr)^{\frac{1}{p}}.
\end{aligned}
\end{equation}
Combined with \eqref{eq:sobcalrepl} this yields
\begin{equation}\label{eq:sob-est}
\begin{aligned}
\Bigl( \int_J [u-v]_{W^{s,p}(\R^n)}^p \dd t\Bigr)^{\frac{p-1}{p}} &\leq C \abs{B_\rho}^{ \frac{1}{q\prime} - \frac{1}{p_s^\star}} \abs{J}^{\frac{1}{r\prime} - \frac{1}{p}} \norm{f}_{L^{q,r}(Q_{\rho,\rho^{sp}})}\\
&= C \abs{B_\rho}^{ \frac{1}{q\prime} - \frac{n-sp}{np}} \abs{J}^{\frac{1}{r\prime} - \frac{1}{p}} \norm{f}_{L^{q,r}(Q_{\rho,\rho^{sp}})},
\end{aligned}
\end{equation}
where $C= C(n,s,p)$. By the Poincar\'{e} inequality
$$
\dashint_J \dashint_{B_\rho} \abs{u-v}^p \dd x \dd t \leq C\abs{B_\rho}^{\frac{p\prime}{q\prime} - p^\prime\frac{(n-sp)}{np} +\frac{sp}{n}-1} \abs{J}^{\frac{p^\prime}{r\prime}- \frac{p\prime}{p}-1} \norm{f}_{L^{q,r}(Q_{\rho,\rho^{sp}})}^{p^\prime}.
$$
Also from \eqref{eq:sob-est} and  \eqref{eq:pq-dual-sob}  we get
$$\norm{u-v}_{L^{q\prime , r\prime}(Q_{\rho,\rho^{sp}})} \leq C(n,s,p)\abs{B_\rho}^{ \frac{p\prime}{q\prime} - p\prime\frac{n-sp}{np}} \abs{J}^{\frac{p\prime}{r\prime} - \frac{p\prime}{p}} \norm{f}_{L^{q,r}(Q_{\rho,\rho^{sp}})}^{\frac{1}{p-1}}. $$ 
\textbf{Case} $sp>n$.
In this case, we use Morrey's inequality \eqref{eq:Morrey} and H\"older's inequality and obtain
\begin{equation}\label{eq:pq-dual-mor}
\begin{aligned}
\norm{u-v}_{L^{q\prime,r \prime}(Q_{\rho,\rho^{sp}})}
&\leq C \abs{B_\rho}^{ \frac{1}{q\prime}} \Bigl( \int_J \norm{u-v}_{L^{\infty}(B_\rho)}^{r\prime} \dd t \Bigr)^{\frac{1}{r \prime}} 
\leq C \abs{B_\rho}^{\frac{1}{q^\prime}}\abs{J}^{\frac{1}{r\prime} - \frac{1}{p}} \Bigl( \int_J \norm{u-v}_{L^{\infty}(B_\rho)}^{p} \dd t \Bigr)^{\frac{1}{p}}\\
& \leq C \abs{B_\rho}^{ \frac{1}{q\prime} + \frac{sp-n}{np}} \abs{J}^{\frac{1}{r\prime} - \frac{1}{p}} \Bigl( \int_J [u-v]_{W^{s,p}(\R^n)}^p \dd t\Bigr)^{\frac{1}{p}}.
 \end{aligned}
\end{equation}
Together with \eqref{eq:sobcalrepl}, this implies
\begin{equation}\label{eq:sob,spgen}
\Bigl( \int_J [u-v]_{W^{s,p}(\R^n)}^p \dd t\Bigr)^{\frac{p-1}{p}} \leq C \abs{B_\rho}^{ \frac{1}{q\prime} - \frac{n-sp}{np}} \abs{J}^{\frac{1}{r\prime} - \frac{1}{p}} \norm{f}_{L^{q,r}(Q_{\rho,\rho^{sp}})}.
\end{equation}
By the Poincar\'{e} inequality
$$
\dashint_J \dashint_{B_\rho} \abs{u-v}^p \dd x \dd t  \leq C\abs{B_\rho}^{\frac{p\prime}{q\prime} - p^\prime\frac{(n-sp)}{np} +\frac{sp}{n}-1} \abs{J}^{\frac{p^\prime}{r\prime}- \frac{p\prime}{p}-1} \norm{f}_{L^{q,r}(Q_{\rho,\rho^{sp}})}^{p^\prime}.
$$
Combining \eqref{eq:pq-dual-mor} and \eqref{eq:sob,spgen}, we get
$$\norm{u-v}_{L^{q\prime , r\prime}(Q_{\rho,\rho^{sp}})} \leq C(n,s,p)\abs{B_\rho}^{ \frac{p\prime}{q\prime} - p\prime\frac{n-sp}{np}} \abs{J}^{\frac{p\prime}{r\prime} - \frac{p\prime}{p}} \norm{f}_{L^{q,r}(Q_{\rho,\rho^{sp}})}^{\frac{1}{p-1}}. $$ 
\textbf{Case} $sp=n$.
In this case, we use the critical case of Sobolev's inequality \eqref{eq:critic-sob} for $ l= q^\prime$ and obtain
\[
 \norm{u-v}_{L^{q^\prime}(B_\rho)}^p \leq C(n,s,p,q) \abs{B_\rho}^{\frac{p}{q^\prime}} [u-v]_{W^{s,p}(\R^n)}^p.
\]
Hence using H\"older's inequality, we have for any $r\geq p^\prime$
\[
\begin{aligned}
\norm{u-v}_{L^{q\prime,r \prime}(Q_{\rho,\rho^{sp}})}
&= \Bigl( \int_J \norm{u-v}_{L^{q^\prime}(B_\rho)}^{r\prime} \dd t \Bigr)^{\frac{1}{r \prime}}\\
& \leq C \abs{B_\rho}^{ \frac{1}{q\prime} } \Bigl( \int_J [u-v]_{W^{s,p}(\R^n)}^{r\prime} \dd t \Bigr)^{\frac{1}{r\prime}} \leq C \abs{B_\rho}^{ \frac{1}{q\prime}} \abs{J}^{\frac{1}{r\prime} - \frac{1}{p}} \Bigl( \int_J [u-v]_{W^{s,p}(\R^n)}^p \dd t\Bigr)^{\frac{1}{p}}.
\end{aligned}
\]
The constant $C=C(n,s,p,q)$ above does blow up as $q$ goes to $1$. In a similar way as in the prior cases, we get for $q>1$ and $r\geq p^\prime$ 
$$ \dashint_J \dashint_{B_\rho} \abs{u-v}^p \leq C(n,s,p,q)\abs{B_\rho}^{\frac{p\prime}{q\prime} } \abs{J}^{\frac{p^\prime}{r\prime}- \frac{p\prime}{p}-1} \norm{f}_{L^{q,r}(Q_{\rho,\rho^{sp}})}^{p^\prime} $$
and
$$\norm{u-v}_{L^{q\prime , r\prime}(Q_{\rho,\rho^{sp}})} \leq C(n,s,p,q)\abs{B_\rho}^{ \frac{p\prime}{q\prime} } \abs{J}^{\frac{p\prime}{r\prime} - \frac{p\prime}{p}} \norm{f}_{L^{q,r}(Q_{\rho,\rho^{sp}})}^{\frac{1}{p-1}}. $$ 

Using that $\abs{B_\rho}\sim \rho^{n}$ and $\abs{I} \sim \rho^{sp}$ we can conclude that
$$
 \dashint_{Q_{\rho , \rho^{s \, p} }(x_0,T_0)} \abs{u-v}^p \dd x \dd t \leq C \rho^{\xi} \, \norm{f}_{L^{q,r}(Q_{\rho,\rho^{s \, p}})}^{p^\prime},
$$
and
$$\norm{u-v}_{L^{q^\prime,r^\prime}(Q_{\rho,\rho^{sp}})} \leq C \rho^{\xi + n} \norm{f}_{L^{q,r}(Q_{\rho,\rho^{sp}})}^{\frac{1}{p-1}}.$$
Here in the case of $sp\neq n$,
\[
\begin{aligned}
\xi &= \frac{n p^\prime}{q^\prime} -p^\prime\frac{n-sp}{p} +sp -n + \frac{spp^\prime}{r^\prime} - \frac{spp^\prime}{p} -sp  = p^\prime (\frac{n}{q^\prime} - \frac{n}{p^\prime} - \frac{n - sp}{p} + \frac{sp}{r^\prime} - \frac{sp}{p}) \\
& = p^\prime (\frac{n}{q^\prime} - n + \frac{sp}{r^\prime} ) 
= p^\prime (\frac{sp}{r^\prime} - \frac{n}{q}) =  spp^\prime(1- \frac{1}{r} - \frac{n}{spq})
\end{aligned}
\]
and in the case $sp=n$, 
\[
\begin{aligned}
\xi= p^\prime(\frac{n}{q^\prime}  + \frac{sp}{r^\prime} - \frac{sp}{p} -\frac{sp}{p^\prime}) &= sp p^\prime(\frac{1}{q^\prime} +\frac{1}{r^\prime} - \frac{1}{p}-\frac{1}{p^\prime}) \\
& = sp p^\prime(1-\frac{1}{q}  + 1-\frac{1}{r}-1) \\
&=spp^\prime(1 - \frac{1}{r} - \frac{1}{q}).
\end{aligned}
\]
\end{proof}

Next, we perform a Moser iteration to get an $L^\infty$ bound for the difference between the solution and its $(s,p)$-caloric replacement.
 \begin{proposition}\label{thm:sub-moser-bound}
Let $u$ be a local weak solution of 
$$\partial_t u + (-\Delta_p)^s u=f, \quad \text{in} \quad Q_{\sigma,\sigma^{sp}}(x_0,T_0),$$
with $f \in L^{q,r}_{\loc} (Q_{\sigma,\sigma^{sp}}(x_0,T_0))$ such that  $r \geq p^\prime$,
$$\frac{1}{r} + \frac{n}{spq} < 1 \quad\text{ and } q\geq (p_s^\star)^\prime \quad \text{in the case } sp < n,$$ 
and
$$\frac{1}{r} + \frac{1}{q} < 1 \quad\text{ and  } q > 1\quad \text{in the case } sp\geq n.$$
Let $v$ be the $(s,p)$-caloric replacement of $u$ in $Q_{R,R^{s p}}(x_0,T_0)$, with $R< \sigma$. Then in the case of $sp \neq n$, we have 
\[
\norm{(u-v)^+}_{L^\infty(Q_{RR^{s p}}(x_0,T_0))} \leq  C(n,s,p) \vartheta^{\frac{(p-1)\vartheta}{(\vartheta -1 )^2}} \Bigl( 1 + R^{sp\nu +\frac{(p-2)sp\nu}{(p-1)(\vartheta-1)}}\norm{f}_{L^{q,r}(Q_{R,R^{sp}}(x_0,T_0))}^{1+\frac{1}{\vartheta-1}\frac{p-2}{p-1} } \Bigr),
\]
where $\nu = 1-\frac{1}{r} - \frac{n}{spq}$,
 $$ \vartheta = 1+\frac{sp\nu}{n} \quad \text{if}\quad sp<n, \quad \text{and} \quad \vartheta= 2-\frac{1}{r}-\frac{1}{q} \quad \text{if} \quad sp> n. $$
In the case of $sp=n$, given any $l$ such that $ \frac{p}{r^\prime}(1-\frac{1}{r} - \frac{1}{q})^{-1} <l < \infty$ we get
\[
\norm{(u-v)^+}_{L^\infty(Q_{R,R^{s p}}(x_0,T_0))}\leq C(n,s,p,q,l) \vartheta^{\frac{(p-1)\vartheta}{(\vartheta - 1)^2}} \Bigl( 1 + R^{sp\nu +\frac{(p-2)sp\nu}{(p-1)(\vartheta-1)}}\norm{f}_{L^{q,r}(Q_{R,R^{sp}}(x_0,T_0))}^{1+\frac{1}{\vartheta-1}\frac{p-2}{p-1} } \Bigr) , 
\]
where $\vartheta= 2-\frac{1}{r}- \frac{1}{q} - \frac{p}{lr^\prime}$ and $\nu= 1-\frac{1}{r} - \frac{1}{q}$.
\end{proposition}
\begin{proof}
Throughout the proof we write $Q_{R,R^{sp}}$ instead of $Q_{R,R^{s,p}}(x_0,T_0)$. We test the equations with powers of $u-v$ and perform a Moser iteration. Using Proposition \ref{prop:testing} with 
$$ F(t)=( \min{ \lbrace t^{+} \, ,M \rbrace} +\delta)^\beta - \delta^\beta, $$
and
\begin{equation}\label{eq:delta-choice}
\delta = \max{ \lbrace 1 , \bigl( R^{sp\nu}\norm{f}_{L^{q,r}(Q_{R,R^{sp}})} \bigr)^{\frac{1}{p-1}} \rbrace},
\end{equation}
we get
\begin{equation}\label{eq:test-mos}
\begin{aligned}
\sup_{t \in J} \int_{B_R}& \Fcal(u-v) \dd x  + \int_J \iint_{\R^n \times \R^n} \frac{J_{p}(u(x,t)-u(y,t)) - J_p(v(x,t)-v(y,t))}{\abs{x-y}^{n+sp}} \\
& \qquad \qquad \qquad \qquad \qquad \times (F(u(x,t)-v(x,t)) - F(u(y,t)-v(y,t))) \dd x \dd y \dd t \\
& \leq \int_J \int_{B_R} \abs{f(x,t)} F(u(x,t)-v(x,t))  \\
& \leq \norm{f}_{L^{q,r}(B_R \times J)} \norm{((u-v)^{+}_M + \delta)^{\beta}}_{L^{q^\prime , r^\prime}(B_R \times J)} .
\end{aligned}
\end{equation}
In the last line, we have used H\"older's inequality. Here $\Fcal(t) = \int_0^t F(t) \dd t$ is
\[
\Fcal(t) = 
\begin{cases} 
\qquad  \qquad 0 \quad & \text{if} \quad t\leq 0, \\
\quad \frac{1}{\beta + 1}(t+\delta)^{\beta +1} -\frac{\delta^{\beta +1}}{\beta + 1} - t \delta^\beta  \quad &\text{if} \quad 0\leq t \leq M, \\
\frac{1}{\beta + 1}(M +\delta)^{\beta +1} -\frac{\delta^{\beta +1}}{\beta + 1} - t \delta^\beta + (t-M)(M+\delta)^\beta & \text{if} \quad t \geq M.
\end{cases}
\]
Notice that by Young's inequality, for $t \geq 0$
$$ \frac{(t+\delta)^{\beta + 1}}{2(\beta + 1)} + \frac{\beta}{\beta + 1} 2\delta^{\beta + 1} \geq \frac{t+\delta}{2^{\frac{1}{\beta+ 1} }} 2^{\frac{\beta}{\beta + 1}} \delta^\beta \geq t \delta^\beta \, . $$
In particular for $0\leq t \leq M$
$$\Fcal(t) \geq \frac{(t+\delta)^{\beta + 1}}{2(\beta + 1)} - \frac{2\beta + 1}{\beta +1} \delta^{\beta + 1} \geq \frac{(t+\delta)^{\beta + 1}}{2(\beta + 1)} - 2\delta^{\beta + 1} ,$$
and for $t\geq M$
\[ \begin{aligned}
\frac{1}{\beta + 1}(M +\delta)^{\beta +1} -\frac{\delta^{\beta +1}}{\beta + 1} &- t \delta^\beta + (t-M)(M+\delta)^\beta = \frac{1}{\beta + 1}(M +\delta)^{\beta +1} -\frac{\delta^{\beta +1}}{\beta + 1} - M \delta^\beta \\
&+ (t-M)\bigl ((M+\delta)^\beta -\delta^\beta \bigr) 
 \geq \Fcal(M) \geq \frac{(M+\delta)^{\beta +1}}{2(\beta + 1)} - 2\delta^{\beta + 1} .
\end{aligned}
\]
Hence
\begin{equation}\label{eq:fcal-lower}
\Fcal (t) \geq \frac{(t^{+}_M +\delta)^{\beta +1}}{2(\beta + 1)} - 2\delta^{\beta + 1} .
\end{equation}
Using Lemma \ref{lm:pointwise-ineq} for the second term in the left hand side of \eqref{eq:test-mos}, and \eqref{eq:fcal-lower} in the first term we obtain
\begin{equation}\label{eq:testing-Morr-sob}
\begin{aligned}
\frac{1}{2(\beta + 1)} &\sup_{t \in J} \int_{B_R} ((u-v)^{+}_M + \delta)^{\beta +1} \dd x + \frac{1}{3\cdot 2^{p-1}}\frac{\beta p^p}{(\beta +p -1)^p}\int_J [((u-v)^{+}_M + \delta)^{\frac{\beta +p-1}{p}}]_{W^{s, p}(\R^n)}^p \dd t  \\
& \leq \sup_{t \in J} \int_{B_R} \Fcal(u-v) \dd x + 2\delta^{\beta +1} \abs{B_R}  + \int_J \iint_{\R^n \times \R^n} \frac{J_{p}(u(x,t)-u(y,t)) - J_p(v(x,t)-v(y,t))}{\abs{x-y}^{n+sp}} \\
& \qquad \; \qquad \qquad \; \qquad \times (F(u(x,t)-v(x,t)) - F(u(y,t)-v(y,t))) \dd x \dd y \dd t \\
& \qquad \qquad \leq \norm{f}_{L^{q,r}(B_R \times J)} \norm{((u-v)^{+}_M+ \delta)^{\beta}}_{L^{q^\prime , r^\prime}(B_R \times J)}  + 2\delta^{\beta +1} \abs{B_R} .
\end{aligned}
\end{equation}
Let $w(x,t) = ((u-v)^{+}_M +\delta)^{\frac{\beta}{p}}$. Since $\delta \leq (u-v)^{+}_M + \delta$, we see that 
\begin{equation}\label{eq:delta-beta}
 \delta^\beta \leq \frac{\norm{w}_{L^{pq^\prime , p r^\prime}(Q_{R,R^{s p}})}^p}{\abs{B_R}^{1-\frac{1}{q}} \abs{J}^{1- \frac{1}{r}}} .
\end{equation}
We consider three cases depending on whether $sp>n$, $sp=n$, or $sp>n$.

\textbf{Case $sp < n$}: Using Sobolev's inequality in the second term in \eqref{eq:testing-Morr-sob} and applying \eqref{eq:delta-beta} we get
\begin{equation}\label{eq:moser-diffu-sob}
\begin{aligned}
\frac{\delta}{2(\beta + 1)} &\norm{w}_{L^{p,\infty}(B_R \times J)}^p \, + \frac{C(n,s,p)^{-1}}{3 \cdot 2^{p-1}}\frac{\beta p^p}{(\beta + p-1)^p} \bigl[ \delta^{p-1} \norm{w}_{L^{p_s^\star, p}(B_R \times J)}^p - \delta^{\beta + p-1} \abs{B_R}^{\frac{n-sp}{n}} \abs{J} \bigr] \\
& \leq \frac{\delta}{\beta + 1} \sup_{t \in J} \int_{B_R} ((u-v)^{+}_M + \delta)^{\beta} \dd x \\
& \quad + \frac{C(n,s,p)^{-1}}{3\cdot 2^{p-1}}\frac{\beta p^p}{(\beta + p-1)^p}\int_J \norm{((u-v)^{+}_M +\delta)^{\frac{\beta + p-1}{p}} - \delta^{\frac{\beta + p-1}{p}}}_{L^{p_s^\star}(B_R)}^p \\
& \leq \norm{f}_{L^{q,r}(B_R \times J)} \norm{w}_{L^{p q^\prime , p r^\prime}(B_R \times J)}^p \, + 2\delta \abs{B_R} \frac{ \norm{w}_{L^{p q^\prime , p r^\prime}(B_R \times J)}^p}{\abs{B_R}^{1-\frac{1}{q}}\abs{J}^{1-\frac{1}{r}}}.
\end{aligned}
\end{equation}
Upon multiplying both sides by $3 \cdot 2^{p-1} \times C(n,s,p)\frac{(\beta + p -1)^p}{\delta \beta}$ and taking $J$ to have length $R^{s p}$, this implies
\begin{equation}\label{eq:energy-sob-moser}
\begin{aligned}
3 &\cdot 2^{p-2}C(n,s,p)\frac{(\beta + p-1)^p}{(\beta+1)\beta }\norm{w}_{L^{p,\infty}(Q_{R,R^{s p}})}^p + \delta^{p-2} p^p  \norm{w}_{L^{p_s^\star , p}(Q_{R,R^{s p}})}^p \\
& \leq 3 \cdot 2^{p-1}\times C(n,s,p)\frac{(\beta + p -1)^p}{\delta \beta} \norm{w}_{L^{pq^\prime ,pr^\prime}}^p \bigl( \norm{f}_{L^{q,r}(Q_{R,R^{sp}})} +\frac{2(n \omega_n)^{\frac{1}{q}}\delta R^n}{R^{n(1-\frac{1}{q}) + sp(1-\frac{1}{r})}} \bigr) \\
&+ p^p \delta^{\beta +p-2} R^{n-sp +sp}.
\end{aligned}
\end{equation}
By replacing the constant $C(n,s,p)$ above in Sobolev's inequality with $\max{\lbrace 1, C(n,s,p)\rbrace}$, we can assume $C(n,s,p) \geq 1$. Using this and that $\delta \geq 1$, and $\frac{(\beta+p-1)^p}{\beta(\beta+1)} \geq 1$ we obtain
\[
\begin{aligned}
\norm{w}&_{L^{p,\infty}(Q_{R,R^{s p}})}^p + \norm{w}_{L^{p_s^\star , p}(Q_{R,R^{s p}})}^p  \\
&\leq 3 \cdot 2^{p-2}C(n,s,p)\frac{(\beta + p-1)^p}{(\beta+1)\beta }\norm{w}_{L^{p,\infty}(Q_{R,R^{s p}})}^p + \delta^{p-2} p^p  \norm{w}_{L^{p_s^\star , p}(Q_{R,R^{s p}})}^p.
\end{aligned}
\]
Using this together with \eqref{eq:energy-sob-moser}, and \eqref{eq:delta-beta} we get
\begin{equation}\label{eq:mos-pre-interpol}
\begin{aligned}
\norm{w}&_{L^{p,\infty}(Q_{R,R^{s p}})}^p + \norm{w}_{L^{p_s^\star , p}(Q_{R,R^{s p}})}^p \\
& \leq C\frac{(\beta + p -1)^p}{\beta} \norm{w}_{L^{p q^\prime , p r^\prime}(Q_{R,R^{s p}})}^p \Bigl(  \frac{\norm{f}_{L^{q,r}(Q_{R,R^{s p}})}}{\delta} + \delta^{p-2} \frac{R^n}{R^{n-\frac{n}{q} +sp -\frac{sp}{r}}}  + \frac{R^n}{R^{n-\frac{n}{q} +sp -\frac{sp}{r}}} \Bigr) \\
&\leq C\frac{(\beta + p -1)^p}{\beta} \norm{w}_{L^{p q^\prime , p r^\prime}(Q_{R,R^{s p}})}^p \Bigl( \frac{\norm{f}_{L^{q,r}(Q_{R,R^{sp}})}}{\delta} + \delta^{p-2} R^{-sp\nu} \Bigr),
\end{aligned}
\end{equation}
where $C=C(n,s,p)$. Recalling our choice of $\delta$, \eqref{eq:delta-choice},
in the case of $\delta >1$
\begin{equation}\label{eq:fnorm-delta1}
\frac{\norm{f}_{L^{q,r}(Q_{R,R^{sp}})}}{\delta} + \delta^{p-2} R^{-sp\nu} = 2R^{\frac{-sp\nu}{p-1}}\norm{f}_{L^{q,r}(Q_{R,R^{sp}})}^{\frac{p-2}{p-1}} = 2 \delta^{p-2}R^{-sp\nu},
\end{equation}
and in the case of $\delta = 1$
\begin{equation}\label{eq:fnorm-delta2}
\frac{\norm{f}_{L^{q,r}(Q_{R,R^{sp}})}}{\delta} + \delta^{p-2} R^{-sp\nu} \leq 2R^{-sp\nu} = 2 \delta^{p-2}R^{-sp\nu}.
\end{equation}
Using the inequality $\frac{(\beta + p-1)^p}{\beta^p p^p} \leq 1$, \eqref{eq:fnorm-delta1}, and \eqref{eq:fnorm-delta2}, in \eqref{eq:mos-pre-interpol} we arrive at
\[
\begin{aligned}
\norm{w}_{L^{p,\infty}(Q_{R,R^{s p}})}^p + \norm{w}_{L^{p_s^\star , p}(Q_{R,R^{s p}})}^p &\leq C(n,s,p) p^p \beta^{p-1} \norm{w}_{L^{p q^\prime , p r^\prime}(Q_{R,R^{s p}})}^p  \bigl( \delta^{p-2} R^{-sp\nu}).
\end{aligned}
\]
Now notice that since $\nu > 0 $, if we take $\vartheta = 1+ \frac{sp\nu}{n}$, the exponents  $(\vartheta r^\prime)^\prime , (\vartheta q^\prime)^\prime $ satisfy the condition of Lemma \ref{lm:interpol}. Indeed,
$$1- \frac{1}{(\vartheta r^\prime)^\prime} - \frac{n}{sp (\vartheta q^\prime)^\prime} = \frac{1}{\vartheta r^\prime} + \frac{n}{sp \vartheta q^\prime} - \frac{n}{sp} = \frac{1}{\vartheta}(\frac{1}{r^\prime} +\frac{n}{spq^\prime}- \frac{\vartheta n}{sp})= \frac{1}{\vartheta}(\nu +\frac{n}{sp} -\frac{\vartheta n}{sp})=0 .$$
Using Lemma \ref{lm:interpol} for the exponents $(\vartheta q^\prime)^\prime$ and $(\vartheta r^\prime)^\prime$ we get
\begin{equation}\label{eq:integ-improv}
\begin{aligned}
\norm{w^\vartheta}_{L^{ p q^\prime ,  p r^\prime}(Q_{R,R^{sp}})}^{\frac{p}{\vartheta}} &= \norm{w}_{L^{\vartheta p q^\prime , \vartheta p r^\prime}(Q_{R,R^{sp}})}^p \leq \norm{w}_{L^{p,\infty}(Q_{R,R^{s p}})}^p + \norm{w}_{L^{p_s^\star , p}(Q_{R,R^{s p}})}^p \\
 &\leq C(n,s,p)\beta^{p-1} \norm{w}_{L^{p q^\prime , p r^\prime}(Q_{R,R^{s p}})}^p \bigl(  R^{-sp \nu} \delta^{p-2} \bigr)
 \end{aligned}
 \end{equation}
%Or in another word 
%$$
%\norm{w^\vartheta}_{L^{ p q^\prime ,  p r^\prime}(Q_{R,R^{s p}})}^p \leq \bigl(C \, R^{-sp\nu} (\max{\lbrace 1 , \frac{\norm{f}_{L^{q,r}}}{R^{-sp\nu}} \rbrace }) \bigr)^\frac{1}{\vartheta} (\beta^\frac{1}{\vartheta})^{p-1} \norm{w}^\frac{p}{\vartheta}
%$$
Now we iterate this inequality with the following choice of exponents
$$\beta_0 = 1 ,\qquad \beta_{m+1} = \vartheta \beta_m = \vartheta^{m+1} .$$
With the notation
$$\phi_m := \norm{((u-v)^{+}_M +\delta)^\frac{\beta_m}{p} }_{L^{p q^\prime , pr^\prime}(Q_{R,R^{sp}})}^\frac{p}{\beta_m} =\norm{(u-v)^{+}_M + \delta}_{L^{\beta_m q^\prime , \beta_m r^\prime}(Q_{R,R^{sp}})} ,$$
\eqref{eq:integ-improv} reads
$$\phi_{m+1} \leq \bigl(C \, R^{-sp\nu} \delta^{p-2} \bigr)^\frac{1}{\vartheta^m}\vartheta^\frac{(p-1)m}{\vartheta^m} \phi_n.$$
Iterating this yields
\begin{equation}\label{eq:mos1}
 \phi_{m+1} \leq \bigl(C \, R^{-sp\nu} \delta^{p-2} \bigr)^{\sum_{j=0}^m \vartheta^{-j}} \vartheta^{(p-1) \sum_{j=0}^m j \vartheta^{-j}} \phi_0.
\end{equation}
Since $\vartheta >1 $, we have the following convergent series
$$\sum_{j=0}^\infty \vartheta^{-j} = \frac{\vartheta}{\vartheta -1}=\frac{n+sp\nu}{sp\nu},$$
and
$$\sum_{j=0}^\infty j \vartheta^{-j} = \frac{\vartheta}{(\vartheta -1)^2} = \frac{n^2 + nsp\nu}{s^2 p^2 \nu^2}.$$
By \eqref{eq:pq-dual} in Lemma \ref{lm:lpavr},
\begin{equation}\label{eq:phi-zero}
\begin{aligned}
\phi_0 =& \norm{(u-v)^{+}_M +\delta}_{L^{q^\prime , r^\prime}(Q_{R,R^{sp}})} \leq C(n,s,p)R^{spp^\prime \nu +n} \norm{f}_{L^{q,r}(Q_{R,R^{sp}})}^{\frac{1}{p-1}} + \delta R^{\frac{n}{q^\prime} + \frac{sp}{r^\prime}} \\
&\leq C(n,s,p) R^{n+sp\nu} (R^{\frac{sp\nu}{p-1}}\norm{f}_{L^{q,r}(Q_{R,R^{sp}})}^{\frac{1}{p-1}} +\delta) \leq 2C(n,s,p)R^{n+sp\nu} \delta.
\end{aligned}
\end{equation}
Inserting \eqref{eq:phi-zero} to \eqref{eq:mos1} and sending $n$ to infinity we obtain 
\[
\begin{aligned}
\norm{(u-v)^{+}_M +\delta}_{L^\infty(Q_{R,R^{sp}})} &\leq C \vartheta^{\frac{(p-1)\vartheta}{(\vartheta -1 )^2}} R^{-n - sp \nu} \delta^{(p-2)\frac{\vartheta}{\vartheta-1}} R^{sp \nu + n} \delta\\
& = C \vartheta^{\frac{(p-1)\vartheta}{(\vartheta -1 )^2}} \delta^{(p-1) + \frac{p-2}{\vartheta -1}} \\
&\leq C \vartheta^{\frac{(p-1)\vartheta}{(\vartheta -1 )^2}} \max{\lbrace 1 , \bigl( R^{sp\nu}\norm{f}_{L^{q,r}(Q_{R,R^{sp}})} \bigr)^{\frac{1}{p-1}} \rbrace}^{p-1 + \frac{p-2}{\vartheta -1}}  \\
& \leq C(n,s,p) \vartheta^{\frac{(p-1)\vartheta}{(\vartheta -1 )^2}} \bigl( 1 + R^{sp\nu +\frac{(p-2)sp\nu}{(p-1)(\vartheta-1)}}\norm{f}_{L^{q,r}(Q_{R,R^{sp}})}^{1+\frac{1}{\vartheta-1}\frac{p-2}{p-1} } \bigr).
\end{aligned}
\]
Since the above estimate is independent of $M$, we get 
\[
\begin{aligned}
&\norm{(u-v)^{+}}_{L^\infty(Q_{R,R^{sp}})}  \leq C(n,s,p) \vartheta^{\frac{(p-1)\vartheta}{(\vartheta -1 )^2}} \Bigl( 1 + R^{sp\nu +\frac{(p-2)sp\nu}{(p-1)(\vartheta-1)}}\norm{f}_{L^{q,r}(Q_{R,R^{sp}})}^{1+\frac{1}{\vartheta-1}\frac{p-2}{p-1} }  \Bigr) .
\end{aligned}
\]
Which is the desired result.

\textbf{Case $sp> n$}. Here we use Morrey's inequality \eqref{eq:Morrey} for the second term in \eqref{eq:testing-Morr-sob}. Instead of \eqref{eq:moser-diffu-sob} we obtain
\[
\begin{aligned}
\frac{\delta}{2(\beta + 1)} &\norm{w}_{L^{p,\infty}(B_R \times J)}^p \, + \frac{(C(n,s,p)R^{sp-n})^{-1}}{3\cdot 2^{p-1}}\frac{\beta p^p}{(\beta + p-1)^p} \bigl[ \delta^{p-1} \norm{w}_{L^{\infty, p}(B_R \times J)}^p - \delta^{\beta + p-1}  \abs{J} \bigr] \\
& \leq \frac{\delta}{\beta + 1} \sup_{t \in J} \int_{B_R} ((u-v)^{+}_M + \delta)^{\beta} \dd x \\
& \quad + \frac{C(n,s,p)^{-1}}{3\cdot 2^{p-1}}\frac{\beta p^p}{(\beta + p-1)^p}\int_J \norm{((u-v)^{+}_M +\delta)^{\frac{\beta + p-1}{p}} - \delta^{\frac{\beta + p-1}{p}}}_{L^{\infty}(B_R)}^p \\
& \leq \norm{f}_{L^{q,r}(B_R \times J)} \norm{w}_{L^{p q^\prime , p r^\prime}(B_R \times J)}^p \, + 2\delta \abs{B_R} \frac{ \norm{w}_{L^{p q^\prime , p r^\prime}(B_R \times J)}^p}{\abs{B_R}^{1-\frac{1}{q}}\abs{J}^{1-\frac{1}{r}}}.
\end{aligned}
\]
Following the same steps as in the case $sp < n$ we arrive at
\[
\begin{aligned}
&\norm{w}_{L^{p,\infty}(Q_{R,R^{s p}})}^p + R^{n-sp}\norm{w}_{L^{\infty , p}(Q_{R,R^{s p}})}^p \\
&\leq C(n,s,p) \beta^{p-1} \norm{w}_{L^{p q^\prime , p r^\prime}(Q_{R,R^{s p}})}^p  \Bigl( \frac{\norm{f}_{L^{q,r}(Q_{R,R^{sp}}}}{\delta} + R^{-sp\nu} + \delta^{p-2} R^{-sp\nu } \Bigr) \\
&\leq C(n,s,p,l)\beta^{p-1} \norm{w}_{L^{p q^\prime , p r^\prime}(Q_{R,R^{s p}})}^p (\delta^{p-2}R^{-sp\nu}).
\end{aligned}
\]
Now choose $\vartheta = 2-\frac{1}{r} - \frac{1}{q}$. Then $\vartheta >1 $, since $\frac{1}{r} + \frac{1}{q} < 1$. Therefore, $(\vartheta r^\prime)^\prime$ and $(\vartheta q^\prime)^\prime$ satisfy 
$$1- \frac{1}{(\vartheta r^\prime)^\prime}-\frac{1}{(\vartheta q^\prime)^\prime} = \frac{1}{\vartheta}(\frac{1}{r^\prime} +\frac{1}{q^\prime}-\vartheta)=0,$$
and we can apply Lemma \ref{lm:interpol-morr} with the exponents $(\vartheta q^\prime)^\prime$ and $(\vartheta r^\prime)^\prime$. This gives
\begin{equation}\label{eq:morr-mos-iter}
\begin{aligned}
\norm{w^\vartheta}^{\frac{p}{\vartheta}}_{L^{p q^\prime, p r^\prime}(Q_{R,R^{sp}})} &= \norm{w}^p_{L^{p \vartheta q^\prime, p \vartheta r^\prime}(Q_{R,R^{sp}})} \leq R^{\frac{sp-n}{(\vartheta q^\prime)^\prime}}(\norm{w}^p_{L^{\infty,p}}(Q_{R,R^{sp}}) + R^{n-sp}\norm{w}^p_{L^{p, \infty}}(Q_{R,R^{sp}}) ) \\
& \leq C(n,s,p) \beta^{p-1} R^{\frac{sp-n}{\vartheta r^\prime} - sp\nu} \norm{w}_{L^{p q^\prime , p r^\prime}(Q_{R,R^{s p}})}^p  \delta^{p-2} .
\end{aligned}
\end{equation}
We apply \eqref{eq:morr-mos-iter} with the exponents
$$\beta_0 = 1 ,\qquad \beta_{m+1} = \vartheta \beta_m = \vartheta^{n+1} .$$
Let
$$\phi_m := \norm{((u-v)^{+}_M +\delta)^\frac{\beta_m}{p} }_{L^{p q^\prime , pr^\prime}(Q_{R,R^{sp}})}^\frac{p}{\beta_m} =\norm{(u-v)^{+}_M + \delta}_{L^{\beta_n q^\prime , \beta_m r^\prime}(Q_{R,R^{sp}})} .$$
Then \eqref{eq:morr-mos-iter} reads
$$\phi_{m+1} \leq \bigl(C \, R^{\frac{sp-n}{\vartheta r^\prime} - sp\nu} \delta^{p-2} \bigr)^\frac{1}{\vartheta^m}\theta^\frac{(p-1)m}{\vartheta^m} \phi_m.$$
By iterating the above inequality, we get
\begin{equation}\label{eq:mos-morre}
 \phi_{m+1} \leq \bigl(C \, R^{\frac{sp-n}{\vartheta r^\prime} - sp\nu} \delta^{p-2} \bigr)^{\sum_{j=0}^m \vartheta^{-j}} \vartheta^{(p-1) \sum_{j=0}^m j \vartheta^{-j}} \phi_0.
\end{equation}
Since $\vartheta > 1$, we have the following convergent series
$$\sum_{j=0}^\infty \vartheta^{-j} = \frac{\vartheta}{\vartheta -1}= 1+ \frac{1}{1-\frac{1}{r}-\frac{1}{q}}$$
and
$$\sum_{j=0}^\infty j \vartheta^{-j} = \frac{\vartheta}{(\vartheta -1)^2}. $$
By \eqref{eq:pq-dual} in Lemma \ref{lm:lpavr} we have
$$\phi_0 = \norm{(u-v)^{+}_M + \delta)}_{L^{q^\prime,r^\prime}(Q_{R,R^{sp}})} \leq C(n,s,p) R^{spp^\prime \nu + n} \norm{f}_{L^{q,r}(Q_{R,R^{sp}})}^{\frac{1}{p-1}} + \delta R^{\frac{n}{q^\prime} + \frac{sp}{r^\prime}} 
\leq C(n,s,p)R^{ n+ sp\nu}(2 \delta).$$
Inserting this into \eqref{eq:mos-morre}, and sending $n$ to infinity we get
\[
\begin{aligned}
\norm{(u-v)^{+}_M +\delta }_{L^\infty(Q_{R,R^{sp}})} &\leq C(n,s,p) \vartheta^{\frac{(p-1)\vartheta}{(\vartheta -1 )^2}} R^{\frac{\vartheta}{\vartheta-1}(\frac{sp-n}{\vartheta r^\prime} - sp\nu)} \delta^{(p-2)\frac{\vartheta}{\vartheta-1}} R^{sp \nu + n} \delta\\
&= C(n,s,p)\vartheta^{\frac{(p-1)\vartheta}{(\vartheta -1 )^2}} R^{\frac{sp-n}{(\vartheta -1)r^\prime} -\frac{sp\nu}{\vartheta-1} +n} \delta^{1+(p-2)\frac{\vartheta}{\vartheta-1}} \\
&  = C(n,s,p)\vartheta^{\frac{(p-1)\vartheta}{(\vartheta -1 )^2}} R^{\frac{n}{\vartheta -1}(\vartheta - 1 -\frac{1}{r^\prime}) + \frac{sp}{\vartheta -1}(\frac{1}{r^\prime} -(1-\frac{1}{r} - \frac{n}{spq})} \delta^{p-1+ \frac{p-2}{\vartheta-1}}
\\
& = C(n,s,p)\vartheta^{\frac{(p-1)\vartheta}{(\vartheta -1 )^2}} R^{\frac{n}{\vartheta -1}(\frac{-1}{q}) + \frac{sp}{\vartheta - 1}(\frac{n}{spq})} \delta^{p-1+ \frac{p-2}{\vartheta-1}}  \\
& \leq  C(n,s,p)\vartheta^{\frac{(p-1)\vartheta}{(\vartheta -1 )^2}} \bigl( 1 + R^{sp\nu +\frac{(p-2)sp\nu}{(p-1)(\vartheta-1)}}\norm{f}_{L^{q,r}(Q_{R,R^{sp}})}^{1+\frac{1}{\vartheta-1}\frac{p-2}{p-1} } \bigr).
\end{aligned}
\]
Hence, we arrive at the desired estimate
\[
\norm{(u-v)^{+} }_{L^\infty(Q_{R,R^{sp}})} \leq C(n,s,p)\vartheta^{\frac{(p-1)\vartheta}{(\vartheta -1 )^2}}  \bigl( 1 + R^{sp\nu +\frac{(p-2)sp\nu}{(p-1)(\vartheta-1)}}\norm{f}_{L^{q,r}(Q_{R,R^{sp}})}^{1+\frac{1}{\vartheta-1}\frac{p-2}{p-1} } \bigr).
\]

\textbf{Case} sp=n. Here we use the critical case of Sobolev-Morrey inequality, \eqref{eq:critic-sob} with 
\begin{equation}\label{eq:choice-of-l}
\max{ \lbrace\frac{p}{r^\prime}(1-\frac{1}{r} - \frac{1}{q})^{-1}, q^\prime \rbrace} < l < \infty.
\end{equation}
This applied for the second term in \eqref{eq:testing-Morr-sob} implies
\[
\begin{aligned}
\frac{\delta}{2(\beta + 1)} &\norm{w}_{L^{p,\infty}(B_R \times J)}^p \, + \frac{(C(n,s,p,l)R^{\frac{np}{l}})^{-1}}{3\cdot 2^{p-1}}\frac{\beta p^p}{(\beta + p-1)^p} \bigl[ \delta^{p-1} \norm{w}_{L^{l, p}(B_R \times J)}^p - \delta^{\beta + p-1} \abs{B_R}^{\frac{p}{l}} \abs{J} \bigr] \\
& \leq \frac{\delta}{\beta + 1} \sup_{t \in J} \int_{B_R} ((u-v)^{+}_M + \delta)^{\beta} \dd x \\
& \quad + \frac{C(n,s,p,l)^{-1}}{3\cdot 2^{p-1}}\frac{\beta p^p}{(\beta + p-1)^p}\int_J \norm{((u-v)^{+}_M +\delta)^{\frac{\beta + p-1}{p}} - \delta^{\frac{\beta + p-1}{p}}}_{L^{l}(B_R)}^p \\
& \leq \norm{f}_{L^{q,r}(B_R \times J)} \norm{w}_{L^{p q^\prime , p r^\prime}(B_R \times J)}^p \, + 2\delta \abs{B_R} \frac{ \norm{w}_{L^{p q^\prime , p r^\prime}(B_R \times J)}^p}{\abs{B_R}^{1-\frac{1}{q}}\abs{J}^{1-\frac{1}{r}}}.
\end{aligned}
\]
Following the same step as in the previous two cases, we arrive at
\[
\begin{aligned}
&\norm{w}_{L^{p,\infty}(Q_{R,R^{s p}})}^p + R^{\frac{-np}{l}}\norm{w}_{L^{l, p}(Q_{R,R^{s p}})}^p \\
&\leq C(n,s,p,l) \beta^{p-1} \norm{w}_{L^{p q^\prime , p r^\prime}(Q_{R,R^{s p}})}^p  \Bigl( \frac{\norm{f}_{L^{q,r}(Q_{R,R^{sp}}}}{\delta} + R^{-sp\nu} + \delta^{p-2} R^{-sp\nu } \Bigr) \\
&\leq C(n,s,p,l)\beta^{p-1} \norm{w}_{L^{p q^\prime , p r^\prime}(Q_{R,R^{s p}})}^p (\delta^{p-2}R^{-sp\nu}).
\end{aligned}
\]
Now we choose $\vartheta = 2-\frac{1}{r}- \frac{1}{q} - \frac{p}{lr^\prime}$. Notice that due to the choice of $l$, \eqref{eq:choice-of-l}, we have $\vartheta > 1$. Then the exponents $(\vartheta r^\prime)^\prime$ and $(\vartheta q^\prime)^\prime$ satisfy
$$1 - \frac{1}{(\vartheta r^\prime)^\prime} - \frac{1}{(\vartheta q^\prime)^\prime} = \frac{p}{l \vartheta r^\prime} . $$
Therefore, we can apply Lemma \ref{lm:interpol-critic} with the exponents $(\vartheta r^\prime)^\prime$ and $(\vartheta q^\prime)^\prime$ to get
\begin{equation}\label{eq:moser-ncritic-preiter}
\begin{aligned}
\norm{w^\vartheta}_{L^{pq^\prime,pr^\prime}(Q_{R,R^{sp}})}^{\frac{p}{\vartheta}}&= \norm{w}_{L^{p \vartheta q^\prime, p\vartheta r^\prime}(Q_{R,R^{sp}})}^p \leq 
R^{sp(1-\frac{1}{(\vartheta r^\prime)^\prime} - \frac{1}{(\vartheta q^\prime)^\prime})}(\norm{w}_{L^{p,\infty}(Q_{R,R^{s p}})}^p + R^{\frac{-np}{l}}\norm{w}_{L^{\infty , p}(Q_{R,R^{s p}})}^p ) \\
&= R^{\frac{np}{l \vartheta r^\prime}} (\norm{w}_{L^{p,\infty}(Q_{R,R^{s p}})}^p + R^{\frac{-np}{l}}\norm{w}_{L^{\infty , p}(Q_{R,R^{s p}})}^p )   \\
&\leq C(n,s,p,l)\beta^{p-1}R^{\frac{np}{l \vartheta r^\prime}-sp\nu} \norm{w}_{L^{pq^\prime,p r^\prime}(Q_{R,R^{sp}})}^{p} \delta^{p-2}.
\end{aligned}
\end{equation}
We apply \eqref{eq:moser-ncritic-preiter} with the exponents
$$\beta_0 = 1 ,\qquad \beta_{m+1} = \vartheta \beta_m = \vartheta^{m+1} .$$
Let
$$\phi_n := \norm{((u-v)^{+}_M +\delta)^\frac{\beta_m}{p} }_{L^{p q^\prime , pr^\prime}(Q_{R,R^{sp}})}^\frac{p}{\beta_m} =\norm{(u-v)^{+}_M + \delta}_{L^{\beta_m q^\prime , \beta_m r^\prime}(Q_{R,R^{sp}})} .$$
Then \eqref{eq:moser-ncritic-preiter} reads
$$\phi_{m+1} \leq \bigl(C \, R^{\frac{np}{l\vartheta r^\prime} - sp\nu} \delta^{p-2} \bigr)^\frac{1}{\vartheta^m}\theta^\frac{(p-1)m}{\vartheta^m} \phi_m.$$
By iterating the above inequality, we get
\begin{equation}\label{eq:mos-critic}
 \phi_{m+1} \leq \bigl(C \, R^{\frac{np}{l \vartheta r^\prime} - sp\nu} \delta^{p-2} \bigr)^{\sum_{j=0}^m \vartheta^{-j}} \vartheta^{(p-1) \sum_{j=0}^m j \vartheta^{-j}} \phi_0.
\end{equation}
Since $\vartheta > 1$, we have the following convergent series
$$\sum_{j=0}^\infty \vartheta^{-j} = \frac{\vartheta}{\vartheta -1}$$
and
$$\sum_{j=0}^\infty j \vartheta^{-j} = \frac{\vartheta}{(\vartheta -1)^2}. $$
By \eqref{eq:pq-dual} in Lemma \ref{lm:lpavr} we obtain
$$\phi_0 = \norm{(u-v)^{+}_M}_{L^{q^\prime, r^\prime}(Q_{R,R^{sp}})} \leq C(n,s,p,q)R^{spp^\prime \nu + n } \norm{f}_{L^{q,r}(Q_{R,R^{sp}})}^{\frac{1}{p-1}} + \delta R^{\frac{n}{q^\prime} + \frac{sp}{r^\prime}} \leq C(n,s,p,q)R^{n+sp\nu} \delta .$$
Inserting this into \eqref{eq:moser-ncritic-preiter}, and sending $n$ to infinity we get
\[
\begin{aligned}
\norm{(u-v)^{+}_M+\delta}_{L^\infty(Q_{R,R^{sp}})} & \leq C(n,s,p,q,l) \vartheta^{\frac{(p-1)\vartheta}{(\vartheta - 1)^2}} R^{\frac{\vartheta}{\vartheta-1}(\frac{np}{l\vartheta r^\prime} - sp\nu)} \delta^{(p-2)\frac{\vartheta}{\vartheta - 1}} R^{n+sp\nu} \delta \\
&= C(n,s,p,q,l) \vartheta^{\frac{(p-1)\vartheta}{(\vartheta - 1)^2}} R^{\frac{np}{(\vartheta-1)l r^\prime}-\frac{sp\nu}{\vartheta-1} +n}\delta^{1+(p-2)\frac{\vartheta}{\vartheta -1}}\\
&= C(n,s,p,q,l) \vartheta^{\frac{(p-1)\vartheta}{(\vartheta - 1)^2}} 
R^{\frac{n}{(\vartheta-1)} (\vartheta-1 + \frac{p}{l r^\prime})  - \frac{sp\nu}{\vartheta-1} }\delta^{p-1+ \frac{(p-2)}{\vartheta -1}} \\
&= C(n,s,p,q,l) \vartheta^{\frac{(p-1)\vartheta}{(\vartheta - 1)^2}} 
R^{\frac{n \nu}{(\vartheta-1)}   - \frac{sp\nu}{\vartheta-1} }\delta^{p-1+ \frac{(p-2)}{\vartheta -1}}  \\
&\leq C(n,s,p,q,l) \vartheta^{\frac{(p-1)\vartheta}{(\vartheta - 1)^2}} \bigl( 1 + R^{sp\nu +\frac{(p-2)sp\nu}{(p-1)(\vartheta-1)}}\norm{f}_{L^{q,r}(Q_{R,R^{sp}})}^{1+\frac{1}{\vartheta-1}\frac{p-2}{p-1} } \bigr) .
\end{aligned}
\]
Hence we arrive at the desired estimate
\[
\norm{(u-v)^{+}}_{L^\infty(Q_{R,R^{sp}})} \leq \delta +  C(n,s,p,q,l) \vartheta^{\frac{(p-1)\vartheta}{(\vartheta - 1)^2}} \bigl( 1 + R^{sp\nu +\frac{(p-2)sp\nu}{(p-1)(\vartheta-1)}}\norm{f}_{L^{q,r}(Q_{R,R^{sp}})}^{1+\frac{1}{\vartheta-1}\frac{p-2}{p-1} } \bigr).
\]
\end{proof}
Notice that $-u$ is a solution to the same type of problem, and we can apply the above proposition to $-u$. 
Since $-v$ is the $(s,p)$-caloric replacement of $-u$ we get the same bound on $\norm{(-u+v)^+}_{L^\infty(Q_{R,R^{sp}})}$, as a result we get a bound on the $\norm{u-v}_{L^\infty(Q_{R,R^{sp}})}$.

 \begin{corollary}\label{thm:moser-bound}
Let $u$ be a solution of $\partial_t u + (-\Delta_p)^s u=f$ in $Q_{\sigma,\sigma^{sp}}(x_0,T_0)$ with $f \in L^{q,r}_{\loc} (Q_{\sigma,\sigma^{sp}}(x_0,T_0))$ and let $v$ be the $(s,p)$-caloric replacement of $u$ in $Q_{R,R^{s p}}$. Assume further that  $r \geq p^\prime$,
$$\frac{1}{r} + \frac{n}{spq} < 1 \quad\text{ and } q\geq (p_s^\star)^\prime \quad \text{in the case } sp < n,$$ 
and
$$\frac{1}{r} + \frac{1}{q} < 1 \quad\text{ and  } q > 1\quad \text{in the case } sp\geq n.$$
If $sp \neq n$ then
$$
\norm{u-v}_{L^\infty(Q_{R,R^{s p}}(x_0,T_0))} \leq C(n,s,p) \vartheta^{\frac{(p-1)\vartheta}{(\vartheta -1 )^2}} \bigl( 1 + R^{sp\nu +\frac{(p-2)sp\nu}{(p-1)(\vartheta-1)}}\norm{f}_{L^{q,r}(Q_{R,R^{sp}}(x_0,T_0))}^{1+\frac{1}{\vartheta-1}\frac{p-2}{p-1} } \bigr),
$$
where $\nu = 1-\frac{1}{r} - \frac{n}{spq}$ and
 $$ \vartheta = 1+\frac{sp\nu}{n} \quad \text{if}\quad sp<n, \quad \text{and} \quad \vartheta= 2-\frac{1}{r}-\frac{1}{q} \quad \text{if} \quad sp> n. $$
If $sp=n$, then for any $l$ such that $ \frac{p}{r^\prime}(1-\frac{1}{r} - \frac{1}{q})^{-1}  <l < \infty$, we have
\[
\norm{u-v}_{L^\infty(Q_{R,R^{s p}}(x_0,T_0))}\leq C(n,s,p,q,l) \vartheta^{\frac{(p-1)\vartheta}{(\vartheta - 1)^2}} \bigl( 1 + R^{sp\nu +\frac{(p-2)sp\nu}{(p-1)(\vartheta-1)}}\norm{f}_{L^{q,r}(Q_{R,R^{sp}}(x_0,T_0))}^{1+\frac{1}{\vartheta-1}\frac{p-2}{p-1} } \bigr) , 
\]
where $\vartheta= 2-\frac{1}{r}- \frac{1}{q} - \frac{p}{lr^\prime}$ and $\nu= 1-\frac{1}{r} - \frac{1}{q}$.
\end{corollary}
 Now we combine the local boundedness results for the equations with zero right hand side (see \cite{Strqv} and also \cite{Ding}) with Proposition \ref{thm:sub-moser-bound} to prove local boundedness for the equation with nonzero right hand side.
 \begin{proof}[Proof of Theorem \ref{thm:loc-bound}]
 For $u$, a local weak solution of 
 $$\partial_t u + (-\Delta_p)^su = f(x,t), \qquad in \; Q_{2R,(2R)^{sp}}(x_0,T_0),$$
 we consider $v$ to be the $(s,p)$-caloric replacement in $Q_{R,R^{sp}}(x_0,T_0)$,
 \[
 \begin{cases}
v_t + (-\Delta_p)^s v=0 \quad & in \;\; Q_{R,R^{sp}}(x_0,T_0), \\
v=u \quad & in \;\; (\R^n \setminus B_R(x_0) ) \times [T_0-R^{sp},T_0], \\
v(x,T_0-R^{sp}) = u(x,T_0- R^{sp}) \quad & in \;\; B_R(x_0).
\end{cases}
 \]
 By \cite[Theorem 1.1]{Ding}
 $$\norm{v}_{L^\infty(Q_{\frac{R}{2},(\frac{R}{2})^{sp}}(x_0,T_0))} \leq C\Bigl[   1+ \Bigl(\dashint_{Q_{R,R^{sp}}(x_0,T_0)} \abs{v}^{p} \dd x \dd t\Bigr)^{\frac{1}{p}} \Bigr] + 
 \sup_{T_0-R^{sp}<t\leq T_0} \mathrm{Tail}_{p-1,sp} \bigl(v(\frarg ,t);x_0,\frac{R}{2} \bigr) $$
with $C$ depending on $n,s$ and $p$. Therefore,
 \begin{equation}\label{eq:L-infty-uplus}
 \begin{aligned}
 \norm{u}_{L^\infty(Q_{\frac{R}{2},(\frac{R}{2})^{sp}}(x_0,T_0))}&\leq  \norm{u-v}_{L^\infty(Q_{\frac{R}{2},(\frac{R}{2})^{sp}}(x_0,T_0))}
 +\norm{v}_{L^\infty(Q_{\frac{R}{2},(\frac{R}{2})^{sp}}(x_0,T_0))} \\
 &\leq  C\Bigl[   1+\Bigl( \dashint_{Q_{R,R^{sp}}(x_0,T_0)} \abs{v}^{p} \dd x \dd t\Bigr)^{\frac{1}{p}} \Bigr] + 
 \sup_{T_0-R^{sp}<t\leq T_0} \mathrm{Tail}_{p-1,sp} \bigl(v(\frarg ,t);x_0,\frac{R}{2} \bigr)  \\
 &\qquad + \norm{u-v}_{L^\infty(Q_{R,R^{sp}}(x_0,T_0))} \\
 & \leq C \Bigl[ 1+ \Bigl( 2^{p-1} \dashint_{Q_{R,R^{sp}}(x_0,T_0)} \abs{u}^{p} \dd x \dd t + 2^{p-1} \dashint_{Q_{R,R^{sp}}(x_0,T_0)} \abs{u-v}^{p} \dd x \dd t \Bigr)^{\frac{1}{p}} \Bigr]\\
 & +  \sup_{T_0-R^{sp}<t\leq T_0} \mathrm{Tail}_{p-1,sp} \bigl(v(\frarg ,t);x_0,\frac{R}{2} \bigr) + \norm{u-v}_{L^\infty(Q_{R,R^{sp}}(x_0,T_0))} \\
& \leq C \Bigl[ 1 +  \Bigl(\dashint_{Q_{R,R^{sp}}(x_0,T_0)} \abs{u}^{p} \dd x \dd t \Bigr)^{\frac{1}{p}} + \norm{u-v}_{L^\infty(Q_{R,R^{sp}}(x_0,T_0))} \Bigr] \\
& + \sup_{T_0-R^{sp}<t\leq T_0} \mathrm{Tail}_{p-1,sp} \bigl(v(\frarg ,t);x_0,\frac{R}{2} \bigr)
 \end{aligned}
 \end{equation}
 Using Lemma \ref{lm:tail_comparison} in \eqref{eq:L-infty-uplus} we arrive at
 \begin{equation}\label{eq:Linfty-compare}
 \begin{aligned}
\norm{u}_{L^\infty(Q_{\frac{R}{2},(\frac{R}{2})^{sp}})}&\leq C \Bigl[ 1 +  \Bigl( \dashint_{Q_{R,R^{sp}}(x_0,T_0)} \abs{u}^{p} \dd x \dd t \Bigr)^{\frac{1}{p}} + \norm{(u-v)}_{L^\infty(Q_{R,R^{sp}}(x_0,T_0))} \Bigr] \\
 + 2 \sup_{T_0-R^{sp}<t\leq T_0} &\mathrm{Tail}_{p-1,sp} \bigl(u(\frarg ,t);x_0,\frac{R}{2} \bigr) + 2^{1 +\frac{n}{p-1}}\Bigl(  \sup_{T_0-R^{sp}<t\leq T_0} \dashint_{B_{R}(x_0)} \abs{u -v}^{p-1} \dd x \Bigr)^{\frac{1}{p-1}} \\
 \leq & C \Bigl[ 1 +  \Bigl(\dashint_{Q_{R,R^{sp}}(x_0,T_0)} \abs{u}^{p} \dd x \dd t \Bigr)^{\frac{1}{p}} + \norm{(u-v)}_{L^\infty(Q_{R,R^{sp}}(x_0,T_0))} \Bigr] \\
 & + 2 \sup_{T_0-R^{sp}<t\leq T_0} \mathrm{Tail}_{p-1,sp} \bigl(u(\frarg ,t);x_0,\frac{R}{2} \bigr),
 \end{aligned}
 \end{equation}
 where $C=C(n,s,p)$. Finally, using Proposition \ref{thm:sub-moser-bound} to estimate the term $\norm{u-v}_{L^\infty(Q_{R,R^{sp}})}$, in \eqref{eq:Linfty-compare} we get the desired result. Here the estimate is written in the case $sp\neq n$ 
 \[
 \begin{aligned}
 \norm{u}_{L^\infty(Q_{\frac{R}{2},(\frac{R}{2})^{sp}})}&\leq 2 \sup_{T_0-R^{sp}<t\leq T_0} \mathrm{Tail}_{p-1,sp} \bigl(u(\frarg ,t);x_0,\frac{R}{2} \bigr) + 
 C(n,s,p) \Bigl [ 1 + \Bigl( \dashint_{Q_{R,R^{sp}}(x_0,T_0)} \abs{u}^{p} \dd x \dd t \Bigr)^{\frac{1}{p}} \\
 &+ \vartheta^{\frac{(p-1)\vartheta}{(\vartheta -1 )^2}} \bigl( 1 + R^{sp\nu +\frac{(p-2)sp\nu}{(p-1)(\vartheta-1)}}\norm{f}_{L^{q,r}(Q_{R,R^{sp}})}^{1+\frac{1}{\vartheta-1}\frac{p-2}{p-1} } \bigr) \Bigr].
 \end{aligned}
 \]
 \end{proof}

\begin{theorem}\label{thm:basic-holder}
Let $f \in L^{q,r}_{\loc}(Q_{R_1,R_1^{sp}}(z,T_1))$ with $r \geq p^\prime$, 
$$\frac{1}{r} + \frac{n}{spq} < 1 \quad\text{and } q\geq (p_s^\star)^\prime \quad \text{in the case } sp < n,$$ 
and
$$\frac{1}{r} + \frac{1}{q} < 1 \quad\text{ and  } q > 1\quad \text{in the case } sp\geq n.$$
Consider a bounded, local weak solution $u \in L^p(I; W^{s,p}_{\loc}(B_{R_1}(z)))\cap C(I;L^2_{\loc}(B_{R_1}(z)))\cap L^\infty(I;L^{p-1}_{sp}(\R^n)) \cap L^\infty(Q_{R_1,R_1^{sp}}(z,T_1))$ of the equation 
$$ \partial_t u + (-\Delta_p)^s u =f \qquad in \; Q_{R_1,R_1^{sp}}(z,T_1). $$
Then $u$ is locally H\"older continuous in time and space.
In particular, there exists a $\zeta >0$, such that for $\sigma <1$, $(x_1,t_1), \; (x_2,t_2) \in Q_{\sigma R_1,(\sigma R_1)^{sp}}(z,T_1) $, there holds
$$\abs{u(x_1,t_1)-u(x_2,t_2)} \leq C (\abs{x_1-x_2}^\zeta  + \abs{t_1-t_2}^{\frac{\zeta}{sp}} ) ,$$
with $C$ depending on 
$$n,\;s,\; p,\; R_1 ,\;  \sigma ,\; \sup_{T_1-R_1^{sp}<t\leq T_1} \mathrm{Tail}_{p-1,sp}(u(\frarg,t);z,R_1),\;  \norm{f}_{L^{q,r}(Q_{ R_1,R_1^{sp}}(z,T_1))} \text{, and }\norm{u}_{L^\infty(Q_{R_1,R_1^{s p}}(z,T_1))}.$$
\end{theorem}
\begin{proof}
Take a cylinder $Q_{\sigma R_1,(\sigma R_1)^{sp}}(z,T_1) \subset Q_{R_1,R_1^{sp}}(z,T_1)$ and let 
%$e=\dist \bigl( Q_{R_0,R_0^{sp}},\partial^\star (I\times\Omega) \bigr)$,
$d:=\min{ \lbrace R_1(1-\sigma),R_1(1-\sigma^{sp})^{\frac{1}{sp}}\rbrace } > 0 $. For any point $(x_0,T_0) \in Q_{\sigma R_1,(\sigma R_1)^{sp}}(z,T_1) $ consider the $(s,p)$-caloric replacement of $u$ in the  cylinder $Q_{R,R^{sp}}(x_0,T_0)$ with $R \leq \min{ \lbrace 1 , d \rbrace}$. The choice of $d$ implies that $Q_{R,R^{sp}}(x_0,T_0)\subset Q_{ R_1,R_1^{sp}}(z,t) $.
%by the previous lemma for $r \leq \frac{R}{16}$
First, we observe that:
\begin{equation}\label{eq:campanto-split-2}
\begin{aligned}
\dashint_{Q_{\rho,\rho^{sp}}(x_0,T_0)} &\abs{u-\bar{u}_{(x_0,T_0),\rho}}^p \dd x \dd t \leq C(p)  \dashint_{Q_{\rho,\rho^{sp}}(x_0,T_0)} \abs{u-v}^p \dd x \dd t \\
&+ C(p) \dashint_{Q_{\rho,\rho^{sp}}(x_0,T_0)} \abs{\bar{u}_{(x_0,T_0),\rho} - \bar{v}_{(x_0,T_0),\rho}}^p \dd x \dd t + C(p) \dashint_{Q_{\rho,\rho^{sp}}(x_0,T_0)} \abs{v-\bar{v}_{(x_0,T_0),\rho}}^p \dd x \dd t \\
& \leq 2C(p) \dashint_{Q_{\rho,\rho^{sp}}(x_0,T_0)} \abs{u-v}^p \dd x \dd t + C(p) \dashint_{Q_{\rho,\rho^{sp}}(x_0,T_0)} \abs{v-\bar{v}_{(x_0,T_0),\rho}}^p \dd x \dd t .
\end{aligned}
\end{equation}
For $\rho \leq \frac{R}{2}$, $v$ is H\"older continuous in $Q_{\rho,\rho^{sp}}(x_0,T_0)$ by Theorem \ref{thm:thm1tail}, and by the Mean Value Theorem there is a point $(\tilde{x}_0,\tilde{t}_0)\in Q_{\rho,\rho^{sp}}$ such that $\bar{v}_{x_0,t_0}= v(\tilde{x}_0,\tilde{t}_0)$. With the notation 
$$\Mcal:= 1+ \norm{v}_{L^\infty(Q_{R,(R)^{sp}})} + \sup_{T_0-R^{sp} < t \leq T_0}\mathrm{Tail}_{p-1, sp}(v(\frarg,t); x_0,R) , $$ 
Theorem \ref{thm:thm1tail} implies :
\[
\begin{aligned}
\abs{v(x,t)-\bar{v}_{(x_0,t_0),\rho}} &\leq C\Bigl( \Mcal\Bigl(\frac{x-\tilde{x}_0}{R}\Bigr)^{\frac{\Theta}{2}} + \Mcal^{p-1} \Bigl( \frac{t-\tilde{t}_0}{R^{sp}}\Bigr)^{\frac{\Gamma}{2}} \Bigr)\\
& \leq C \Mcal^{p-1} \Bigl( \bigl(\frac{2\rho}{R} \bigr)^{\frac{\Theta }{2}} + \bigl((\frac{\rho}{R} )^{sp}\bigr)^{\frac{\Gamma}{2}} \Bigr),\quad \text{for } \; (x,t) \in Q_{\rho,\rho^{sp}}(x_0,T_0)
\end{aligned}
\]
with $C=C(n,s,p)$. Therefore,
\begin{equation}\label{eq:decay-camapnato-homogen}
\begin{aligned}
\dashint_{Q_{\rho,\rho^{sp}}(x_0,T_0)}& \abs{v-\bar{v}_{(x_0,t_0),\rho}}^p \dd x \dd t \leq  C\Mcal^{p(p-1)} \dashint_{Q_{\rho,\rho^{sp}}(x_0,T_0)}  \bigl(\frac{2\rho}{R} \bigr)^{\frac{\Theta p}{2}} + \bigl((\frac{\rho}{R})^{sp} \bigr)^{\frac{\Gamma p}{2}} \dd x \dd t \\
&\leq C \Mcal^{p(p-1)}\Bigl(  \bigl(\frac{\rho}{R}\bigr)^{\frac{\Theta p}{2}} + \bigl(\frac{\rho}{R}\bigr)^{\frac{\Gamma p}{2}}  \Bigr) \\
&\leq C \Bigl( \frac{\rho}{R}\Bigr)^{\delta p} \Bigl[ 1+ \norm{v}_{L^\infty(Q_{R,R^{sp}}(x_0,T_0))}^p + \sup_{T_0-R^{sp} < t \leq T_0}\mathrm{Tail}_{p-1, sp}(v(\frarg,t);  x_0,R)^p \Bigr]^{p-1} ,
\end{aligned}
\end{equation}
where the constants $C$ depends on $n,s,$ and $p$, and we have defined $\delta:= \min{ \bigl\lbrace  \frac{\Theta}{2} , \frac{\Gamma}{2}\bigr \rbrace}$.

Moreover, by Lemma \ref{lm:lpavr}
\begin{equation}\label{eq:campanato-difference}
\begin{aligned}
\dashint_{Q_{\rho,\rho^{sp}}(x_0,T_0)} \abs{u-v}^p \dd x \dd t & \leq \Bigl( \frac{R}{\rho} \Bigr)^n \frac{R^{sp}}{\rho^{sp}} \dashint_{Q_{R,(R)^{sp}}} \abs{u-v}^p \dd x \dd t \\
&\leq C(n,s,p) \bigl(\frac{R}{\rho} \bigr)^{n+sp} R^{\xi} \; \norm{f}_{L^{q,r}(Q_{R,R^{sp}}(x_0,T_0))}^{p^\prime},
\end{aligned}
\end{equation}
where $\xi$ is defined  in Lemma \ref{lm:lpavr}. Notice that $\xi>0$ by our assumptions on $q$ and $r$. Inserting \eqref{eq:campanato-difference} and \eqref{eq:decay-camapnato-homogen} in \eqref{eq:campanto-split-2} we arrive at
\[
\begin{aligned}
&\dashint_{Q_{\rho,\rho^{sp}}(x_0,T_0)} \abs{u-\bar{u}_{(x_0,T_0),\rho}}^p \dd x \dd t \leq   C(n,s,p) \bigl(\frac{R}{\rho} \bigr)^{n+sp} R^{\xi} \; \norm{f}_{L^{q,r}(Q_{R,R^{sp}}(x_0,T_0))}^{p^\prime} \\
&+ C(n,s,p) \bigl( \frac{\rho}{R}  \bigr)^{\delta p} \Bigl(1+ \norm{v}_{L^\infty(Q_{R,R^{sp}}(x_0,T_0))}^p + \sup_{T_0-R^{sp}\leq t\leq T_0} \mathrm{Tail}_{p-1,sp}(u(\frarg,t);x_0,R)^p \Bigr)^{p-1} \\
&\leq C(n,s,p) \bigl(\frac{R}{\rho} \bigr)^{n+sp} R^{\xi} \; \norm{f}_{L^{q,r}(Q_{R,R^{sp}}(x_0,T_0))}^{p^\prime}  + C(n,s,p) \bigl( \frac{\rho}{R}  \bigr)^{\delta p} \Bigl( \norm{u}_{L^\infty(Q_{R,R^{sp}}(x_0,T_0))}^p \\
&+ \norm{u-v}_{L^\infty(Q_{R,R^{sp}}(x_0,T_0))}^p 
 + \sup_{T_0-R^{sp}\leq t\leq T_0} \mathrm{Tail}_{p-1,sp}(u(\frarg,t);x_0,R)^p \Bigr)^{p-1} .
\end{aligned}
\]
Using Corollary \ref{thm:moser-bound} we get:
\[
\begin{aligned}
&\dashint_{Q_{\rho,\rho^{sp}}(x_0,T_0)} \abs{u-\bar{u}_{(x_0,T_0),r}}^p \dd x \dd t \leq C(n,s,p) (\frac{R}{\rho})^{n+sp} R^\xi \norm{f}_{L^{q,r}(Q_{R,R^{sp}}(x_0,T_0))}^{p^\prime} \\
&+ C(n,s,p)(\frac{\rho}{R})^{\delta p} \Bigl[\norm{u}_{L^\infty(Q_{R,R^{sp}}(x_0,T_0))}^p + \sup_{T_0-R^{sp}\leq t\leq T_0} \mathrm{Tail}_{p-1,sp}(u(\frarg,t);x_0,R)^p \\
&+ C(n,s,p) \Bigl(   \vartheta^{\frac{(p-1)\vartheta}{(\vartheta -1 )^2}} \bigl( 1 + R^{sp\nu +\frac{(p-2)sp\nu}{(p-1)(\vartheta-1)}}\norm{f}_{L^{q,r}(Q_{R,R^{sp}}(x_0,T_0))}^{1+\frac{1}{\vartheta-1}\frac{p-2}{p-1} } \bigr) \Bigr)^p
\Bigr]^{p-1} ,
\end{aligned}
\]
with $\vartheta$ and $\nu$ defined in Corollary \ref{thm:moser-bound}, here the estimate is only written in the case $sp \neq n$ for simplicity. Since $Q_{R,R^{sp}}(x_0,T_0)\subset Q_{ R_1,R_1^{sp}}(z,T_1)$ the above expression is less than
\[
\begin{aligned}
&\leq C(n,s,p)( \frac{R}{\rho})^{n+sp} R^\xi \norm{f}_{L^{q,r}(Q_{ R_1,R_1^{sp}}(z,T_1))}^{p^\prime} \\
& + C(n,s,p)(\frac{\rho}{R})^{\delta p} \Bigl[ 1+ \norm{u}_{L^\infty(Q_{ R_1,R_1^{sp}}(z,T_1))}^{p} + \sup_{T_0-R^{sp}<t\leq T_0} \mathrm{Tail}_{p-1,sp}(u(\frarg,t);x_0,R)^{p} \\
&+\Bigl(  \vartheta^{\frac{(p-1)\vartheta}{(\vartheta -1 )^2}} \bigl( 1 + d^{sp\nu +\frac{(p-2)sp\nu}{(p-1)(\vartheta-1)}}\norm{f}_{L^{q,r}(Q_{ R_1,R_1^{sp}}(z,T_1))}^{1+\frac{1}{\vartheta-1}\frac{p-2}{p-1} } \bigr) \Bigr)^p   \Bigr]^{p-1}.
\end{aligned}
\]
Concerning the tail term , since $B_{R}(x_0) \subset B_{R_1}(z)$, using Lemma \ref{lm:tail-shrink} we have
\begin{equation}\label{eq:tail-expan}
\begin{aligned}
\mathrm{Tail}_{p-1,sp} (u(\frarg,t);x_0,R)^{p-1} &\leq \Bigl( \frac{R}{R_1} \Bigr)^{sp}\Bigl(  \frac{R_1}{R_1-\abs{x_0-z}} \Bigr)^{n +s\,p} \mathrm{Tail}_{p-1,sp}(u(\frarg,t);z,R_1)^{p-1} \\
&+ C  \norm{u(\frarg,t)}_{L^\infty(B_{R_1}(z))}^{p-1},
\end{aligned}
\end{equation}
and by the choice of the radii, we have 
$$\frac{R}{R_1} < \frac{d}{R_1} <1-\sigma \quad and \quad \frac{R_1}{R_1- \abs{x_0-z}} \leq \frac{R_1}{R_1-\sigma R_1} \leq \frac{1}{1-\sigma} .$$
Hence, taking the supremum in time and using Minkowski's inequality in \eqref{eq:tail-expan}, we arrive at
\[
\begin{aligned}
\sup_{T_0-R^{sp}<t\leq T_0}& \mathrm{Tail}_{p-1,sp}(u;x_0,R)^{p} \\ 
&\leq C\frac{1}{(1-\sigma)^n} \Bigl( \sup_{T_0-R^{sp}<t\leq T_0} \mathrm{Tail}_{p-1,sp}(u(\frarg,t);z,R_1)^{p}  + \norm{u}_{L^\infty([T_0-R^{sp},T_0]\times B_{R_1}(z)}^p \Bigr) \\
& \leq C\frac{1}{(1-\sigma)^n} \Bigl( \norm{u}_{L^\infty(Q_{R_1,R_1^{sp}}(z,T_1))}^p +\sup_{T_1-R_1^{sp}<t\leq T_1} \mathrm{Tail}_{p-1,sp}(u(\frarg,t);z,R_1)^{p} \Bigr),
\end{aligned}
\]
where the constant $C$ above depends on $n,  s, \text{and }p$.
In conclusion,
\[
\begin{aligned}
&\dashint_{Q_{\rho,\rho^{s\,p}}(x_0,T_0)} \abs{u- \bar{u}_{(x_0,T_0),\rho}}^p \dd x \dd t \leq C(n,s,p) \bigl(\frac{R}{\rho} \bigr)^{n+s\,p} R^\xi \norm{f}_{L^{q,r}(Q_{R_1,R_1^{s\,p}}(z,T_1))}^{p^\prime} \\
&\quad + C(n,s,p,\sigma) (\frac{\rho}{R})^{\delta p}\Bigl[ 1+ \norm{u}_{L^\infty(Q_{R_1,R_1^{s p}}(z,T_1))}^p +\sup_{T_1-R_1^{sp}<t\leq T_1} \mathrm{Tail}_{p-1,sp}(u(\frarg,t);z,R_1)^{p} \\
&\quad +\Bigl( \vartheta^{\frac{(p-1)\vartheta}{(\vartheta -1 )^2}} \bigl( 1 + d^{sp\nu +\frac{(p-2)sp\nu}{(p-1)(\vartheta-1)}}\norm{f}_{L^{q,r}(Q_{R_1,R_1^{sp}}(z,T_1))}^{1+\frac{1}{\vartheta-1}\frac{p-2}{p-1} } \bigr) \Bigr)^p  \Bigr]^{p-1} .
\end{aligned}
\]
Now we make the choice $\rho=\frac{R^\theta}{2}$ with 
$$\theta := 1+ \frac{\xi}{\delta p + n + sp}.$$
This yields
\[
\begin{aligned}
&\rho^{-\zeta p} \dashint_{Q_{\rho,\rho^{s\, p}}(x_0,T_0)} \abs{u - \bar{u}_{(x_0,T_0),\rho}}^p \dd x \dd t \leq C(n,s,p, \sigma)\Bigl[  \Bigl( 1+ \norm{u}_{L^\infty(Q_{ R_1,R_1^{sp}}(z,T_1))}^p  
\\
 &\qquad +\bigl(  \vartheta^{\frac{(p-1)\vartheta}{(\vartheta -1 )^2}} \bigl( 1 + d^{sp\nu +\frac{(p-2)sp\nu}{(p-1)(\vartheta-1)}}\norm{f}_{L^{q,r}(Q_{ R_1,R_1^{sp}}(z,T_1))}^{1+\frac{1}{\vartheta-1}\frac{p-2}{p-1} } \bigr) \bigr)^p
  \\
&\qquad+\sup_{T_1-R_1^{sp}<t\leq T_1} \mathrm{Tail}_{p-1,sp}(u(\frarg,t);z,R_1)^{p}   \Bigr)^{p-1}  + \norm{f}_{L^{q,r}(Q_{ R_1,R_1^{sp}}(z,T_1))}^{p^\prime} \Bigr]\, ,
\end{aligned}
\]
for any $0< \rho< \frac{\min{\lbrace 1,d \rbrace}^\theta}{2}$, where
$$\zeta =  \frac{\xi \delta}{n +sp + \delta p + \xi} .$$
For values of $\rho\geq \frac{\min{\lbrace1,d \rbrace}^\theta}{2}$ 
$$\rho^{-\zeta p} \dashint_{Q_{\rho,\rho^{s p}}(x_0,T_0) \cap Q_{R_1,R_1^{sp}}(z,T_1)} \abs{u - \bar{u}_{(x_0,T_0),\rho}}^p \dd x \dd t \leq 2^{(1+\zeta) p} \min{\lbrace1,d \rbrace}^{-\theta \zeta p} \norm{u}_{L^\infty(Q_{R_1,R_1^{sp}}(z,T_1))}^p .$$
We can then conclude that for any cylinder of arbitrary size we have
$$\rho^{-\zeta p} \dashint_{Q_{\rho,\rho^{s p}}(x_0,T_0) \cap Q_{R_1,R_1^{sp}}(z,T_1)} \abs{u - \bar{u}_{(x_0,T_0),\rho}}^p \dd x \dd t \leq C \; ,$$
with $C$ depending on 
$$n,\;s,\; p,\; R_1 ,\;  \sigma ,\; \sup_{T_1-R_1^{sp}<t\leq T_1} \mathrm{Tail}_{p-1,sp}(u(\frarg,t);z,R_1),\;  \norm{f}_{L^{q,r}(Q_{ R_1,R_1^{sp}}(z,T_1))} \text{, and }\norm{u}_{L^\infty(Q_{R_1,R_1^{s p}}(z,T_1))}.$$
Now we use the characterization of the Campanato spaces in $\R^{n+1}$ with a general metric in \cite{Gorka}, see also \cite{dapra}. Our setting does not fit directly in the context considered there, since we only work with cylinders that are one sided in the time direction, that is $(t-r^{sp},t]\times B_r(x)$ instead of $(t-r^{sp},t+r^{sp})\times B_r(x)$. Still, if you follow the proof in \cite{Gorka} with small modifications, you can also conclude the result in this setting.

In the case of $sp \geq 1$, using \cite[Theorem 3.2]{Gorka} we get the the H\"older continuity of $u$ with exponent $\zeta$ in $Q_{\sigma R , (\sigma R)^{sp}}$ with respect to the metric
$$d((x,\tau_1),(y,\tau_2))= \max{\lbrace\abs{x-y} , \abs{\tau_2 - \tau_1}^{\frac{1}{sp}}\rbrace},$$
for which the balls of radius $r$ are of the form $(t-r^{sp},t+r^{sp})\times B_r(x)$. which means 
$$\abs{u(x_1,t_1)-u(x_2,t_2)} \leq C \bigl( \abs{x_1-x_2} + \abs{t_1 - t_2}^{\frac{1}{sp}} \bigr)^\zeta
\leq  C \bigl(  (\abs{x_1-x_2}^\zeta ) + \abs{t_1-t_2}^{\frac{\zeta}{sp}} \Bigr) .$$
In the case of $sp< 1$ we use the metric 
$$d((x,\tau_1),(y,\tau_2))= \max{\lbrace\abs{x-y}^{sp} , \abs{\tau_2 - \tau_1} \rbrace } .$$
The balls of radius $r$ are of the form $(t-r,t+r)\times B_{r^{\frac{1}{sp}}}(x)$.Hence we have a decay of order $r^{\frac{\xi}{sp} p}$ of the average of $u$ on the half balls. \cite[Theorem 3.2]{Gorka} implies the following H\"older continuity on $Q_{\sigma R_1,(\sigma R_1)^{sp}}$
$$\abs{u(x_1,t_1)-u(x_2,t_2)} \leq C \bigl( \abs{x_1-x_2}^{sp} + \abs{t_1 - t_2} \bigr)^{\frac{\zeta}{sp}} \leq C \bigl(  (\abs{x_1-x_2}^{\zeta} + \abs{t_1-t_2}^{\frac{\zeta}{sp}} \Bigr).$$
\end{proof}
\begin{lemma}[Stability in $L^\infty$] \label{lm:stabil}
Let $f \in L^{q,r}_{\loc}(Q_{2R,(2R)^{sp}})$ with $r \geq p^\prime$, 
$$\frac{1}{r} + \frac{n}{spq} < 1 \quad\text{and } q\geq (p_s^\star)^\prime \quad \text{in the case } sp < n,$$ 
%$$\frac{1}{r} + \frac{n}{spq} < 1 \quad\textit{ and  } q > 1\quad \textit{in the case } sp=n,$$
and
$$\frac{1}{r} + \frac{1}{q} < 1 \quad\text{ and  } q > 1\quad \text{in the case } sp\geq n.$$
Let $u$ be a local weak solution to the equation
$$
u_t +(-\Delta_p)^s u = f \quad in \; Q_{2R,(2R)^{sp}} ,
$$
with
$$\norm{u}_{L^\infty(Q_{R,R^{sp}})} + \sup_{-R^{sp} < t \leq 0} \mathrm{Tail}_{p-1,sp}(u(\frarg,t);0,R) \leq M ,$$
and 
$$\norm{f}_{L^{q,r}(Q_{R,R^{sp}}(x_0,T_0))}\leq \omega .$$
%Where
% $$b > \frac{(2s+n)(1+\frac{sp}{n})}{sp(1+\frac{2s}{n})-1}$$
Consider the $(s,p)$-caloric replacement
\[
\begin{cases}
\phi_t + (-\Delta_p)^s \phi =0 \quad & in \;\; Q_{R,R^{s \, p}} \\
\phi=u \quad & in \;\; (\R^n \setminus B_R(x_0)) \times [-R^{s\,p},0] \\
\phi(x,-R^{s\,p}) = u(x,-R^{sp}) \quad & in \;\; B_R
\end{cases}
\]
Then for $\sigma< 1$ , there is a $\delta_{M,R,\sigma}(\omega)$ such that
$$\norm{u - \phi}_{L^\infty(Q_{\sigma R ,(\sigma R)^{s\,p}})(x_0,T_0)} < \delta_{M,R,\sigma}(\omega), $$
and $\delta_{M,R, \sigma}(\omega)$ converges to $0$ as $\omega$ goes to $0$.
\end{lemma}
\begin{proof}
The existence of such a bound follows immediately from Corollary \ref{thm:moser-bound}.

To show the convergence of $\tau_{M,R,\sigma}$ to zero we argue by contradiction,
suppose that there is a sequence $f_n \in L^{q,r}(Q_{R,R^{s \, p}})$ and $u_n$ such that 
\[
\norm{u_n}_{L^\infty(Q_{R,R^{s \, p}})} + \sup_{T_0 - R^{s \, p} \leq t \leq T_0}  \mathrm{Tail}_{p-1,sp}(u_n(\frarg,t);0,R)  \leq M \quad \text{and} \quad \norm{f_n}_{L^{q,r}(Q_{R,R^{s \, p}})} \to 0 ,
\] 
but
\begin{equation}\label{eq:l-inft-contradic}
 \norm{u_n - \phi_n}_{L^\infty(Q_{\sigma R, (\sigma R)^{s \, p}})} > \epsilon >0.
\end{equation}

Using \eqref{eq:sob-est} from  Lemma \ref{lm:lpavr}, we have
\begin{equation}\label{eq:sobolevLp}
\lim_{n \to \infty} \int_{T_0 - R^{sp}}^{T_0} [u_n-\phi_n]_{W^{s,p}(\R^n)}^p \dd t \leq C(n,s,p,q,r,R) \lim_{n \to \infty} \norm{f_n}_{L^{q,r}(Q_{R,R^{s \, p}})}^{p^\prime} =0.
\end{equation}
By assumption, $u_n$ is uniformly bounded in $L^\infty(Q_{ R, R^{sp}})$. Now we show that $\phi_n$ is also uniformly bounded in $L^\infty(Q_{R, R^{sp}})$.
\begin{equation}\label{eq:phi_n-linfty}
\begin{aligned}
\norm{\phi_n}_{L^\infty(Q_{ R , R^{s\,p}}(x_0,T_0))}  &\leq \norm{u_n}_{L^\infty(Q_{ R ,(\sigma R)^{s\,p}})(x_0,T_0)} + \norm{u_n - \phi_n}_{L^\infty(Q_{R , R^{s\,p}})(x_0,T_0)} \\
 &\leq  M+ \norm{u_n-\phi_n}_{L^\infty(Q_{R,R^{s \, p}})}.
 \end{aligned}
\end{equation}
By Corollary \ref{thm:moser-bound}
\begin{equation}\label{eq:linfty-difference}
\begin{aligned}
\norm{u_n-\phi_n}_{L^\infty(Q_{R,R^{s \, p}}(x_0,T_0))} \leq  C(n,s,p) \vartheta^{\frac{(p-1)\vartheta}{(\vartheta -1 )^2}} \bigl( 1 + R^{sp\nu +\frac{(p-2)sp\nu}{(p-1)(\vartheta-1)}}\norm{f_n}_{L^{q,r}(Q_{R,R^{sp}})}^{1+\frac{1}{\vartheta-1}\frac{p-2}{p-1} } \bigr).
\end{aligned}
\end{equation}
Since $\norm{f_n}_{L^{q,r}(Q_{R,R^{s \, p}})}^{p^\prime}$ is uniformy bounded, \eqref{eq:linfty-difference} and \eqref{eq:phi_n-linfty} gives us a uniform bound on $ \norm{\phi_n}_{L^\infty(Q_{ R , R^{s\,p}}(x_0,T_0))} $.

Now we are in a position to use Theorem \ref{thm:basic-holder} for both of the sequences $u_n$ and $\phi_n$, which gives us a uniform bound on the H\"older seminorms of $u_n$ and $\phi_n$ in $Q_{\sigma R,(\sigma R)^{s p}}$. Therefore, by Arzela-Ascoli's Theorem $u_n - \phi_n$ has a uniformly convergent subsequence in $Q_{\sigma R,(\sigma R)^{sp}}$. By \eqref{eq:sobolevLp} the limit is $0$, contradicting \eqref{eq:l-inft-contradic}.
\end{proof}

\section{Improved H\"older regularity for nonhomogeneous equation}\label{sec4}
% First a notation and a comment about the scaling
\begin{proposition}\label{prop:reg-scal1}
Let $f\in L^{q,r}(Q_{1,2})$ with $q,r$ satisfying
$r \geq p^\prime$, 
$$\frac{1}{r} + \frac{n}{spq} < 1 \quad\text{and } q\geq (p_s^\star)^\prime \quad \text{in the case } sp < n,$$ 
and
$$\frac{1}{r} + \frac{1}{q} < 1 \quad\text{ and  } q > 1\quad \text{in the case } sp\geq n.$$
Assume $u$ is a weak solution of $u_t + (-\Delta_p)^s u=f$ in $Q_{1,2}$ that satisfies 
$$\norm{u}_{L^\infty(Q_{1,2})} \leq 1 \; , \quad \sup_{-2\leq t \leq 0}\mathrm{Tail}_{p-1,sp}(u;0,1) \leq 1.$$
Then there exists $\omega$ such that if
$$\norm{f}_{L^{q,r}(Q_{1,2})} \leq \omega(n,s,p,q,r,\alpha), $$
$u$ is locally H\"older continuous in $Q_{\frac{1}{2},\frac{1}{2^{sp}}}$ with exponents $\alpha$ in space and $\frac{\alpha}{sp-(p-2)\alpha}$ in time, as long as 
\begin{equation}\label{eq:space_exponent}
 \alpha < \min{ \left\lbrace  \Theta , \frac{r(spq-n)-spq}{q(r(p-1)-(p-2))} \right\rbrace}.
\end{equation}
Recall that $\Theta= \min{ \left\lbrace \frac{sp}{p-1},1 \right\rbrace}$.

More precisely for $(x_1,t_1),\; (x_2,t_2) \in Q_{\frac{1}{2},\frac{1}{2^{sp}}}$ we have
$$\abs{u(x_2,t_2) - u(x_1,t_1)}\leq C(n,s,p,q,r,\alpha) \left(\abs{x_2 - x_1}^\alpha + \abs{t_2 - t_1}^{\frac{\alpha}{sp-(p-2)\alpha}} \right).$$
\end{proposition}
\begin{proof}
\textbf{Step 1}: Decay at the origin. 

For this part, we prove a decay at the origin for $u$ under the assumptions
\begin{equation}\label{eq:step1-assump}
\norm{u}_{L^\infty(Q_{1,1})} \leq 1 \; , \quad \sup_{-1\leq t \leq 0}\mathrm{Tail}_{p-1,sp}(u;0,1) \leq 1 , \quad \text{and} \;\; \norm{f}_{L^{q,r}(Q_{1,1})} \leq \omega.
\end{equation}
We introduce the parabolic cylinder
$$G_r:=B_r(0)\times(-r^\beta,0],$$
with $\beta=sp-(p-2)\alpha$.
We show that for any exponent $\alpha$ satisfying \eqref{eq:space_exponent}, the following holds for $r< 1$
$$\norm{u(x,t)-u(0,0)}_{L^\infty(G_{r})} \leq C r^{\alpha }. $$
%we will need to have $\abs{c_{k+1}-c_k} \leq C \lambda^{\alpha k}$
It is enough to prove the inequality for a sequence of $r= \lambda^k ,\; (k)_0^\infty$, for some $\lambda < 1$. Without loss of generality, we assume $u(0,0)=0$. Consider the rescaled functions
$$v_k(x,t) := \frac{u(\lambda^k x , \lambda^{k\beta}t)}{\lambda^{\alpha k}},$$
with $\lambda$ small enough to be determined later.
We will prove the following by induction,
\begin{equation}\label{eq:scaled-bound}
\norm{v_k(x,t)}_{L^\infty(G_1)} \leq 1 \qquad \text{and } \qquad \sup_{-1\leq t \leq 0} \int_{\R^n \setminus B_1} \frac{\abs{v_k(x,t)}}{\abs{x}^{n+ s\, p}} \dd x\leq 1.
\end{equation}
For $k=0$, \eqref{eq:scaled-bound} follows from our assumptions \eqref{eq:step1-assump}.
 
Observe that 
\[
\begin{cases}
\frac{\partial v_k(x,t)}{\partial t} =\lambda^{\beta k - \alpha k} u_t(\lambda^k x, \lambda^{\beta k}t) &  \\
(-\Delta_p)^s v_k(x,t) = \lambda^{k[sp-(p-1)\alpha]}(-\Delta_p)^s u(\lambda^k x , \lambda^{\beta k}t)  &  \\
\end{cases}
\]
With $\beta = sp - (p-2)\alpha$, $v_k(x,t)$ solves
$$\frac{\partial v_k}{\partial t} + (-\Delta_p)^s v_k= \lambda^{k[sp- (p-1)\alpha]}f(\lambda^k x, \lambda^{\beta k}t) =: f_k(x,t) \qquad \text{in} \; Q_{\frac{1}{\lambda^k}, \frac{1}{\lambda^{\beta k}}}.$$

Moreover,
\[
\begin{aligned}
\norm{f_{k}}_{L^{q,r}(G_1)}^r &= \int_{-1}^0 \Bigl( \int_{B_1}\abs{f_k(x,r)}^q \dd x \Bigr)^{\frac{r}{q}} \dd t \\
& = \int_{-1}^0 \Bigl( \int_{B_{\lambda^k}} \lambda^{kq[sp - (p-1) \alpha]-kn}\abs{f(x,\lambda^{\beta k}t)}^q \dd x \Bigr)^\frac{r}{q} \dd t \\
&=\int_{-1}^0 \lambda^{rk[sp-(p-1)\alpha] -\frac{krn}{q}} \Bigl( \int_{B_{\lambda^k}} \abs{f(x,\lambda^{\beta k} t)}^q \dd x \Bigr)^\frac{r}{q} \dd t \\
&= \lambda^{rk[sp-(p-1)\alpha] -\frac{krn}{q} - \beta k} \norm{f}_{L^{q,r}(G_{\lambda^k})}.
\end{aligned}
\]
Since $\lambda<1$, and the exponent of $\lambda$ is positive by \eqref{eq:space_exponent}, we get $\norm{f_{k}}_{L^{q,r}(G_1)} \leq \omega$.

Assume that \eqref{eq:scaled-bound} holds for $k$. Now we prove that it holds for $k+1$. Consider the $(s,p)$-caloric replacement of $v_k(x,t)$ in $Q_{1,1}$, say $\phi_k(x,t)$. Then
$$\abs{v_k(x,t)} \leq \abs{v_k(x,t) - \phi_k(x,t)} + \abs{\phi_k(x,t) -\phi_k(0,t)} + \abs{\phi_k(0,t)-v_k(0,t)}.$$
By Theorem \ref{thm:thm1tail}, $\phi_k$ is locally H\"older continuous in $Q_{1,1}$ and for $(x,t) \in Q_{\frac{1}{2} , \frac{1}{2^{sp}}}$
$$\abs{\phi_k(x,t) - \phi_k(0,0)} \leq C_1 \abs{x}^{\Theta - \epsilon} + C_2 \abs{t}^{\Gamma - \frac{\epsilon}{\beta}}.$$
Here we take $\epsilon= \frac{\Theta-\alpha}{2}$. Since $\norm{f_k}_{L^{q,r}(Q_{1,1})} \leq \omega$, Lemma \ref{lm:stabil} implies
\begin{equation}\label{eq:resc_ hol}
 \abs{v_k(x,t)} \leq 2\delta(\omega) + C_1 \abs{x}^{\Theta - \epsilon} + C_2 \abs{t}^{\Gamma - \frac{\epsilon}{\beta}}, \qquad \text{in} \quad Q_{\frac{1}{4} , \frac{1}{4^{sp}}} .
\end{equation}
In Theorem \ref{thm:thm1tail} the H\"older constants are bounded by
\[
\begin{aligned}
(C_2)^{\frac{1}{p-1}} \leq  C_1 &\leq C\bigl(1+ \norm{\phi_k}_{L^\infty(Q_{1,1})} + \sup_{-\frac{1}{2^{sp}}\leq t \leq 0} \mathrm{Tail}_{p-1,sp}(\phi_k;0,1)\bigr) \\
& \leq C \bigl( 1+ \norm{\phi_k}_{L^\infty(Q_{1,1})} + \sup_{-1 \leq t \leq 0} \mathrm{Tail}_{p-1,sp}(v_k;0,1) \bigr)\\
&\leq C \bigl( 1+ \norm{v_k - \phi_k}_{L^\infty(Q_{1,1})} + \norm{v_k}_{L^\infty(Q_{1,1})} + \sup_{-1 \leq t \leq 0} \mathrm{Tail}_{p-1,sp}(v_k;0,1) \bigr).
\end{aligned}
\]
Therefore, using \eqref{eq:scaled-bound}, for $v_k$ we have
$$(C_2)^{\frac{1}{p-1}} \leq  C_1 \leq C(n,s,p,\alpha)( 3 +  \norm{v_k - \phi_k}_{L^\infty(Q_{1,1})}).$$
By Corollary \ref{thm:moser-bound}
$$C_1 \leq C( 3 + C(n,s,p,q,r) (1 + \norm{f_k}_{L^{q,r}(Q_{1,1})}^{1+\frac{1}{\vartheta-1}\frac{p-2}{p-1} }) )\leq C( 3 + C(n,s,p,q,r)(1+ \omega^{1+\frac{1}{\vartheta-1}\frac{p-2}{p-1} })).$$
This is a bound independent of $k$. We can take $\omega$ to be less than $1$ and take $C_1=C(n,s,p)( 3 + 2C(n,s,p,q,r),$ with the $C(n,s,p,q,r)$ coming from Corollary \ref{thm:moser-bound}, so that the constants $C_1,\; C_2$ are independent of $\omega$ as well.

Now we proceed and prove \eqref{eq:scaled-bound} for $k+1$. First, we state our choice of $\lambda$
\begin{equation}\label{eq:lambda-smallness}
\lambda \leq  \min{ \Bigl\lbrace \frac{1}{4} , \frac{1}{4^{\frac{sp}{\beta}}}, \frac{1}{(2 C_1+ 2C_2)^{\frac{2}{\Theta-\alpha}}} , \Bigl(1+ \frac{\omega_n(4^{sp} - 1)}{sp}+ \frac{(1+C_1 +C_2)^{p-1}}{(p-1)(\Theta-\alpha)/2} \Bigr)^{\frac{2}{(p-1)(\Theta - \alpha)}}  \Bigr\rbrace }.
\end{equation}
Since $\lambda < \frac{1}{4}$, and $\lambda^\beta < \frac{1}{4^{s\, p}}$, $Q_{\lambda , \lambda^{\beta}} \subset Q_{\frac{1}{4},\frac{1}{4^{s \, p}}}$. Therefore, from \eqref{eq:resc_ hol} we obtain 
$$ \norm{v_k(x,t)}_{L^\infty (G_ \lambda)} \leq \delta(\omega) + C_1 \lambda ^{\Theta - \epsilon } + C_2 \lambda^{\beta(\Gamma - \frac{\epsilon}{\beta})}  .
$$
Notice that $\beta \Gamma \geq \Theta$, by the above choice of $\beta$. Thus,
\begin{equation}\label{eq:vk-decay-delta}
\norm{v_k(x,t)}_{L^\infty(G_\lambda)} \leq \delta(\omega) + (C_1 + C_2) \lambda^{\Theta - \epsilon}.
\end{equation}
Recall that $\epsilon= \frac{\Theta -\alpha}{2}$ and by the assumption \eqref{eq:lambda-smallness}
 $$(C_1 +C_2) \lambda^{\Theta - \epsilon} < \frac{1}{2} \lambda^\alpha .$$
Now we choose $\omega$ so that
$$2\delta(\omega) \leq \frac{1}{2} \lambda^{\Theta } \leq \frac{1}{2} \lambda^\alpha .$$
This is possible since $\delta(\omega)$ converges to zero as $\omega \to 0$. Then, \eqref{eq:vk-decay-delta} implies
$$ \norm{v_k(x,t)}_{L^\infty(G_\lambda)} \leq \lambda ^\alpha ,$$ 
which translates to 
\begin{equation}\label{eq:induction-first}
 \norm{v_{k+1}(x,t)}_{L^\infty(G_1)} = \normB{\frac{v_k(\lambda x, \lambda^\beta t)}{\lambda^\alpha}}_{L^\infty(G_1)} \leq 1,
\end{equation}
which is the first part of \eqref{eq:scaled-bound}. For the second part, we want to show
$$\sup_{-1<t<0} \int_{\R^n \setminus B_1} \frac{\abs{v_{k+1}(x ,t)}^{p-1}}{\abs{x}^{n+sp}} \dd x \leq 1. $$
We split the integral into three parts. Using the induction hypothesis
\[
\begin{aligned}
\sup_{-1<t<0} \int_{\R^n \setminus B_{\frac{1}{\lambda}}} \frac{\abs{v_{k+1}(x ,t)}^{p-1}}{\abs{x}^{n+sp}} \dd x &\leq \sup_{-\lambda^{-\beta} \leq t\leq 0} \int_{\R^n \setminus B_{\frac{1}{\lambda}}} \frac{\abs{v_{k+1}(x ,t)}^{p-1}}{\abs{x}^{n+sp}} \dd x \\
&= \lambda^{sp -\alpha(p-1)} \sup_{-1<t<0} \int_{\R^n \setminus B_1} \frac{\abs{v_{k}(x ,t)}^{p-1}}{\abs{x}^{n+sp}} \dd x \\
(\text{using  } \Theta \leq \frac{sp}{p-1})\quad & \leq \lambda^{(p-1)(\Theta -\alpha)}.
\end{aligned}
\]
Moreover, $\norm{v_k}_{L^\infty(G_1)} \leq 1$ and hence
\[
\begin{aligned}
\sup_{-1<t<0} \int_{B_{\frac{1}{\lambda}} \setminus B_{\frac{1}{4\lambda}}} \frac{\abs{v_{k+1}(x ,t)}^{p-1}}{\abs{x}^{n+sp}} \dd x &\leq 
\lambda^{sp-\alpha(p-1)} \sup_{-\lambda^\beta< t<0} \int_{B_{1} \setminus B_{\frac{1}{4}}} \frac{\abs{v_{k}(x ,t)}^{p-1}}{\abs{x}^{n+sp}} \dd x \\
 & \leq \lambda^{sp-\alpha(p-1)} \int_{B_{1} \setminus B_{\frac{1}{4}}} \frac{1}{\abs{x}^{n+sp}}\dd x \\
 &\leq \lambda^{(p-1)(\Theta - \alpha)} \frac{\omega_n(4^{sp} - 1)}{sp} := C_3 \lambda^{2(p-1)\epsilon}.
\end{aligned}
\]
For remaining part, we transfer the estimate \eqref{eq:resc_ hol} to $v_{k+1}$ and obtain
$$\abs{v_{k+1}(x,t)} \leq \delta(\omega) \lambda^{-\alpha} + C_1 \lambda^{\Theta -\epsilon -\alpha} \abs{x}^{\Theta -\epsilon} 
+ C_2 \lambda^{\beta\Gamma - \epsilon-\alpha} \abs{t}^{\Gamma -\frac{\epsilon}{\beta}} \qquad \text{in} \quad Q_{\frac{1}{4\lambda}, \frac{1}{4^{sp}\lambda^\beta}}.$$
In particular, since $\lambda^\beta \leq \frac{1}{4^{sp}}$ , $Q_{\frac{1}{4\lambda},1} \subset Q_{\frac{1}{4\lambda}, \frac{1}{4^{sp}\lambda^\beta}}$, and $\delta(\omega) \leq \lambda^{\Theta} \leq  \lambda^{\Theta -\epsilon}$ we get
$$\sup_{-1\leq t\leq 0} \abs{v(x,t)} \leq \lambda^{\Theta - \epsilon - \alpha} (1+ C_2\lambda^{\beta \Gamma - \Theta} + C_1 \abs{x}^{\Theta -\epsilon}).$$
Therefore,
\[
\begin{aligned}
\sup_{-1<t<0} \int_{B_{\frac{1}{4\lambda}} \setminus B_1} &\frac{\abs{v_{k+1}(x ,t)}^{p-1}}{\abs{x}^{n+sp}} \dd x \leq \lambda^{(p-1)(\Theta - \epsilon - \alpha)} \int_{B_{\frac{1}{4\lambda}} \setminus B_1} \frac{\abs{1+ C_2\lambda^{\beta \Gamma - \Theta} + C_1 \abs{x}^{\Theta -\epsilon}}^{p-1}}{\abs{x}^{n+sp}} \dd x \\
(\text{using}\;\; \abs{x}\geq 1)\quad \qquad & \leq (1+C_2\lambda^{\beta \Gamma - \Theta} + C_1 )^{p-1} \lambda^{(p-1)(\Theta - \epsilon - \alpha)} \int_ {B_{\frac{1}{4\lambda}} \setminus B_1} \frac{1}{\abs{x}^{n+sp-(p-1)(\Theta -\epsilon)}} \dd x \\
(\text{using} \;\; sp\geq (p-1)\Theta) \quad \qquad& \leq (1+ C_2 + C_1)^{p-1} \lambda^{(p-1)(\Theta - \epsilon - \alpha)} \int_{\R^n \setminus B_1} \frac{1}{\abs{x}^{n+\epsilon(p-1)}} \dd x  \\
&\leq \frac{(1+C_1+C_2)^{p-1}}{\epsilon(p-1)} \lambda^{(p-1)(\Theta - \epsilon - \alpha)}:= C_4 \lambda^{(p-1)\epsilon} .
\end{aligned}
\]
Hence,
$$\sup_{-1<t<0} \int_{\R^n \setminus B_1} \frac{\abs{v_{k+1}(x ,t)}^{p-1}}{\abs{x}^{n+sp}} \dd x \leq \lambda^{2(p-1)\epsilon} + C_3\lambda^{2(p-1)\epsilon} + C_4 \lambda^{(p-1)\epsilon} \leq \lambda^{(p-1)\epsilon}(1+C_3 + \frac{(1+C_1+C_2)^{p-1}}{\epsilon(p-1)}). $$
%The smallness that we require for $\lambda$ is summarized here:
%$$\lambda \leq \min{\lbrace (1+C_3 + \frac{C_4}{\epsilon(p-1)})^{-\frac{1}{\epsilon(p-1)}} , 4^{\frac{-sp}{\beta}} , \frac{1}{4} ,(2(C_1+C_2))^{-\frac{2}{\epsilon}}  \rbrace} \qquad \textit{and} \quad \delta(\omega) \leq \frac{1}{2} \lambda^\Theta$$
Using the assumption \eqref{eq:lambda-smallness} on $\lambda$, we obtain 
$$\sup_{-1<t<0} \int_{\R^n \setminus B_1} \frac{\abs{v_{k+1}(x ,t)}^{p-1}}{\abs{x}^{n+sp}} \dd x \leq 1.$$
\textbf{Step 2}: Regularity in a cylinder.
We choose $\alpha$ as in \eqref{eq:space_exponent} and let $\omega$ be as in Step 1. For a point $(x_0,t_0) \in Q_{\frac{1}{2} , \frac{1}{2^{sp}}}$ define
$$\tilde{u}(x,t) = \frac{1}{L} u(\frac{x}{2} + x_0 , L^{2-p} \frac{1}{2^{sp}}t + t_0 ),$$
where $L = 2^\frac{n}{p-1}(1+\abs{B_1})^{\frac{1}{p-1}}$. Then $\tilde{u}$ is a solution of 
$$\partial_t \tilde{u} + (- \Delta_p)^s \tilde{u} = \frac{L^{-(p-1)}}{2^{sp}} f \Bigl(\frac{x}{2} +x_0 , L^{2-p} \frac{1}{2^{sp}}t +t_0 \Bigr):= \tilde{f} \quad \text{in} \; Q_{1,2^{sp-1}L^{p-2}} . $$
By the choice of $L$, $\tilde{u}$ satisfies the conditions \eqref{eq:step1-assump} in Step 1. Since $L\geq 1$ we immediately have
$$\norm{\tilde{u}}_{L^\infty(Q_{1,1}(0,0))} \leq \frac{1}{L}\norm{u}_{ L^\infty(Q_{\frac{1}{2}, \frac{L^{2-p}}{2^{sp}}}(x_0,t_0))} \leq \norm{u}_{L^\infty(Q_{1,2})} \leq 1,$$
since  $Q_{\frac{1}{2},\frac{L^{2-p}}{2^{sp}}}(x_0,t_0) \subset Q_{1,2}$. As for the $L^{q,r}$ norm of $\tilde{f}$ we have
\[
\begin{aligned}
\qquad\norm{\tilde{f}}_{L^{q,r}(Q_{1,1})} &= \frac{L^{-(p-1)}}{2^{sp}} \bigl( 2^{\frac{n}{q}+\frac{sp}{r}} L^\frac{p-2}{r} \norm{f}_{L^{q,r}(Q_{\frac{1}{2},\frac{L^{2-p}}{2^{sp}}}(x_0,t_0))}\bigr) \\
& \leq L^{-(p-1)} (1/2)^{sp(1-\frac{1}{r} - \frac{n}{spq})} \norm{f}_{Q_{1,2}} \\
&\leq L^{-(p-1)} (1/2)^{sp(1-\frac{1}{r} - \frac{n}{spq})} \omega
\leq \omega.
\end{aligned}
\]
Here we have used $1-\frac{1}{r}-\frac{n}{spq} > 0$. Notice that in the case of $sp \geq n$, we are assuming $1-\frac{1}{r} -\frac{1}{q} > 0$ which is a stronger assumption.
Now we verify the assumption on the tail.
\[
\begin{aligned}
\sup_{-1\leq t \leq 0} \int_{\R^n \setminus B_1} \frac{\abs{\tilde{u}}^{p-1}}{\abs{x}^{n+sp}} \dd x &= \frac{2^{-sp}}{L^{p-1}} \sup_{t_0 - \frac{L^{2-p}}{2^{sp}}\leq t \leq t_0} \int_{\R^n \setminus B_{\frac{1}{2}(x_0)}} \frac{\abs{u(y)}^{p-1}}{\abs{y-x_0}^{n+sp}} \dd y \\
&\leq \frac{1}{L^{p-1}} \sup_{-2\leq t \leq 0} \mathrm{Tail}_{p-1,sp}(u(\frarg ,t);x_0, \frac{1}{2})^{p-1} \\
& \leq \frac{1}{L^{p-1}} (\frac{1}{2})^{sp}(\frac{1}{1-\abs{x-x_0}})^{n+sp} \sup_{-2\leq t\leq 0}\mathrm{Tail}_{p-1,sp}(u(\frarg,t);0,1)^{p-1} \\
& \quad + \frac{2^n}{L^{p-1}} \sup_{-2\leq t\leq 0} \norm{u(\frarg,t)}_{L^{p-1}(B_1(0))}^{p-1} \\
& \leq \frac{2^n}{L^{p-1}} \bigl( 1+ \abs{B_1} \norm{u}_{L^\infty(Q_{1,2})} \bigr) \leq \frac{2^n(1+\abs{B_1})}{L^{p-1}} \leq 1 .
\end{aligned}
\]
Now we can apply Step 1 to $\tilde{u}$ and we get the decay
$$\norm{\tilde{u}-\tilde{u}(0,0)}_{L^\infty(G_r)} \leq Cr^\alpha , \quad \text{for} \;\; 0<r<1$$
or in other words
$$\abs{\tilde{u}(x,t)-\tilde{u}(0,0)} \leq C (\abs{x}^\alpha + \abs{t}^{\frac{\alpha}{\beta}}), \quad \text{for} \;\; (x,t) \in Q_{1,1} .$$
In terms of $u$, this means 
\begin{equation}\label{eq:decay-poin}
\abs{u(x,t)-u(x_0,t_0)}\leq CL (2^\alpha \abs{x-x_0}^{\alpha} + (2^{sp}L^{p-2})^{\frac{\alpha}{\beta}} \abs{t-t_0}^{\frac{\alpha}{\beta}}) , \quad \text{for} \;\; (x,t) \in Q_{\frac{1}{2},\frac{1}{2^{sp} L^{p-2}}}(x_0,t_0) .
\end{equation}
Now take two points $(x_1,t_1)\, , \, (x_2,t_2) \in Q_{\frac{1}{2}, \frac{1}{2^{sp}}}$ and split the line joining them into $1+ [L^{p-2}]$ pieces,  say  
$(y_i,\tau_i)_{i=0}^{1+[L^{p-2}]}$ with $(x_1,t_1)=(y_0,\tau_0)$, $(x_2,t_2)=(y_{1+[L^{p-2}]} ,\tau_{1+[L^{p-2}]})$ , $\abs{y_{i+1}-y_i}= \frac{\abs{x_2-x_1}}{1+[L^{p-2}]} < \frac{1}{2}$ and $\abs{\tau_{i+1}- \tau_i}= \frac{\abs{t_2-t_1}}{1+[L^{p-2}]} < \frac{1}{2^{sp}L^{p-2}}$ so that $(y_{i+1},\tau_{i+1}) \in Q_{\frac{1}{2} , \frac{1}{2^{sp}L^{p-2}}}(y_i,\tau_i)$.
 By \eqref{eq:decay-poin} applied in each of $Q_{\frac{1}{2} , \frac{1}{2^{sp}L^{p-2}}}(y_i,\tau_i)$ obtain
 \[
 \begin{aligned}
 \abs{u(x_2,t_2) - u(x_1,t_1) } & \leq \sum_{i=0}^{[L^{p-2}]} \abs{u(y_{i+1},\tau_{i+1})-u(y_i,\tau_i)} \\
  &\leq CL\sum_{i=0}^{[L^{p-2}]} 2^\alpha \abs{y_{i+1} -y_i}^\alpha + (2^{sp} L^{p-2})^{\frac{\alpha}{\beta}} \abs{\tau_{i+1}-\tau_i}^{\frac{\alpha}{\beta}} \\ &\leq C(1+L)^{p-1} \Bigl( (2\frac{\abs{x_2-x_1}}{1+[L^{p-2}]})^\alpha
  + (2^{sp}L^{p-2} \frac{\abs{t_2-t_1}}{1+\floor{L^{p-2}}})^{\frac{\alpha}{\beta}} \Bigr)\\
  & \leq C(n,s,p,q,r,\alpha) (\abs{x_2 - x_1}^\alpha + \abs{t_2 - t_1}^{\frac{\alpha}{\beta}} ).
 \end{aligned}
 \]
\end{proof}
Now we prove the H\"older regularity at any scale.
\begin{proof}[Proof of Theorem \ref{thm:main-holder}]
We will consider the rescaled functions
$$\tilde{u}_{\iota}(x,t) = \frac{1}{\mu} u(Rx + x_0, \mu^{2-p}R^{sp} t + \iota + T_0)$$
with 
\[
\begin{aligned}
\mu =& 1+ \norm{u}_{L^\infty(Q_{R,2R^{sp}}(x_0,T_0))} + \sup_{ T_0- 2R^{sp}\leq t \leq T_0}\mathrm{Tail}_{p-1,sp} (u(\frarg ,t);x_0,R) \\
&\quad+ \left(  \frac{R^{sp-\frac{n}{q} - \frac{sp}{r}} \norm{f}_{L^{q,r}(Q_{R,2R^{sp}}(x_0,T_0))}}{\omega} \right)^{\frac{1}{p-1+ \frac{p-2}{r}}},
\end{aligned}
\]
where $\omega= \omega(n,s,p,q,r,\alpha)$ is the same as in the proof of Proposition \ref{prop:reg-scal1} and $\iota \in [ -(R/2)^{sp} (1-\mu^{2-p}),0]$. The interval $[ -(R/2)^{sp} (1-\mu^{2-p}),0]$ is chosen so that the cylinders $Q_{\frac{R}{2},\frac{\mu^{2-p} R^{sp}}{2^{sp}}}(x_0,T_0+\iota)$ cover all of $Q_{\frac{R}{2},(\frac{R}{2})^{sp}}(x_0,T_0)$ by varying $\iota$ over. Note that for these choices of $\iota$ we have $Q_{R,2\mu^{2-p}R^{sp}}(x_0,T_0+\iota)\subset Q_{R,2R^{sp}}(x_0,T_0)$. Then $\tilde{u}$ is a solution of
$$\partial_t \tilde{u}_\iota + (-\Delta_p^s)\tilde{u}_\iota = R^{sp}\frac{f(Rx +x_0,\mu^{2-p} R^{sp} t + \iota + T_0)}{\mu^{p-1}}, \quad \text{in} \quad Q_{1,2} . $$
We now verify that $\tilde{u}_\iota$ satisfies the conditions of  Proposition \ref{prop:reg-scal1}. The $L^{q,r}$ norm of the right hand side is
 \[
 \begin{aligned}
 \normB{R^{sp} \frac{f(Rx,\mu^{2-p} R^{sp} t + \iota)}{\mu^{p-1}}}_{L^{q,r}(Q_{1,2})} & = \frac{\mu^{\frac{p-2}{r}}}{2^{\frac{1}{r}}\mu^{(p-1)}} R^{sp-\frac{n}{q}- \frac{sp}{r}} \norm{f}_{L^{q,r}(Q_{R,2\mu^{2-p}R^{sp}}(x_0,T_0+ \iota))} \\
 & \leq \frac{R^{sp-\frac{n}{q} - \frac{sp}{r}} \norm{f}_{L^{q,r}(Q_{R,2R^{sp}}(x_0,T_0))}}{2^{\frac{1}{r}}\mu^{p-1 -\frac{p-2}{r}}} \\
 &\leq \frac{\omega}{2^{\frac{1}{r}}} < \omega .
 \end{aligned}
 \]
 The $L^\infty$ norm of $\tilde{u}_\iota$ satisfies
 $$\norm{\tilde{u}_\iota}_{L^\infty(Q_{1,2}(0,0))} = \frac{1}{\mu} \norm{u}_{L^\infty(Q_{R,2\mu^{2-p}R^{sp}(x_0,T_0+\iota)})} \leq \frac{1}{\mu} \norm{u}_{L^\infty(Q_{R,2R^{sp}})} \leq 1.$$
Similarly 
\[
\begin{aligned}
\sup_{ -2 \leq t \leq 0}\mathrm{Tail}_{p-1,sp} (\tilde{u}(\frarg ,t);0 ,1)& \leq\frac{1}{\mu} \sup_{ T_0+\iota- 2\mu^{2-p} R^{sp}\leq t \leq T_0+\iota}\mathrm{Tail}_{p-1,sp} (u(\frarg ,t);x_0,R) \\
&\leq \frac{1}{\mu }\sup_{ T_0- 2R^{sp}\leq t \leq T_0}\mathrm{Tail}_{p-1,sp} (u(\frarg ,t);x_0,R) \leq 1 .
\end{aligned}
\]
Hence, using Proposition \ref{prop:reg-scal1} for $\tilde{u}_\iota$, we get
$$\abs{\tilde{u}_\iota (\tilde{x}_2,\tilde{t}_2) - \tilde{u}_\iota(\tilde{x}_1,\tilde{t}_1)}\leq C(\abs{\tilde{x}_2 - \tilde{x}_1}^\alpha + \abs{\tilde{t}_2 - \tilde{t}_1}^{\frac{\alpha}{sp-(p-2)\alpha}})  \quad \text{ for } (\tilde{x}_1,\tilde{t}_1) ,\; (\tilde{x}_2,\tilde{t}_2) \in Q_{\frac{1}{2},\frac{1}{2^{s \, p}}}(0,0),$$

with $C= C(n,s,p,q,r,\alpha) $. This translates to 
\begin{equation}\label{eq:holder-final-coveri}
\abs{u(x_2,\tau_2) - u(x_1,\tau_1)} \leq \mu C \Bigl[ \Bigl(\frac{\abs{x_2-x_1}}{R} \Bigr)^\alpha + \Bigl(\frac{\abs{\tau_2-\tau_1}}{R^{s\, p} \mu^{2-p}} \Bigr)^{\frac{\alpha}{sp-(p-2)\alpha}} \Bigr],
\end{equation}
for $(x_1,\tau_1), \; (x_2,\tau_2) \in Q_{\frac{R}{2},  \frac{R^{s \, p}\mu^{2-p}}{2^{s \, p}}}(x_0,T_0+\iota)$. Now we vary $\iota$ to obtain an estimate in the whole $Q_{\frac{R}{2}, (\frac{R}{2})^{s \, p}}$. Specifically we split the interval $[t_1,t_2]$ into $1+ \floor{\mu^{p-2}}$ pieces, say $[\tau_{i+1},\tau_i]$, with $\tau_i - \tau_{i+1}=\frac{\abs{t_2-t_1}}{1+\floor{\mu^{p-2}}}$, $\tau_0=t_2$, and $\tau_{\floor{1+\mu^{p-2}}} = t_1$. Using \eqref{eq:holder-final-coveri} we obtain
\[
\begin{aligned}
\abs{u(x_2,t_2) - u(x_1,t_1)} &\leq \abs{u(x_2,t_1) - u(x_1,t_1)} + \abs{u(x_2,t_2) - u(x_2,t_1)} \\
&\leq \mu C \Bigl(\frac{\abs{x_2-x_1}}{R} \Bigr)^\alpha + \sum_{i=0}^{\floor{\mu^{p-2}}} \abs{u(x_2,\tau_{i}) - u(x_2,\tau_{i+1})} \\
&\leq \mu C \Bigl[  \Bigl(\frac{\abs{x_2-x_1}}{R} \Bigr)^\alpha + \sum_{i=0}^{\floor{\mu^{p-2}}} \Bigl( \frac{\abs{\tau_i-\tau_{i+1}}}{R^{s\, p} \mu^{2-p}} \Bigr)^{\frac{\alpha}{sp-(p-2)\alpha}}    \Bigr] \\
&= \mu C \Bigl[  \Bigl(\frac{\abs{x_2-x_1}}{R} \Bigr)^\alpha + \sum_{i=0}^{\floor{\mu^{p-2}}} \Bigl( \frac{\abs{t_2-t_1}}{R^{s\, p} \mu^{2-p}(1+ \floor{\mu^{p-2}})} \Bigr)^{\frac{\alpha}{sp-(p-2)\alpha}}    \Bigr] \\
&\leq \mu C \Bigl[  \Bigl(\frac{\abs{x_2-x_1}}{R} \Bigr)^\alpha + \sum_{i=0}^{\floor{\mu^{p-2}}} \Bigl( \frac{\abs{t_2-t_1}}{R^{s\, p}} \Bigr)^{\frac{\alpha}{sp-(p-2)\alpha}}    \Bigr] \\
& \leq \mu C \Bigl[  \Bigl(\frac{\abs{x_2-x_1}}{R} \Bigr)^\alpha + 2 \mu ^{p-2}\Bigl( \frac{\abs{t_2-t_1}}{R^{s\, p}} \Bigr)^{\frac{\alpha}{sp-(p-2)\alpha}}    \Bigr],
\end{aligned}
\]
which concludes the desired result.
\end{proof}

\section{Appendix A}
In this section, we spell out the necessary modifications to prove the following theorem \ref{thm:thm1tail} which is a modified version of \cite[Theorem 1.2]{BLS}. 
\begin{theorem}\label{thm:thm1tail}
Let $\Omega\subset\R^n$ be a bounded and open set, $I=(t_0,t_1]$, $p\geq 2$ and $0<s<1$. Suppose $u$
is a local weak solution of
\[
u_t+(-\Delta_p)^s u=0\qquad \mbox{ in }\Omega\times I,
\]
such that 
\begin{equation}
\label{finitesup}
u\in L^{\infty}_{\loc}(I;L^\infty_{\loc}(\Omega)) \cap L^\infty_{\loc} (I; L_{sp}^{p-1}(\R^n)).
\end{equation}
Define the exponents
\begin{equation}
\label{exponents}
\Theta(s,p):=\left\{\begin{array}{rl}
\dfrac{s\,p}{p-1},& \text{ if } s<\dfrac{p-1}{p},\\
&\\
1,& \mbox{ if } s\ge \dfrac{p-1}{p},
\end{array}
\right.\quad \mbox{ and }\quad \Gamma(s,p):=\left\{\begin{array}{rl}
1,& \mbox{ if } s<\dfrac{p-1}{p},\\
&\\
\dfrac{1}{s\,p-(p-2)},& \text{ if } s\ge \dfrac{p-1}{p}.
\end{array}
\right.
\end{equation}
Then 
\[
u\in C^\delta_{x,\loc}(\Omega\times I)\cap C^\gamma_{t,\loc}(\Omega\times I),\qquad \mbox{ for every }0<\delta<\Theta(s,p) \ \mbox{ and } \ 0<\gamma<\Gamma(s,p).
\] 
More precisely, for every $0<\delta<\Theta(s,p)$, $0<\gamma<\Gamma(s,p)$, $R>0$, $x_0\in\Omega$ and $T_0$ such that 
\[
Q_{R,R^{s\,p}}(x_0,T_0)\Subset\Omega\times (t_0,t_1],
\] 
there exists a constant $C=C(n,s,p,\delta, \gamma,\sigma)>0$ such that
\begin{equation}
\label{eq:apriori_spacetime-sigma}
\begin{aligned}
|u(x_1,\tau_1)- & u(x_2,\tau_2)| \leq C\,\bigl(\|u\|_{L^\infty(Q_{R,R^{sp}(x_0,T_0)})} +\sup_{t \in [T_0-R^{sp},T_0]} \mathrm{Tail}_{p-1,sp}(u;x_0,R) + 1\bigr)\, \left(\frac{|x_1-x_2|}{R}\right)^\delta\\
&+C\,\bigl(\|u\|_{L^\infty(Q_{R,R^{sp}(x_0,T_0)}} +\sup_{t \in [T_0-R^{sp},T_0]} \mathrm{Tail}_{p-1,sp}(u;x_0,R) + 1 \bigr)^{p-1}\,\left(\frac{|\tau_1-\tau_2|}{R^{s\,p}}\right)^\gamma.
\end{aligned}
\end{equation}
for any $(x_1,\tau_1),\,(x_2,\tau_2)\in Q_{\sigma R,(\sigma R)^{s\,p}}(x_0,T_0)$. 
\end{theorem}

First we reproduce a modified version of \cite[Proposition 4.1]{BLS}, where instead of a global $L^\infty$ bound we assume $\norm{u}_{L^\infty(B_1 \times [-1,0])} +\sup_{t\in [-1,0]} \mathrm{Tail}_{p-1,sp}(u;0,1)) \leq 1$ .
\begin{proposition}
\label{prop:improve} 
Assume $p\ge 2$ and $0<s<1$. Let  $u$ be a local weak solution of $u_t+(-\Delta_p)^s u=0$ in $B_2\times (-2,0]$. We assume that 
\[
\norm{u}_{L^\infty(B_1\times [-1,0])} + \sup_{t\in [-1,0]} \mathrm{Tail}_{p-1,sp}(u(\frarg,t);0,1) \leq 1,
\]
and that, for some $q\ge p$ and $0<h_0<1/10$, we have
\[
\int_{T_0}^{T_1}\sup_{0<|h|< h_0}\left\|\frac{\delta^2_h u }{|h|^s}\right\|_{L^q(B_{R+4\,h_0})}^q \dd t<+\infty,
\]
for a radius $4\,h_0<R\le 1-5\,h_0$ and two time instants $-1<T_0<T_1\le 0$. Then we have
\begin{equation}
\label{iteralo}
\begin{split}
\int_{T_0+\mu}^{T_1}\sup_{0<\abs{h}< h_0}\left\|\frac{\delta^2_h u}{|h|^{s}}\right\|_{L^{q+1}(B_{R-4\,h_0})}^{q+1} \dd t &+\frac{1}{q+3-p}\sup_{0<\abs{h}< h_0}\left\|\frac{\delta_h u(\frarg,T_1)}{|h|^{\frac{(q+2-p)\,s}{q+3-p}}}\right\|_{L^{q+3-p}(B_{R-4\,h_0})}^{q+3-p}\\
&\leq C\,\int_{T_0}^{T_1}\left(\sup_{0<\abs{h}< h_0}\left\|\frac{\delta^2_h u }{\abs{h}^s}\right\|_{L^q(B_{R+4h_0})}^q+1\right)\dd t,
\end{split}
\end{equation}
for every $0<\mu<T_1-T_0$.
Here $C=C(n,s,p,q,h_0,\mu)>0$ and $C\nearrow +\infty$ as $h_0\searrow 0$ or $\mu\searrow 0$.
\end{proposition}
\begin{proof}
In the proof of \cite[Proposition 4.1]{BLS}, the $L^\infty(\R^n \times [0,1])$ boundedness is only used in Step 3, in the estimation of the nonlocal terms $\Ical_2$ and $\Ical_3$, which are defined by
\[
\begin{split}
\mathcal{I}_2(t):=\int_{B_\frac{R+r}{2}\times (\R^n\setminus B_R)}& \frac{\Big(J_p(u_h(x)-u_h(y))-J_p(u(x)-u(y))\Big)}{|h|^{1+\vartheta\,\beta}}\, \times J_{\beta+1}(u_h(x)-u(x))\,\eta(x)^p\dd \mu ,
\end{split}
\]
and
\[
\begin{split}
\mathcal{I}_3(t):=-\iint_{(\R^n\setminus B_R)\times B_\frac{R+r}{2}}& \frac{\Big(J_p(u_h(x)-u_h(y))-J_p(u(x)-u(y))\Big)}{\abs{h}^{1+\vartheta\,\beta}}\,\times J_{\beta+1}(u_h(y)-u(y))\,\eta(y)^p\dd \mu .
\end{split}
\]
We also recall the definition of $\tilde{\Ical}_2$ and $\tilde{\Ical}_3$
$$\tilde{\Ical}_i := \int_{T_0}^{T_1} \Ical_i(t) \,\tau(t) \dd t, \qquad i=2,3,$$
where $\tau$ is smooth function $0\leq \tau \leq 1$ such that
$$\tau \equiv 1 \qquad \text{on} \; [T_0+\mu, +\infty), \qquad \tau \equiv 0 \qquad \text{on} \; (-\infty,T_0], \qquad \tau^\prime \leq \frac{C}{\mu}. $$
 The general argument is the same but instead of using the $L^\infty$ norm of $u(y)$ we can keep the inequality as it is and write
\[
\absb{(J_p(u_h(x)-u_h(y))-J_p(u(x)-u(y)))J_{\beta + 1}(\delta_h u(x))} \leq C(1+\abs{u_h(y)}^{p-1} + \abs{u(y)})\abs{\delta_h u(x)}^\beta,
\]
where $x \in B_{R-2h_0} $ and $4h_0 < R < 1-5h_0$. Therefore, $\abs{x-y} \geq (1-\frac{R-2h_0}{R})\abs{y} \geq C(h_0) \abs{y}$ and we get
\[
\begin{aligned}
\int_{\R^n \setminus B_R} \frac{1 + \abs{u(y)}^{p-1} + \abs{u_h(y)}^{p-1}}{\abs{x-y}^{n+sp}} \dd y \leq \,
&C(n,s,p,h_0) + (C(h_0))^{n+sp}\; \int_{\R^n \setminus B_R} \frac{\abs{u(y)}^{p-1}}{\abs{y}^{n+sp}} \dd y\\
&+ (C(h_0))^{n+sp}\; \int_{\R^n \setminus B_R} \frac{\abs{u_h(y)}^{p-1}}{\abs{y}^{n+sp}} \dd y .
\end{aligned}
\]
Now
\[
\begin{aligned}
\int_{\R^n \setminus B_R} \frac{\abs{u(y)}^{p-1}}{\abs{y}^{n+sp}} \dd y &\leq \int_{\R^n \setminus B_1} \frac{\abs{u(y)}^{p-1}}{\abs{y}^{n+sp}} \dd y + R^{-n-sp} \int_{B_1} \abs{u}^{p-1} \dd y \\
&\leq 1 + n\omega_n R^{-n-sp} \leq 1+ n \omega_n (4h_0)^{-n-sp}  ,
\end{aligned}
\]
and for $u_h$ 
\[
\begin{aligned}
\int_{\R^n \setminus B_R} \frac{\abs{u(y+h)}^{p-1}}{\abs{y}^{n+sp}} \dd y &\leq
\int_{\R^n \setminus B_R(h)} \frac{\abs{u(y)}^{p-1}}{\abs{y-h}^{n+sp}} \dd y \leq 
(\frac{3}{2})^{n+sp} \int_{\R^n \setminus B_R(h)} \frac{\abs{u(y)}^{p-1}}{\abs{y}^{n+sp}} \dd y \\ 
&\leq (\frac{3}{2})^{n+sp} \bigl[ \int_{\R^n \setminus B_1} \frac{\abs{u(y)}^{p-1}}{\abs{y}^{n+sp}} \dd y + R^{-n-sp} \int_{B_1} \abs{u(y)}^{p-1} \dd y \bigr] \\
&\leq \bigl(\frac{3}{2}\bigl)^{n+sp}(1 + n\omega_n R^{-n-sp}) \leq \bigl(\frac{3}{2}\bigl)^{n+sp}(1 + n\omega_n (4h_0)^{-n-sp}).
\end{aligned}
\]
Here we have used $B_R(h) \subset B_1$, and $\frac{\abs{y-h}}{\abs{y}} = \abs{\frac{y}{\abs{y}}-\frac{h}{\abs{y}}} \geq \abs{\frac{y}{\abs{y}}} - \abs{\frac{h}{\abs{y}}} \geq 1- \abs{\frac{h_0}{R-h_0}} \geq \frac{2}{3}$ .
Using this we get
$$\int_{\R^n \setminus B_R} \frac{1 + \abs{u(y)}^{p-1} + \abs{u_h(y)}^{p-1}}{\abs{x-y}^{n+sp}} \dd y \leq C(n,s,p,h_0),$$
and we can conclude
$$\abs{\tilde{\Ical}_2}+\abs{\tilde{\Ical}_3} \leq C(n,s,p,h_0)\int_{T_0}^{T_1}\int_{B_{\frac{R+r}{2}}} \frac{\abs{\delta_h u}^\beta}{\abs{h}^{1+\vartheta \, \beta}} \tau \dd x \dd t 
\leq C(h_0,n,s,p,q,\beta) \int_{T_0}^{T_1} \Bigl ( 1+  \int_{B_R} \absB{\frac{\delta_h u}{\abs{h}^{\frac{1+\vartheta \, \beta}{\beta}}}}^{\frac{\beta q }{q-p+2}}\Bigr)\tau \dd t .$$
Which is the same as equation (4.6) in \cite{BLS}.
\end{proof}
We can estimate the $W^{s,p}$ semi-norm of a solution as follows. The proof follows the argument in \cite[Lemma 7.1]{BLS}.
\begin{lemma}\label{lm:seminormestimate}
Let $p\ge 2$ and $0<s<1$. Let $u$ be a local weak solution of
\[
\partial_t u+(-\Delta_p)^s u=0,\quad \text{ in } B_{2R}\times (-2\,R^{s\,p},0],
\]
such that $u\in L^\infty(B_{2R}\times[-R^{s\,p},0])$.
Then 
\[
\left(R^{-n}\,\int_{-\frac{7}{8}\,R^{s\,p}}^{0}[u]_{W^{s,p}(B_{R}(x_0))}^p\,\dd t\right)^{\frac{1}{p}} \le C\,\Big(\|u\|_{L^\infty(B_{2R}\times[-R^{s\,p},0])}+ \sup_{t\in [-R^{sp},0]}\mathrm{Tail}_{p-1,sp}(u;0,2R)+1\Big),
\]
for some $C=C(n,s,p)>0$.
\end{lemma}
 \begin{proof}
Without loss of generality, we may suppose that $x_0=0$.
Let
\[
k = \norm{u}_{L^\infty(B_{2R}\times[-R^{s\,p},0])}+ \sup_{t\in [-R^{sp},0]}\mathrm{Tail}_{p-1,sp}(u(\frarg,t);0,2R) + 1\qquad text{and }\qquad \widetilde u = u + k.
\]
Then $\widetilde u$ is a local weak solution in $B_2\times (-2\,R^{s\,p},0]$ and $\widetilde u\ge 1$ in $B_{2R}\times[-R^{s\,p},0]$. We choose $\phi$ and $\psi$ exactly as \cite[Lemma 7.1]{BLS}, that is 
\[
\eta \in C^\infty_0(2R) ,\qquad\eta \equiv 1\ \text{ in }\ B_{ R},\qquad |\nabla \eta|\le \frac{C}{R}\qquad \mbox{ and }\qquad \eta \equiv 0 \text{ in } \R^n \setminus B_{\frac{3}{2}R} ;
\]
and 
\[
\psi \in C^\infty(\R), \qquad \psi(t)=0 \text{ for }  t\leq -R^{s\,p} , \qquad \psi \equiv 1\ \text{ in }\ \left[-\frac{7}{8}\,R^{s\,p},0\right]\qquad \mbox{ and }\qquad |\psi'|\le \frac{C}{R^{s\,p}}.
\]
Then for $\phi(x,t)= \eta(x) \phi(t)$ we get
\[
\begin{aligned}
        \int_{-\frac{7}{8}R^{s\, p}}^0 &[\tilde{u}(\frarg,t)]_{W^{s,p}(B_R)}^p \dd t \leq \int_{-R^{s\,p}}^{0} \Big[\widetilde u(\frarg,t)\,\phi(\frarg,t)\Big]^p_{W^{s,p}(B_{2R})}\, \dd t\\
        &\leq C \int_{-R^{s\,p}}^{0} \iint_{B_{2R}\times B_{2R}} \max\Big\{\widetilde u(x,t),\, \widetilde u(y,t)\Big\}^p\,|\phi(x,t)-\phi(y,t)|^p\, \dd \mu \dd t\\
&+C\left (\sup_{x\in\mathrm{supp\,}\eta}\int_{\R^n\setminus B_{2R}}\frac{\dd y}{|x-y|^{n+s\,p}}\biggr)\biggl (\int_{-R^{s\,p}}^{0}\int_{B_{2R}}\widetilde u(x,t)^p\,\phi(x,t)^p\, \dd x \dd t\right)\\
&+C\Bigl(\sup_{t \in [-R^{s\,p},0]} \sup_{x\in\mathrm{supp\,}\eta}\int_{\R^n\setminus B_{2R}}\frac{(u(y,t)^+)^{p-1}}{|x-y|^{n+s\,p}}\,\dd y \Bigr) \int_{-R^{s\,p}}^{0} \,\int_{B_{2R}}\widetilde u(x,t)\,\phi(x,t)^p\,\dd x \dd t\\
                &+ \frac{1}{2}\int_{-R^{s\,p}}^{0} \int_{B_{2R}}\widetilde u(x,t)^{2} \left ( \frac{\partial \phi^p}{\partial t} \right )^{+}\, \dd x\, \dd t+\int_{B_{2R}} \widetilde u (x,0)\, \dd x \\
                &\leq C\,R^n\,(k^p+k^2+k)\leq C\,R^n\,k^p.
\end{aligned}
\]
The only difference in the proof is in estimating the term  
$$\sup_{x\in\mathrm{supp\,}\eta}\int_{\R^n\setminus B_{2R}}\frac{(u(y,t)^+)^{p-1}}{|x-y|^{n+s\,p}}\dd y .$$
 Noticing that for $x \in \supp \eta \subset B_{\frac{3}{2}R}$ we have $\frac{\abs{x-y}}{\abs{y}} \geq 1- \frac{\abs{x}}{\abs{y}} \geq 1-\frac{3/2 R}{2R} = \frac{1}{4}$, we get
\[
\int_{\R^n\setminus B_{2R}}\frac{(u(y,t)^+)^{p-1}}{|x-y|^{n+s\,p}}\dd y \leq 4^{n+s\,p}R^{-s\,p} \; \mathrm{Tail}_{p-1,sp}^{p-1}(u;0,2R) \leq C\,R^{-s\,p} k^{p-1} .
\]
\end{proof}
We can now prove the following modified version of \cite[Theorem 4.2]{BLS}.
\begin{theorem}[Spatial almost $C^s$ regularity]
\label{teo:1}
Let $\Omega\subset \R^n$ be a bounded and open set, $I=(t_0,t_1]$, $p\geq 2$ and $0<s<1$. Suppose $u$ is a local weak solution of 
\[
u_t+(-\Delta_p)^s u=0\qquad \text{ in }\Omega\times I,
\]
such that $u\in L^\infty_{\loc}(I;L^\infty(\Omega)) \cap L^\infty_{\loc}(I;L_{sp}^{p-1}(\R ^n))$.
Then $u\in C_{x,\loc}^\delta(\Omega\times I)$ for every $0<\delta<s$. 
\par
More precisely, for every $0<\delta<s$, $R>0$ and every $(x_0,T_0)$ such that 
\[
Q_{2R,2R^{s\,p}}(x_0,T_0)\Subset\Omega\times (t_0,t_1],
\] 
there exists a constant $C=C(n,s,p,\delta)>0$ such that
\begin{equation}
\label{apriori1}
\sup_{t\in \left[T_0-\frac{R^{s\,p}}{2},T_0\right]} [u(\frarg,t)]_{C^\delta(B_{R/2}(x_0))}\leq \frac{C}{R^\delta}\,\left( 1+ \norm{u}_{L^\infty(B_{2R}(x_0) \times [T_0-R^{s\,p},T_0])}+ \sup_{t \in [T_0-R^{sp},T_0]} \mathrm{Tail}_{p-1,sp}(u;x_0,2R)\right)
\end{equation}
\end{theorem}
\begin{proof}
The proof is essentially the same as the proof of \cite[Theorem 4.2]{BLS}. We assume for simplicity that $x_0=0$ and $T_0=0$, and set
\[
\mathcal{M}_R = \|u\|_{L^\infty(B_{2R} \times [-R^{s\,p},0])}+ \sup_{t \in [-R^{sp},0]} \mathrm{Tail}_{p-1,sp}(u;0,R)+\left(R^{-n}\,\int_{-\frac{7}{8}\,R^{s\,p}}^{0} [u]^p_{W^{s,p}(B_{R})}\dd t \right)^\frac{1}{p} +1.
\]
Notice that by Lemma \ref{lm:seminormestimate} we have
$$\Mcal_R \leq C\Bigl( \|u\|_{L^\infty(B_{2R} \times [-R^{s\,p},0])}+ \sup_{t \in [-R^{sp},0]} \mathrm{Tail}_{p-1,sp}(u;0,2R) \Bigr) + 1.$$
Let $\alpha\in [-R^{s\,p}(1-\Mcal_R^{2-p}),0]$ and define
\[
u_{R,\alpha}(x,t):=\frac{1}{\Mcal_R}\,u\left(R\,x,\frac{1}{\Mcal_R^{p-2}}\,R^{s\,p}\,t+\alpha\right),\qquad \text{ for }x\in B_2,\ t\in (-2,0].
\]
Then $u_{R,\alpha}(x,t)$ is a local weak solution of 
$$u_t+(-\Delta_p)^s u=0,\qquad \text{ in } B_2\times (-2,0],$$
that satisfies
$$\norm{u_{R,\alpha}}_{L^\infty(B_{2} \times [-1,0])} + \sup_{t \in [-1,0]} \mathrm{Tail}_{p-1,sp}(u(\frarg,t);0,1) \leq 1,\qquad  \int_{-\frac{7}{8}}^0 [u_{R,\alpha}]^p_{W^{s,p}(B_1)}\, \dd t \leq 1  .$$
This function satisfies the assumption of Proposition \ref{prop:improve}, and we can do the same argument as in \cite{BLS} to obtain
$$\sup_{t\in [-1/2,0]}[u_{R,\alpha}(\frarg,t)]_{C^{\delta}(B_{1/2})}\leq C(n,s,p,\delta), $$
for a $C$ independent of $\alpha$ and by scaling back we get
$$
\sup_{\alpha-\frac{1}{2}\Mcal_R^{2-p}\,R^{s\,p}\leq t\leq 0}[u(\frarg,t)]_{C^{\delta}(B_{R/2})}\leq \frac{C}{R^\delta}\Mcal_R.
$$
By varying $\alpha\in [-R^{s\,p}(1-\Mcal_R^{2-p}),0]$ we get the desired result. 
\end{proof}
We now address the improved regularity and start with the following modified version of \cite[Proposition 5.1]{BLS}.
\begin{proposition}
Assume $p\ge 2$ and $0<s<1$. Let  $u$ be a local weak solution of $u_t+(-\Delta_p)^s u=0$ in $B_2\times (-2,0]$, such that
$$
\|u\|_{L^{\infty}(B_2 \times[-1,0])} + \sup_{t \in [-1,0]} \mathrm{Tail}_{p-1,sp}(u;0,2) \le 1,
$$
Assume further that for some $0<h_0<1/10$ and $\vartheta<1$, $\beta\ge 2$ such that $(1+\vartheta\, \beta)/\beta<1$, we have
$$
\int_{T_0}^{T_1}\sup_{0<|h|\leq h_0}\left\|\frac{\delta^2_h u }{|h|^\frac{1+\vartheta\, \beta}{\beta}}\right\|_{L^\beta(B_{R+4\,h_0})}^\beta \dd t<+\infty ,
$$
for a radius $4\,h_0<R\le 1-5\,h_0$ and two time instants $-1<T_0< T_1\le 0$. Then 
\begin{equation}
\label{iteralo2}
\begin{aligned}
\int_{T_0+\mu}^{T_1} &\sup_{0<\abs{h}< h_0}\left\|\frac{\delta^2_h u}{\abs{h}^{\frac{1+s\,p+\vartheta\,\beta}{\beta-1+p}}}\right\|_{L^{\beta-1+p}(B_{R-4\,h_0})}^{\beta-1+p} \dd t +\frac{1}{\beta+1}\sup_{0<\abs{h}< h_0}\left\|\frac{\delta_h u(\frarg,T_1)}{|h|^{\frac{1+\vartheta\,\beta}{\beta+1}}}\right\|_{L^{\beta+1}(B_{R-4\,h_0})}^{\beta+1}\\
&\leq C\,\int_{T_0}^{T_1}\sup_{0<\abs{h}< h_0}\left(\left\|\frac{\delta^2_h u }{\abs{h}^\frac{1+\vartheta\, \beta}{\beta}}\right\|_{L^\beta(B_{R+4\,h_0})}^\beta+1\right) \dd t ,
\end{aligned}
\end{equation}
for every $0<\mu<T_1-T_0$.
Here $C$ depends on the $n$, $h_0$, $s$, $p$, $\mu$ and $\beta$.
\end{proposition}
\begin{proof}
The only major difference from the proof of Proposition \ref{prop:improve} is in the estimation of term $\Ical_{1 1}$ and it can be treated in the exact same way as in the proof of \cite[Proposition 5.1]{BLS}.
\end{proof}
Using the previous Proposition with the same type of modifications as in the proof of Theorem \ref{teo:1} we can state the following version of \cite[Theorem 5.2]{BLS}.
\begin{theorem}
\label{teo:1higher} Let $\Omega$ be a bounded and open set, let $I=(t_0,t_1]$, $p\geq 2$ and $0<s<1$. Suppose $u$ is a local weak solution of 
\[
u_t+(-\Delta_p)^s u=0\qquad \mbox{ in }\Omega\times I,
\]
such that $u\in L^\infty_{\loc}(I;L^\infty_{\loc}(\Omega)) \cap L^\infty_{\loc} (I; L_{sp}^{p-1}(\R^n))$.
Then $u\in C^\delta_{x,\loc}(\Omega\times I)$ for every $0<\delta<\Theta(s,p)$, where $\Theta(s,p)$ is defined in \eqref{exponents}.
\par
More precisely, for every $0<\delta<\Theta(s,p)$, $R>0$, $x_0\in\Omega$ and $T_0$ such that 
\[
B_{2R}(x_0)\times [T_0-2\,R^{s\,p},T_0]\Subset\Omega\times (t_0,t_1],
\] 
there exists a constant $C=C(n,s,p,\delta)>0$ such that
\begin{equation}
\label{apriori}
\sup_{t\in \left[T_0-\frac{R^{s\,p}}{2},T_0\right]} [u(\frarg,t)]_{C^\delta(B_{R/2}(x_0))} \leq \frac{C}{R^\delta}\,\left(\|u\|_{L^\infty(B_{2R}\times [T_0-R^{s\,p},T_0])} +\sup_{t \in [T_0-R^{sp},T_0]} \mathrm{Tail}_{p-1,sp}(u;x_0,2R) + 1\right) .
\end{equation}
\end{theorem}
Now we modify the argument regarding the regularity in time (see \cite[Proposition 6.2]{BLS}).
\begin{proposition}\label{prop_time}
Suppose that $u$ is a local weak solution of 
$$
\partial_t u+(-\Delta_p)^s u =0,\qquad\text{ in } B_2\times (-2,0], 
$$
such that 
$$
\norm{u}_{L^{\infty}(B_2 \times[-1,0])} + \sup_{t \in [-1,0]} \mathrm{Tail}_{p-1,sp}(u;0,2) \leq 1,
$$
and
\begin{equation}
\label{eq:holderseminorm-condit}
\sup_{t\in [-1/2,0]}[u(\frarg,t)]_{C^\delta(B_{1/2})}\le K_\delta,\qquad \mbox{ for any } s<\delta<\Theta(s,p),
\end{equation}
where $\Theta(s,p)$ is the exponent defined in \eqref{exponents}.
Then there is a constant $C= C(n,s,p,K_\delta,\delta)>0$ such that 
$$
\abs{u(x,t)-u(x,\tau)} \leq C\, \abs{t-\tau}^{\gamma},\qquad \text{ for every } (x,t),(x,\tau)\in Q_{\frac{1}{4},\frac{1}{4}},
$$
where 
$$
\gamma = \frac{1}{\dfrac{s\,p}{\delta}-(p-2)}. 
$$
In particular, $u\in C^\gamma_{t}(Q_{\frac{1}{4},\frac{1}{4}})$ for any $\gamma<\Gamma(s,p)$, where $\Gamma(s,p)$ is the exponent defined in \eqref{exponents}.
\end{proposition}
\begin{proof}
The only part that needs to be modified is the estimation of the nonlocal term $J_2$
\[
J_2:= \int_{T_0}^{T_1}\iint_{(\R^n\setminus B_r(x_0))\times B_{r/2}(x_0)}J_p(u(x,\tau)-u(y,\tau))\,\eta(x)\,\dd \mu(x,y) \dd \tau,
\]
here $T_0,T_1 \in (t_0-\theta , t_0)$ with $T_0<T_1$. We recall that $0< \theta < \frac{1}{8}$, $x_0 \in B_{\frac{1}{4}}$, and $r < \frac{1}{8}$.
%in particular for $x \in B_{\frac{r}{2}(x_0)}$ and $y \in \R^N \setminus B_{r}(x_0)$ we have 
Thus $x\in B_{\frac{r}{2}}(x_0)$ implies $x\in B_{\frac{5}{16}}$. 

For $y \in B_{\frac{1}{2}}(0)$, assumption \eqref{eq:holderseminorm-condit} implies
$$ \abs{u(x,\tau) - u(y,\tau)} \leq K_\delta \abs{x-y}^\delta. $$
For $y \in B_2(0)\setminus B_{\frac{1}{2}}(0)$ the $L^\infty$ bound on $u$ implies
$$ \abs{u(x,\tau) - u(y,\tau)} \leq 2 \leq C(\delta) \abs{x-y}^\delta .$$
Also notice that for $x\in B_{r/2}(x_0)$ and $y \in \R^n \setminus B_{r}(x_0)$,we have $\abs{x-y} \geq \frac{1}{2}\abs{y-x_0}$. Using these we obtain
\[
\begin{aligned}
J_2 &\leq 2\,(T_1-T_0)\,\|\eta\|_{L^\infty(B_{r/2}(x_0))} 
\sup_{t \in [-\frac{1}{2} , 0]} \iint_{(\R^n\setminus B_r(x_0))\times B_{r/2}(x_0)}\frac{|u(x,t)-u(y,t)|^{p-1}}{|x-y|^{n+s\,p}}\,\dd y \dd x \\
& \leq 2\,(T_1-T_0)\,\norm{\eta}_{L^\infty(B_{r/2}(x_0))} \Bigl( \sup_{t \in [-\frac{1}{2} , 0]}\iint_{(\R^n\setminus B_2)\times B_{r/2}(x_0)}\frac{\abs{u(x,t)-u(y,t)}^{p-1}}{\abs{x-y}^{n+sp}}\,\dd y\dd x \\
&\quad + C(\delta , K_\delta ) \iint_{(B_2 \setminus B_r(x_0)) \times B_{r/2}(x_0))} \abs{x-y}^{\delta(p-1) - n -sp} \dd y \dd x \, \Bigr) \\
&\leq C\theta \Bigl( \sup_{t \in [-\frac{1}{2} , 0]} \iint_{(\R^n \setminus B_2)\times B_{r/2}(x_0)} \frac{1+ \abs{u(y,t)}^{p-1}}{\abs{x_0 -y}^{n +sp}} \dd y \dd x \\
& \qquad +\iint_{(B_2(0) \setminus B_r(x_0) \times B_{r/2}(x_0))} \abs{x_0-y}^{\delta(p-1) - n-sp} \dd y \dd x \,\Bigr) \\
& \leq C \theta \int_{B_{r/2}(x_0)} \Bigl( \sup_{t \in [-\frac{1}{2} ,0]} \int_{\R^n \setminus B_2} \frac{1+ \abs{u(y,t)}^{p-1}}{\abs{y}^{n+sp}} \dd y + \int_{B_2 \setminus B_r(x_0)} \abs{x_0-y}^{\delta(p-1) - n -sp} \dd y \Bigr) \dd x \\
& \leq C\, \theta\, r^n \Bigl( 2^{-sp} +1 + r^{\delta(p-1) - sp} \Bigr) \leq C\,\theta \, r^{n-sp + \delta(p-1)} . \quad \text{(since} \; \delta(p-1) -sp \; \text{is not positive)}
\end{aligned}
\]
\end{proof}

Finally, we are ready to prove a modified version of \cite[Theorem 1.1]{BLS}, which is Theorem \ref{thm:thm1tail}.
\begin{proof}[Proof of Theorem \ref{thm:thm1tail}]
Consider a cylinder $Q_{2\rho,2\rho^{sp}}(\tilde{x},\tau) \Subset \Omega \times I$, first, we prove the following type of bound on the H\"older seminorm in $Q_{\rho/4,\rho^{s\,p}/4}(\tilde{x},\tau)$, and later with the aid of a covering argument we conclude the claim of the theorem.

\textbf{Claim:} For any $(x_1,\tau_1),\,(x_2,\tau_2)\in Q_{\rho/4,\rho^{s\,p}/4}(\tilde{x},\tau)$ we have
\begin{equation}
\label{apriori_spacetime-pre}
\begin{aligned}
|u(x_1,\tau_1)- & u(x_2,\tau_2)| \leq C\,\bigl(\|u\|_{L^\infty(B_{2\rho}\times [T_0-\rho^{s\,p},T_0])} +\sup_{t \in [T_0-\rho^{sp},T_0]} \mathrm{Tail}_{p-1,sp}(u;x_0,2\rho) + 1\bigr)\, \left(\frac{|x_1-x_2|}{\rho}\right)^\delta\\
&+C\,\bigl(\|u\|_{L^\infty(B_{2\rho}\times [T_0-R^{s\,p},T_0])} +\sup_{t \in [T_0-\rho^{sp},T_0]} \mathrm{Tail}_{p-1,sp}(u;x_0,2\rho) + 1 \bigr)^{p-1}\,\left(\frac{|\tau_1-\tau_2|}{\rho^{s\,p}}\right)^\gamma.
\end{aligned}
\end{equation}
The regularity in space variable has been proven in Theorem \ref{teo:1higher}. To prove the part on time regularity we set 
$$\Mcal_\rho(\tilde{x},\tau):= 1 + \norm{u}_{L^\infty(Q_{2\rho,\rho^{sp}}(\tilde{x},\tau))} + \sup_{\tau - \rho^s \leq t \leq \tau} \mathrm{Tail}_{p-1,sp}(u;\tilde{x},2\rho) $$
and consider the rescaled functions 
$$\tilde{u}_{\rho,\iota}(x,t):= \frac{1}{\Mcal_\rho(\tilde{x},\tau)} u(\rho x + \tilde{x}, \Mcal_\rho(\tilde{x},\tau)^{2-p}\rho^{sp}t + \tau + \iota),$$
for $\iota \in (-\frac{\rho^{sp}}{4}(1- \Mcal_\rho^{2-p}),0)$. Then $\tilde{u}_{\rho,\iota}(x,t)$ is a solution of 
$$\partial_t \tilde{u}_{\rho,\iota} + (-\Delta_p)\tilde{u}_{\rho,\iota} = 0, \qquad \text{in}\qquad Q_{2,2}. $$
Moreover, $\tilde{u}_{\rho,\iota}(x,t)$ satisfies the conditions of Proposition \ref{prop_time}. Indeed by construction
$$\|\tilde{u}_{\rho,\iota}\|_{L^{\infty}(B_2 \times[-1,0])} + \sup_{t \in [-1,0]} \mathrm{Tail}_{p-1,sp}(\tilde{u}_{\rho,\iota};0,2) \le 1$$
and the estimate \eqref{eq:holderseminorm-condit} follows from \eqref{apriori} in Theorem \ref{teo:1higher}. From Proposition \ref{prop_time} we obtain 
$$\sup_{x \in B_{\frac{1}{4}}} [\tilde{u}_{\rho,\iota}(x,\frarg)]_{C^\gamma[-\frac{1}{4},0]} \leq C,$$
with $C= C(n,s,p,\gamma)$ for every $0< \gamma < \Gamma(s,p)$. By scaling back this translates to
\begin{equation}\label{eq:holder-time-homogen}
\abs{u(x,t_1)-u(x,t_2)} \leq \Mcal_\rho(\tilde{x},\tau) C \Bigl( \frac{\abs{t_1 - t_2}}{\rho^{sp} \Mcal_\rho^{2-p}}\Bigr)^\gamma \quad \text{for}\quad (x,t_1) , (x,t_2) \in Q_{\frac{\rho}{4},\frac{\rho^{sp}}{4}\Mcal_\rho^{2-p} }(\tilde{x}, \tau+\iota).
\end{equation} 
By varying $\iota$ with an argument similar to the proof of Theorem \ref{thm:main-holder} we arrive at the claim \eqref{apriori_spacetime-pre}. We have to point out that the H\"older constant does change, unlike what is suggested in the proof of \cite[Theorem 1.1]{BLS}. Here is a detailed computation 

We split the time interval $[t_1,t_2]$ into $1+ \floor{\Mcal_{\rho}(\tilde{x},\tau)^{p-2}}$ pieces, say $[\tau_{i+1},\tau_i]$, with $\tau_i - \tau_{i+1}=\frac{\abs{t_2-t_1}}{1+\floor{\Mcal_{\rho}(\tilde{x},\tau)^{p-2}}}$, $\tau_0=t_2$, and $\tau_{\floor{1+\mu^{p-2}}} = t_1$. Then using \eqref{eq:holder-time-homogen} and the triangle inequality we get 
\[
\begin{aligned}
\abs{u(x,t_2) - u(x,t_1)} &\leq  \abs{u(x,t_2) - u(x,t_1)} \\
&\leq \sum_{i=0}^{\floor{\Mcal_{\rho}^{p-2}}} \abs{u(x_2,\tau_{i}) - u(x_2,\tau_{i+1})} \\
&\leq C\Mcal_{\rho} \sum_{i=0}^{\floor{\Mcal_{\rho}^{p-2}}} \Bigl( \frac{\abs{\tau_i-\tau_{i+1}}}{R^{s\, p} \Mcal_{\rho}^{2-p}} \Bigr)^{\gamma}     \\
&= C\Mcal_{\rho}(\tilde{x},\tau) \sum_{i=0}^{\floor{\Mcal_{\rho}^{p-2}}} \Bigl( \frac{\abs{t_2-t_1}}{R^{s\, p} \Mcal_{\rho}^{2-p}(1+ \floor{\Mcal_{\rho}^{p-2}})} \Bigr)^{\gamma}     \\
&\leq C \Mcal_{\rho} \sum_{i=0}^{\floor{\Mcal_{\rho}^{p-2}}} \Bigl( \frac{\abs{t_2-t_1}}{R^{s\, p}} \Bigr)^{\gamma}  \\
& \leq C\Mcal_{\rho} \Bigl[   \Mcal_{\rho} ^{p-2}\Bigl( \frac{\abs{t_2-t_1}}{R^{s\, p}} \Bigr)^{\gamma}    \Bigr] \leq C \Mcal_\rho^{p-1}\Bigl( \frac{\abs{t_2-t_1}}{R^{s\, p}} \Bigr)^{\gamma}   .
\end{aligned}
\]
Now use \eqref{apriori_spacetime-pre} in cylinders of the form
 $$Q_{\frac{r}{4},\frac{r^{s p}}{4}}(y,t), \quad \text{for} \; (y,t) \in Q_{\sigma R,(\sigma R)^{sp}} ,$$
where the radius $r= \frac{R}{C(n,s,p,\sigma)}$ is so small, such that
$$Q_{2r,2r^{s\, p}}(y,t) \subset Q_{R,R^{sp}} .$$
Consider a sequence of points $(\tilde{x}_i,\tilde{\tau}_i)$ on the segment joining $(x_1,\tau_1)$ and $(x_2,\tau_2)$ such that 
$$(\tilde{x}_i,\tilde{\tau}_i) \in Q_{\frac{r}{4}, \frac{r^{sp}}{4}}(x_{i-1},\tau_{i-1}).$$
Using \eqref{apriori_spacetime-pre} together with the triangle inequality, we obtain 
\[
\begin{aligned}
|u(x_1,\tau_1)- & u(x_2,\tau_2)| \leq C\,\bigl(\|u\|_{L^\infty(Q_{R,R^{sp}(x_0,T_0)})} +\sup_{t \in [T_0-R^{sp},T_0]} \mathrm{Tail}_{p-1,sp}(u;x_0,R) + 1\bigr)\, \left(\frac{|x_1-x_2|}{R}\right)^\delta\\
&+C\,\bigl(\|u\|_{L^\infty(Q_{R,R^{sp}(x_0,T_0)}} +\sup_{t \in [T_0-R^{sp},T_0]} \mathrm{Tail}_{p-1,sp}(u;x_0,R) + 1 \bigr)^{p-1}\,\left(\frac{|\tau_1-\tau_2|}{R^{s\,p}}\right)^\gamma,
\end{aligned}
\]
with $C=C(n,s,p,\delta,\gamma,\sigma)$, which is the desired result.

\end{proof}
\section{Appendix B}
Here we will justify the insertion of $u-v$ and $\abs{u-v}^{p-2}(u-v)$ as test functions.

%We recall the following contraction property of mollification 
%\begin{lemma}\label{lm:molif}
%\[
%\phi^\epsilon(x,t) := \frac{1}{\epsilon} \int_{t-\frac{\epsilon}{2}}^{t+ \frac{\epsilon}{2}} \zeta(\frac{t-\ell}{\epsilon})\phi(x,\ell) \dd \ell =  \int_{-\frac{1}{2}}^{\frac{1}{2}} \zeta(\sigma)\phi(x,t-\epsilon \sigma) \dd \sigma ,
%\]
%Then 

%\end{lemma}

\begin{proposition}\label{prop:testing}
Let $B= B_R(x_0)$ be a ball of radius $r$, $B_2=B_{\sigma R}(x_0)$ with $\sigma> 1$, and $I=(\tau_0,\tau_1]$ be an interval. Let  $f \in L^{(p_s^\star)^\prime,p^\prime}(B\times I)$ and assume that  $u \in L^p(I, W^{s,p}(B_2)) \cap L^{p-1}(I;L_{sp}^{p-1}(\R^n)) \cap C(I;L^2(B))$ be a local weak solution of
$$ u_t + (-\Delta_p)^s u = f , \quad \text{in} \; B_2 \times I$$
 with 
$$\mathrm{Tail}_{p-1,sp}(u(\frarg,t);x_0, R)\in L^{p}(I).$$
(in particular, this will be the case under the stronger assumption $\sup_{t \in I} \mathrm{Tail}_{p,sp}(u(\frarg,t);x_0,R) < \infty $ that we use in this article.)
For any time interval $[T_0,T_1] \Subset I$ and let $v \in L^p([T_0,T_1], W^{s,p}(B_2)) \cap L^{p-1}([T_0,T_1];L_{sp}^{p-1}(\R^n)) \cap C([T_0,T_1];L^2(B))$ be a weak solution to
\[
\begin{cases}
v_t + (-\Delta_p)^s v=0 \quad & in \;\; B \times [T_0,T_1] \\
v=u \quad & in \;\; (\R^n \setminus B ) \times [T_0,T_1] \\
v(x,T_0) = u(x,T_0) \quad & in \;\; B
\end{cases}
\]
In addition, assume that $F$ is a globally Lipschitz function with $F(0)=0$, which is either bounded or  $F(a)=a$. Then we have:
\[
\begin{aligned}
\int_{T_0}^{T_1} \iint_{\R^n \times \R^n}& \Bigl( J_p(u(x,t) - u(y,t)) - J_p(v(x,t) - v(y,t)) \Bigr)\\
&\times \Bigl( F(u(x,t)-v(x,t)) - F(u(y,t) -v(y,t))\Bigr)  \dd \mu \dd t \\
& + \int_{B}  \Fcal(u(x,t) - v(x,t)) \dd x \bigl|_{T_0}^{T_1} \\
 &=   \int_{T_0}^{T_1} \int_{B} F(u-v)f \dd x \dd t ,
\end{aligned}
\]
where $\Fcal(a) := \int_{0}^{a} f(t) \dd t$ is the primitive function of $F$.
% The assumption of boundedness is only needed in dealing with $\Fcal$, we also justify the case of unbounded function $F(t)= t$ .
\end{proposition}
\begin{proof}
The proof is essentially the same as \cite[Lemma 3]{BLS}, except that here we don't use a cut off function and don't have the global boundedness of $u$ in the ball.
For simplicity we assume $x_0= 0$, $R=1$ and $\sigma=2$.

 For a function $\phi \in C((T_0,T_1) ; L^2(B)) \cap L^p((T_0,T_1) ; X^{s,p}_0(B,B_2))$, we use the following regularization of functions
\[
\phi^\epsilon(x,t) := \frac{1}{\epsilon} \int_{t-\frac{\epsilon}{2}}^{t+ \frac{\epsilon}{2}} \zeta(\frac{t-\ell}{\epsilon})\phi(x,\ell) \dd \ell =  \int_{-\frac{1}{2}}^{\frac{1}{2}} \zeta(-\sigma)\phi(x,t+\epsilon \sigma) \dd \sigma ,
\]
where $\zeta(\sigma)$ is a smooth function with compact support in $(-\frac{1}{2}, \frac{1}{2})$ satisfying
$$\abs{\zeta} \leq 1, \qquad \text{and} \qquad  \abs{\zeta^\prime} \leq 8.$$
 This regulization process gives us a test function $\phi^\epsilon \in C^1((T_0+ \epsilon,T_1- \epsilon); L^2(B)) \cap L^p((T_0+\epsilon,T_1-\epsilon); X^{s,p}_0(B,B_2) )$.
Let $t_0 = T_0 + \epsilon_0$ and $t_1 = T_1- \epsilon_0$ and we test the equation with $\phi^\epsilon$ as above, for $\epsilon < \frac{\epsilon_0}{2}$. First, we will show the claim for the smaller interval $[t_0,t_1] \subset [T_0,T_1]$, and then through a limiting argument, prove the result for the whole interval.
As in equation (3.5) in \cite{BLS}, we get
\[
\begin{aligned}
\int_{t_0}^{t_1} &\iint_{\R^n \times \R^n} \Bigl( J_p(u(x,t) - u(y,t))(\phi^\epsilon(x,t) - \phi^\epsilon(y,t)) \dd \mu \dd t  \\ 
& + \int_{B} \int_{t_0+\frac{\epsilon}{2}}^{t_1-\frac{\epsilon}{2}} \partial_t u^\epsilon(x,t) \phi(x,t) \dd t \dd x + \Sigma_u(\epsilon) \\
& = \int_{B}\bigl[ u(x,t_0)\phi(x,t_0) -u^\epsilon(x,t_0+\frac{\epsilon}{2})\phi(x,t_0 + \frac{\epsilon}{2} )\bigr] \dd x \\
& - \int_{B}\bigl[ u(x,t_1)\phi(x,t_1) -u^\epsilon(x,t_1 -\frac{\epsilon}{2})\phi(x,t_1-\frac{\epsilon}{2}) \bigr] \dd x +
\int_{t_0}^{t_1} \int_{B} \phi^\epsilon f \dd x \dd t,
\end{aligned}
\]
and we obtain a similar identity for $v$ without $\int_{t_0}^{t_1} \int_{B} \phi^\epsilon f \dd x \dd t$ in the right hand side. Here $\Sigma_u$ is defined by
\[
\begin{aligned}
\Sigma_u(\varepsilon)=&-\int_{B}\int_{t_0-\frac{\varepsilon}{2}}^{t_0+\frac{\varepsilon}{2}}  \left(\frac{1}{\varepsilon}\,\int_{t_0}^{\ell+\frac{\varepsilon}{2}} u(x,t)\,\zeta\left(\frac{\ell-t}{\varepsilon}\right)\dd t\right)\,\partial_\ell\phi(x,\ell)\dd \ell\dd x\\
&-\int_{B}\int_{t_1-\frac{\varepsilon}{2}}^{t_1+\frac{\varepsilon}{2}}  \left(\frac{1}{\varepsilon}\,\int_{\ell-\frac{\varepsilon}{2}}^{t_1} u(x,t)\,\zeta\left(\frac{\ell-t}{\varepsilon}\right)\dd t\right)\,\partial_\ell\phi(x,\ell)\dd \ell\dd x.
\end{aligned}
\]

Observe that by using an integration by parts, the term $\Sigma_u(\epsilon)$ can be rewritten as
\begin{equation}\label{eq:testing-error-term}
\begin{aligned}
\Sigma_u(\epsilon)=&-\int_{B}\left(\frac{1}{\varepsilon}\,\int_{T_0}^{T_0+\varepsilon} u(x,t)\,\zeta\left(\frac{T_0-t}{\varepsilon}+\frac{1}{2}\right)\dd t\right)\,\phi\left(x,T_0+\frac{\varepsilon}{2}\right)\dd x\\
&+\int_{B}\int_{T_0-\frac{\varepsilon}{2}}^{T_0+\frac{\varepsilon}{2}}  \left(\frac{1}{\varepsilon^2}\,\int_{T_0}^{\ell+\frac{\varepsilon}{2}} u(x,t)\,\zeta'\left(\frac{\ell-t}{\varepsilon}\right)\dd t\right)\,\phi(x,\ell) \dd \ell\dd x\\
&+\int_{B}\left(\frac{1}{\varepsilon}\,\int^{T_1}_{T_1-\varepsilon} u(x,t)\,\zeta\left(\frac{T_1-t}{\varepsilon}-\frac{1}{2}\right)\dd t\right)\,\phi\left(x,T_1-\frac{\varepsilon}{2}\right)\dd x\\
&-\int_{B}\int_{T_1-\frac{\varepsilon}{2}}^{T_1+\frac{\varepsilon}{2}}  \left(\frac{1}{\varepsilon^2}\,\int^{T_1}_{\ell-\frac{\varepsilon}{2}} u(x,t)\,\zeta'\left(\frac{\ell-t}{\varepsilon}\right)\dd t\right)\,\phi(x,\ell)\dd \ell\dd x,
\end{aligned}
\end{equation}
where we also used that $\zeta$ has compact support in $(-1/2,1/2)$.
By subtracting the identities for $u$ and $v$, we obtain
\[
\begin{aligned}
\int_{t_0}^{t_1} &\iint_{\R^n \times \R^n} \Bigl( J_p(u(x,t) - u(y,t)) -J_p(v(x,t) - v(y,t))\bigr)(\phi^\epsilon(x,t) - \phi^\epsilon(y,t)) \dd \mu \dd t  \\ 
& + \int_{B} \int_{t_0+\frac{\epsilon}{2}}^{t_1-\frac{\epsilon}{2}} \partial_t (u-v)^\epsilon(x,t) \phi(x,t) \dd t \dd x + \Sigma_u(\epsilon) -\Sigma_v(\epsilon) \\
& = \int_{B}\bigl[ (u-v)(x,t_0)\phi(x,t_0) -(u-v)^\epsilon(x,t_0+\frac{\epsilon}{2})\phi(x,t_0+\frac{\epsilon}{2} )\bigr] \dd x \\
& - \int_{B}\bigl[ (u-v)(x,t_1)\phi(x,t_1) -(u-v)^\epsilon(x,t_1 -\epsilon)\phi(x,t_1-\frac{\epsilon}{2}) \bigr] \dd x + \int_{t_0}^{t_1} \int_{B} \phi^\epsilon(x,t) f(x,t) \dd x \dd t.
\end{aligned}
\]
Now we take $\phi$ to be $F(u^\epsilon -v^\epsilon)$. Observe that
$$\partial_t (u-v)^\epsilon F(u^\epsilon -v^\epsilon) = \partial_t \Fcal(u^\epsilon -v^\epsilon).$$
After an integration by parts we get
\begin{equation}\label{eq:testing-epsilon}
\begin{aligned}
\int_{t_0}^{t_1} &\iint_{\R^n \times \R^n} \Bigl( J_p(u(x,t) - u(y,t)) -J_p(v(x,t) - v(y,t))\bigr)\\
 &\times([F(u^\epsilon -v^\epsilon)(x,t) ]^\epsilon - [F(u^\epsilon -v^\epsilon)(y,t) ]^\epsilon) \dd \mu \dd t  \\
&+\int_B \Fcal (u^\epsilon - v^\epsilon) \dd x \Bigr]_{t_0+\frac{\epsilon}{2}}^{t_1-\frac{\epsilon}{2}} 
 + \Sigma_u(\epsilon) -\Sigma_v(\epsilon) \\
=& \int_{B}\bigl[ (u-v)(x,t_0)F(u^\epsilon - v^\epsilon)(x,t_0) -(u-v)^\epsilon(x,t_0+\frac{\epsilon}{2})F(u^\epsilon - v^\epsilon)(x,t_0+\frac{\epsilon}{2}) \bigr] \dd x \\
& - \int_{B}\bigl[ (u-v)(x,t_1)F(u^\epsilon - v^\epsilon)(x,t_1) -(u-v)^\epsilon(x,t_1 -\frac{\epsilon}{2})F(u^\epsilon - v^\epsilon)(x,t_1-\frac{\epsilon}{2}) \bigr] \dd x \\
&+ \int_{t_0}^{t_1} \int_{B} (F(u^\epsilon - v^\epsilon) )^\epsilon(x,t) f(x,t) \dd x \dd t := \Ical_1- \Ical_2 + \Ical_3.
\end{aligned}
\end{equation}
We now wish to pass to the limit in $\Ical_1, \Ical_2$, and $\Ical_3$.
Let $w=u-v$, we now treat $\Ical_1$. The fact that $F$ is globally Lipschitz together with $F(0)=0$ implies $\abs{F(t)} \leq C \abs{t}$. Therefore,
\[
\begin{aligned}
& \abs{w(x,t_0)F(w^\epsilon)(x,t_0) -w^\epsilon(x,t_0+\frac{\epsilon}{2})F(w^\epsilon (x,t_0+ \frac{\epsilon}{2} )) } \\
&\leq \abs{(w(x,t_0)-w^\epsilon(x,t_0+\frac{\epsilon}{2}))F(w^\epsilon(x,t_0))} + \abs{w^\epsilon(x,t_0+\frac{\epsilon}{2})(F(w^\epsilon(x,t_0))-F(w^\epsilon(x,t_0+ \frac{\epsilon}{2})))} \\
& \leq C\bigl[ \: \abs{(w(x,t_0) - w^\epsilon(x,t_0+ \frac{\epsilon}{2}))w^\epsilon(x,t_0)} + \abs{w^\epsilon(x,t_0+ \frac{\epsilon}{2})(w^\epsilon(x,t_0)-w^\epsilon(x, t_0+\frac{\epsilon}{2}))}  \bigr],
\end{aligned}
\]
where $C$ is the Lipschitz constant of $F$. After integrating and using H\"older's inequality, we obtain
\[
\begin{aligned}
\Ical_1 =& \int_{B}\bigl[ w(x,t_0)F(w^\epsilon)(x,t_0) - w^\epsilon(x,t_0+ \frac{\epsilon}{2})F(w^\epsilon)(x,t_0+\epsilon ) \bigr] \dd x \\
&\leq C \Bigl[ \norm{w(\frarg ,t_0)-w^\epsilon(\frarg ,t_0+ \frac{\epsilon}{2})}_{L^2(B)}\norm{w^\epsilon(\frarg ,t_0)}_{L^2(B)}\\ 
&+ \norm{w^\epsilon(\frarg,t_0+ \frac{\epsilon}{2})}_{L^2(B)} \norm{w^\epsilon(\frarg ,t_0) - w^\epsilon(\frarg , t_0+ \frac{\epsilon}{2})}_{L^2(B)} \Bigr] .
\end{aligned}
\]
Since $w^\epsilon \in C((T_0+ \epsilon_0 , T_1-\epsilon_0); L^2(B))$, uniformly, we have

Observe that
\begin{equation}\label{eq:testing-linfty-l2}
\begin{aligned}
&\int_B \abs{ w(x,t_0) -w^\epsilon(x,t_0 + \frac{\epsilon}{2})}^2 \dd x =  \int_B \absB{\int_{-\frac{1}{2}}^{\frac{1}{2}} \zeta(-\sigma)[w(x,t_0) - w(x,t_0+ \frac{\epsilon}{2} +\epsilon \sigma)] \dd \sigma}^2 \dd x \\
&\leq \int_B \int_{-\frac{1}{2}}^{\frac{1}{2}} \absb{\zeta(-\sigma)[w(x,t_0) - w(x,t_0+ \frac{\epsilon}{2} +\epsilon \sigma)]}^2 \dd \sigma\dd x = \int_{-\frac{1}{2}}^{\frac{1}{2}} \int_B \absb{\zeta(-\sigma)[w(x,t_0) - w(x,t_0+ \frac{\epsilon}{2} +\epsilon \sigma)]}^2 \dd x \dd \sigma \\
& \leq \int_{-\frac{1}{2}}^{\frac{1}{2}} \int_B \absb{w(x,t_0) - w(x,t_0+ \frac{\epsilon}{2} +\epsilon \sigma)}^2 \dd x \dd \sigma \leq \sup_{ 0\leq t \leq  \epsilon} \int_B \absb{w(x,t_0) - w(x,t_0+ t)}^2 \dd x
\end{aligned} 
\end{equation}
which tends to zero since $w$ is in $C([T_0,T_1],L^2(B))$. In a similar way oe can argue that
\begin{equation}\label{eq:milif-l2-con}
\lim_{\epsilon \to 0} \norm{w^\epsilon(\frarg ,t_0) - w^\epsilon(\frarg , t_0+ \frac{\epsilon}{2})}_{L^2(B)} = 0.
\end{equation}
Using the traingle inequality we get
$$\norm{w^\epsilon(\frarg ,t_0) - w^\epsilon(\frarg , t_0+ \frac{\epsilon}{2})}_{L^2(B)} \leq \norm{w^\epsilon(\frarg ,t_0) - w(\frarg , t_0)}_{L^2(B)} + \norm{ w(\frarg , t_0) - w^\epsilon(\frarg ,t_0+\frac{\epsilon}{2}) }_{L^2(B)}.$$
using a computation similar to \eqref{eq:testing-linfty-l2} we obtain
\[
\norm{w^\epsilon(\frarg ,t_0) - w(\frarg , t_0)}_{L^2(B)} \leq \sup_{-\frac{\epsilon}{2} \leq t \leq \frac{\epsilon}{2}} \norm{w(\frarg,t_0+t)-w(\frarg,t_0)}_{L^2(B)} 
\]
and
\[
\norm{w(\frarg ,t_0) - w^\epsilon(\frarg , t_0+ \frac{\epsilon}{2})}_{L^2(B)} \leq \sup_{0 \leq t \leq \epsilon} \norm{w(\frarg,t_0+t)-w(\frarg,t_0)}_{L^2(B)}.
\]
These two expressons converge to zero, since $w \in C([T_0,T_1],L^2(B))$ and $(t_0-\epsilon,t_1+\epsilon)\Subset (T_0,T_1)$. This shows that $\Ical_1$ converges to zero. By similar reasoning, $\Ical_2$ also tends to zero.
%$$\Ical_2 = \int_{B}\bigl[ (u-v)(x,t_1)F(u^\epsilon - v^\epsilon)(x,t_1) -(u-v)^\epsilon(x,t_1 -\epsilon)F(u^\epsilon - v^\epsilon)(x,t_1-\epsilon) \bigr] \dd x , $$
For the term $\Ical_3$, we have
\[
\begin{aligned}
\absB{ \int_{t_0}^{t_1} &\int_B \Bigl[ (F(w^\epsilon) )^\epsilon f - F(w) f \Bigr] \dd x \dd t } = \absB{ \int_{t_0}^{t_1} \int_B \bigl( F(w^\epsilon))^\epsilon - F(w^\epsilon)  \bigr) f + \bigl( F(w^\epsilon) - F(w) \bigr)f \dd x \dd t } \\
& \leq  \int_{t_0}^{t_1} \int_B \abs{ (F(w^\epsilon))^\epsilon - F(w^\epsilon) } \abs{f(x,t)} + C \abs{w^\epsilon -w}\abs{f(x,t)}  \dd x \dd t . \\
\end{aligned}
\]
The sequence $w^\epsilon$ is bounded in $L^{p_s^\star,p}(B\times(t_0,t_1))$, therefore, it has a weakly convergent subsequence. Using the pointwise convergence of $w^\epsilon$ to $w$, we get the weak convergence of $w^\epsilon - w$ to zero. By the assumptions on $q,r$ together with H\"older's inequality \eqref{eq:Holder}, $f(x,t)$ belongs to the dual space $L^{(p_s^\star)^\prime,p^\prime}(B\times(t_0,t_1))$. Therefore,
$$\lim_{\epsilon \to 0 } \int_{t_0}^{t_1} \int_B \abs{w^\epsilon -w}\abs{f(x,t)}  \dd x \dd t = 0 .$$
On the other hand,
\[
\begin{aligned}
 \int_{t_0}^{t_1} \int_B \abs{ (F(w^\epsilon))^\epsilon &- F(w^\epsilon) } \abs{f(x,t)} \dd x \dd t \\
 &= \int_{t_0}^{t_1} \int_B \absBB{ \int_{-\frac{1}{2}}^{\frac{1}{2}} \zeta(-\sigma) (F(w^\epsilon(x,t+\epsilon \sigma))) - F(w^\epsilon(x,t)) \dd \sigma } \abs{f(x,t)} \dd x \dd t \\
 & \leq  \int_{t_0}^{t_1} \int_B  \int_{-\frac{1}{2}}^{\frac{1}{2}} \zeta(-\sigma) \abs{(F(w^\epsilon(x,t+\epsilon \sigma)) - F(w^\epsilon(x,t))} \abs{f(x,t)} \dd \sigma   \dd x \dd t \\
 & \leq C\int_{t_0}^{t_1} \int_B  \int_{-\frac{1}{2}}^{\frac{1}{2}} \zeta(-\sigma) \abs{(w^\epsilon(x,t+\epsilon \sigma) - w^\epsilon(x,t)} \abs{f(x,t)} \dd \sigma   \dd x \dd t \\
 & \leq C\int_{-\frac{1}{2}}^{\frac{1}{2}} \int_{t_0}^{t_1} \int_B \abs{(w^\epsilon(x,t+\epsilon \sigma) - w^\epsilon(x,t)} \abs{f(x,t)}   \dd x \dd t \dd \sigma \\
 &\leq C\int_{-\frac{1}{2}}^{\frac{1}{2}} \normb{\; \norm{w^\epsilon(x,t+\epsilon \sigma) -w^\epsilon(x,t)}_{L^{p_s^\star}(B)} \;}_{L^p(t_0,t_1)} \norm{f}_{L^{(p_s^\star)^\prime,p^\prime}(B\times(t_0,t_1))} \dd \sigma .
\end{aligned}
\]
Recalling that the shift operator, 
$$T(a)(g):= \norm{g(t+a)}_{L^p((t_0,t_1))} $$
for a function $g \in L^p(t_0-\epsilon_0,t_1+\epsilon_0)$ is continuous for $-\epsilon_0 \leq  a \leq \epsilon_0$. Hence, we get 
$$\lim_{\epsilon \to 0} \norm{w^\epsilon(x,t)}_{L^{p_s^\star,p}(B\times(t_0,t_1))} =\lim_{\epsilon \to 0} \norm{w^\epsilon(x,t+\epsilon \sigma)}_{L^{p_s^\star,p}(B\times(t_0,t_1))} =  \norm{w}_{L^{p_s^\star,p}(B\times(t_0,t_1))}. $$
Upon passing to a subsequence  $w^\epsilon(x,t+\epsilon \sigma)$ and $w^\epsilon(x,t)$ converge weakly in $L^{p_s^\star,p}(B \times (t_0,t_1))$, since they converge to $w(x,t)$ pointwise, we get the weak convergence
$$w^\epsilon(x,t) \toweak w(x,t) \quad \text{and} \quad w^\epsilon(x,t+\epsilon \sigma) \toweak w(x,t) \quad \text{in} \quad L^{p_s^\star,p}(B \times (t_0,t_1)). $$
Combined with the convergence of the norms, this implies the strong convergence in the norm, in particular, we have 
$$\normb{\; \norm{w^\epsilon(x,t+\epsilon \sigma) -w^\epsilon(x,t)}_{L^{p_s^\star}(B)} \;}_{L^p(t_0,t_1)} \to 0.$$
Now we turn our attention to the terms on the left hand side of \eqref{eq:testing-epsilon}. The terms $\Sigma_u(\epsilon)$ and $\Sigma_v(\epsilon)$ converge to zero. To show this we start with the following computation, borrowed from \cite[Lemma 3.3.]{BLS}. Using a suitable change of variables in \eqref{eq:testing-error-term} and recalling $\phi=F(w^\epsilon)$, we can also write
\begin{equation}
\label{sigma}
\begin{split}
\Sigma_u(\epsilon)&=-\int_{B}\left(\int_{-\frac{1}{2}}^{\frac{1}{2}} u(x,t_0 - \epsilon \rho + \frac{\epsilon}{2})\,\zeta( \rho)\, \dd \rho \right)\,F(w^\epsilon)\left(x,t_0+\frac{\varepsilon}{2}\right)\dd x\\
&+\int_{B}\int_{-\frac{1}{2}}^{\frac{1}{2}}  \left(\int_{-\frac{1}{2}}^{\rho} u(x,\varepsilon\,\rho+t_0-\varepsilon\,\sigma)\,\zeta'\left(\sigma\right)\dd \sigma\right)\,F(w^\epsilon)(x,\varepsilon\,\rho+t_0)\dd \rho\dd x\\
&+\int_{B}\left(\int^{\frac{1}{2}}_{-\frac{1}{2}} u(x,t_1-\epsilon\rho -\frac{\epsilon}{2})\,\zeta(\rho)\dd \rho\right)\,F(w^\epsilon)\left(x,t_1-\frac{\varepsilon}{2}\right)\dd x\\
&-\int_{B}\int_{-\frac{1}{2}}^{\frac{1}{2}}  \left(\int^{\rho}_{\frac{1}{2}} u(x,\varepsilon\,\rho+T_1-\varepsilon\,\sigma)\,\zeta'\left(\sigma\right)\dd \sigma\right)\,F(w^\epsilon)(x,\varepsilon\,\rho+T_1)\dd \rho\dd x\\
&:= \Sigma_u^1(\epsilon)+\Sigma_u^2(\epsilon)+\Sigma_u^3(\epsilon)+\Sigma_u^4(\epsilon).
\end{split}
\end{equation}
In a similar way to the argument for convergence of $\Ical_1$, we can see that
$$\lim_{\epsilon \to 0}\Sigma_u^1(\epsilon) = -\int_B u(x,t_0)F(w)(x,t_0) \dd x.$$
We spell out the details of the arguments for convergence of $\Sigma_u^2(\epsilon)$.
\[
\begin{aligned}
\absB{&\Sigma_u^2(\epsilon) - \int_B u(x,t_0)F(w)(x,t_0) \dd x} \\
&= \absBB{\int_B \int_{-\frac{1}{2}}^{\frac{1}{2}}  \left(\int_{-\frac{1}{2}}^{\rho} \left(u(x,\varepsilon\,\rho+t_0-\varepsilon\,\sigma) - u(x,t_0)\right)\,\zeta'\left(\sigma\right)\dd \sigma\right)\,F(w^\epsilon)(x,\varepsilon\,\rho+t_0)\dd \rho\dd x \\
& \;+ \int_B \int_{-\frac{1}{2}}^{\frac{1}{2}}  \left(\int_{-\frac{1}{2}}^{\rho} u(x,t_0)\zeta^\prime(\sigma) \dd \sigma \right) \Bigl( F(w^\epsilon)(x,\epsilon \rho +t_0)-F(w)(x,t_0) \Bigr) \dd \rho \dd x } \\
&\leq \int_{-\frac{1}{2}}^{\frac{1}{2}}  \int_{-\frac{1}{2}}^{\rho}\left(\int_B \absB{u(x,\epsilon \rho + t_0 - \epsilon \sigma) -u(x,t_0))F(w^\epsilon)(x,\epsilon \rho + t_0) \zeta^\prime(\sigma)} \dd x \right) \dd \sigma \dd \rho \\
&\;+ \int_{-\frac{1}{2}}^{\frac{1}{2}} \int_B  \abs{\zeta(\rho)}  \,  \absB{u(x,t_0)\Bigl( F(w^\epsilon)(x,\epsilon \rho +t_0)-F(w)(x,t_0) \Bigr)}\dd x \dd \rho \\
&\leq 8 \int_{-\frac{1}{2}}^{\frac{1}{2}}  \int_{-\frac{1}{2}}^{\rho} \norm{u(\frarg, \epsilon \rho + t_0 - \epsilon \sigma)-u(\frarg,t_0)}_{L^2(B)} \norm{F(w^\epsilon)(\frarg, \epsilon \rho + t_0)}_{L^2(B)} \dd \sigma \dd \rho \\
&\; + \int_{-\frac{1}{2}}^{\frac{1}{2}} \norm{u(\frarg,t_0)}_{L^2(B)} \norm{F(w^\epsilon)(\frarg,\epsilon \rho + t_0) - F(w)(\frarg,t_0)}_{L^2(B)} \dd \rho \\
& \leq 8 \sup_{t_0 \leq t \leq t_0+\epsilon } \norm{u(\frarg,t_0+t)-u(\frarg,t_0)}_{L^2(B)} \sup_{t_0-\frac{\epsilon}{2}\leq t \leq t_0+\frac{\epsilon}{2}} \norm{F(w^\epsilon)(\frarg,t_0+t)}_{L^2(B)} \\
&+ C\norm{u(\frarg,t_0)}_{L^2(B)} \sup_{t_0-\frac{\epsilon}{2}\leq t \leq t_0+ \frac{\epsilon}{2}} \norm{w^\epsilon(\frarg,t_0+t)-w(\frarg,t_0)}_{L^2(B)},
\end{aligned}
\]
where $C$ is the Lipschitz constant of $F$. We have used $\abs{\zeta} \leq 1$ and $\abs{\zeta^\prime} \leq 8$ in the computation. Since $u\in C([T_0,T_1];L^2(B))$, we get
$$\lim_{\epsilon \to 0} \sup_{t_0 \leq t \leq t_0+\epsilon } \norm{u(\frarg,t_0+t)-u(\frarg,t_0)}_{L^2(B)} =0.$$
Using a computation similar to \eqref{eq:testing-linfty-l2} we obtain
$$ \sup_{t_0-\frac{\epsilon}{2}\leq t \leq t_0+ \frac{\epsilon}{2}} \norm{w^\epsilon(\frarg,t_0+t)-w(\frarg,t_0)}_{L^2(B)} \leq \sup_{t_0-\epsilon \leq t \leq t_0+ \epsilon} \norm{w(\frarg,t_0+t)-w(\frarg,t_0)}_{L^2(B)}  .$$
This converges to zero since $w \in C([T_0,T_1];L^2(B))$, and $(t_0-\epsilon,t_1+\epsilon) \Subset (T_0,T_1)$ due to the choice of $\epsilon$. In conclusion 
$$\lim_{\epsilon \to 0} \Sigma_u^1(\epsilon) + \Sigma_u^2(\epsilon)= 0.$$
In a similar fashion, we can argue that
$$\lim_{\epsilon \to 0} \Sigma_u^3(\epsilon) + \Sigma_u^4(\epsilon)= 0.$$
Hence, $\lim_{\epsilon \to 0} \Sigma_u(\epsilon)=0$. The treatment of $\Sigma_v(\epsilon)$ is similar.

The term 
$$\int_B \Fcal (u^\epsilon - v^\epsilon) \dd x \Bigr]_{t_0+\frac{\epsilon}{2}}^{t_1-\frac{\epsilon}{2}}= \int_B \Fcal (w^\epsilon)(x , t_1-\frac{\epsilon}{2}) \dd x - \int_B \Fcal (w^\epsilon)(x, t_0 +\frac{\epsilon}{2}) \dd x, $$
converges to 
$$\int_B \Fcal (w)( x , t_1) \dd x - \int_B \Fcal (w)(x , t_0 ) \dd x.$$
To show this, we consider two cases. 

\textbf{Case A:} $F$ is bounded. In this case $\Fcal$ is globally Lipschitz, that is  $ \abs{\Fcal(a) - \Fcal(b)} \leq C \abs{a-b}$, therefore,
\[
\begin{aligned}
\absB{ \int_B \Fcal(w^\epsilon(x,t_0+\frac{\epsilon}{2})) - \Fcal(w(x,t_0)) \dd x} &\leq \int_B C\abs{w^\epsilon(x,t_0+ \frac{\epsilon}{2}) - w(x,t_0)} \dd x  \\
& \leq C \abs{B}^{\frac{1}{2}} \norm{w^\epsilon(\frarg,t_0+\frac{\epsilon}{2}) - w( \frarg ,t_0)}_{L^2(B)} \dd x,
\end{aligned}
\]
which converges to zero as was explained before, see \eqref{eq:milif-l2-con}.

\textbf{Case B:} In this case, we have $\Fcal(a)=a^2$. Therefore,
\[
\begin{aligned}
\absB{ \int_B \Fcal(w^\epsilon(x,t_0 &+\frac{\epsilon}{2})) - \Fcal(w(x,t_0)) \dd x} \leq  
\int_B \abs{w^\epsilon(x,t_0+\frac{\epsilon}{2}))^2 - w(x,t_0)^2} \dd x \\
& \leq \int_B \abs{w^\epsilon(x,t_0+\frac{\epsilon}{2})) - w(x,t_0)}\abs{w^\epsilon(x,t_0+\frac{\epsilon}{2})) - w(x,t_0)} \dd x \\
& \leq \norm{w^\epsilon(\frarg ,t_0+\frac{\epsilon}{2})) - w(\frarg ,t_0)}_{L^2(B)} \norm{w^\epsilon(\frarg ,t_0+\frac{\epsilon}{2})) + w(\frarg ,t_0)}_{L^2(B)}
\end{aligned}
\]
and since $w \in C([T_0,T_1];L^2(B))$, with an argument similar to the treatment of $\Ical_1$, as we let $\epsilon $ go to zero this term converges to zero.

Now we discuss the convergence of the nonlocal term. Our treatment is similar to the argument in \cite[Appendix B]{BLS}. The aim is to show that the following converges to zero.
\[
\begin{aligned}
\int_{t_0}^{t_1} &\iint_{\R^n \times \R^n} (J_p(u(x)-u(y))- J_p(v(x) - v(y))) 
\times \bigl[ (F(w^\epsilon(x,t)) )^\epsilon - F(w(x,t)   \\
 &- \bigl( (F(w^\epsilon(y,t)) )^\epsilon - F(w(y,t)) \bigr) \bigr] \dd \mu \dd t.
\end{aligned}
\]
 We split it into the two parts
 \[
 \begin{aligned}
\int_{t_0}^{t_1} &\iint_{B_2 \times B_2} (J_p(u(x)-u(y))- J_p(v(x) - v(y))) 
\times \bigl[ (F(w^\epsilon(x,t)) )^\epsilon - F(w(x,t)   \\
 &- \bigl( (F(w^\epsilon(y,t)) )^\epsilon - F(w(y,t)) \bigr) \bigr] \dd \mu \dd t \\
 &+ 2 \int_{t_0}^{t_1} \iint_{B \times \, (\R^n \setminus B_2) } (J_p(u(x)-u(y))- J_p(v(x) - v(y))) 
\times \bigl[ (F(w^\epsilon(x,t)) )^\epsilon - F(w(x,t)  \bigr] \dd \mu \dd t \\
&:= \Theta_1(\epsilon) + 2\Theta_2(\epsilon).
\end{aligned}
 \]
 Here we have used the boundary condition $u=v (w=0) $ for $y \in \R^n \setminus B$ .
 Since $\abs{F(a)-F(b)} \leq C \abs{a-b}$ we have 
 \[
 \begin{aligned}
 \int_{t_0}^{t_1} \norm{(F(w^\epsilon))^\epsilon}_{W^{s,p}(B_2)}^p \dd t \leq C \int_{t_0}^{t_1} \norm{w^\epsilon}_{W^{s,p}(B_2)}^p \dd t.
 \end{aligned}
 \]
 After passing to a subsequence this sequence converges weakly in $L^p((t_0,t_1);W^{s,p}(B_2))$ to $F(w(x,t))$ or in another words 
$$
\frac{(F(w^\epsilon(x,t))^\epsilon  - (F(w^\epsilon(y,t) ))^\epsilon }{\abs{x-y}^{\frac{n}{p} +s}}
$$
converges weakly in $L^p \bigl((t_0,t_1);L^p(B_2 \times B_2) \bigr)$ and since 
 $$\frac{J_p(u(x) - u(y)) - J_p(v(x)-v(y))}{\abs{x-y}^{\frac{n}{p^\prime}+(p-1)s}}$$
 belongs to $L^{p^\prime} \bigl((t_0,t_1); L^{p^\prime}(B_2 \times B_2) \bigr)$ we get the desired convergence for $\Theta_1(\epsilon)$. Now for $\Theta_2(\epsilon)$ consider
$$
\Gcal(x,t):= \int_{\R^n \setminus B_2} \frac{J_p(u(x)-u(y))-J_p(v(x)-v(y))}{\abs{x-y}^{n+sp}} \dd y.
$$
Then for almost every $x \in B $
\begin{equation}\label{eq:gcal-tail}
\begin{aligned}
\abs{\Gcal(x,t)} &\leq C(n,s,p) \int_{\R^n \setminus B_2} \frac{\abs{u(x,t)}^{p-1}+\abs{u(y,t)}^{p-1}+\abs{v(x,t)}^{p-1}+\abs{v(y,t)}^{p-1}}{\abs{y}^{n+sp}} \dd y \\
&\leq C \bigl( 2\mathrm{Tail}_{p-1,sp}(u(\frarg,t);0,2)^{p-1} +\abs{u(x,t)}^{p-1} + \abs{v(x,t)}^{p-1} \bigr)
\end{aligned}
\end{equation}
The terms $\abs{u(x,t)}^{p-1}$ and $\abs{v(x,t)}^{p-1}$ belongs to $L^{p^\prime}\bigl((t_0,t_1); L^{p^\prime}(B)\bigr)$ since $u,v \in L^p((t_0,t_1); L^p(B))$. The tail term its independent of $x$ and belongs to $L^{p^\prime}(t_0,t_1)$ by the assumption
$$\int_{T_0}^{T_1} \bigl(\mathrm{Tail}_{p-1,sp}(u(\frarg,t);0,2))\bigr)^{p^\prime} \dd t \leq \infty.$$
Thus, $\Gcal(x,t) \in L^{p^\prime}([T_0,T_1]; L^{p^\prime}(B_2))$ and as before after extracting a subsequence:
$$ F(w^\epsilon(x,t))^\epsilon  \rightharpoonup F(w(x,t)  \quad \text{in} \; L^p([t_0,t_1]; L^p(B)). $$
This shows that 
$$\Theta_2(\epsilon)= \int_{t_0}^{t_1} \int_B \Gcal(x,t) \bigl( F(w^\epsilon(x,t))^\epsilon - F(w(x,t) \bigr)$$
converges to zero.

Finally, we let $\epsilon_0$ go to zero to get the desired result for $[T_0,T_1]$.
We need to show that the following converge to zero as $\epsilon_0$ tends to $0$.
$$\Jcal_1 := \int_B \Fcal(w(x,T_0)) - \Fcal(w(x,T_0+\epsilon_0)) \dd x \qquad \text{,} \qquad \Jcal_2 := \int_B \Fcal(w(x,T_1)) - \Fcal(w(x,T_1 -\epsilon_0)) \dd x,  $$
$$\Jcal_3 := \int_{T_0}^{T_0+ \epsilon_0} \int_B F(w(x,t))f(x,t) \dd x \dd t \qquad , 
\qquad \Jcal_4 := \int_{T_1-\epsilon_0}^{T_1 } \int_B F(w(x,t))f(x,t) \dd x \dd t , $$
and
$$\Ncal_1 := \int_{T_0}^{T_0+\epsilon_0} \iint_{\R^n \times \R^n} \Bigl( \frac{J_p(u(x,t)- u(u,t)) - J_p(v(x,t)-v(y,t))}{\abs{x-y}^{n+sp}}   \Bigr) \times \Bigl( F(w(x,t)) - F(w(y,t)) \Bigr) \dd x \dd y \dd t,$$
and 
$$\Ncal_2 := \int_{T_1 - \epsilon_0}^{T_1} \iint_{\R^n \times \R^n} \Bigl( \frac{J_p(u(x,t)- u(u,t)) - J_p(v(x,t)-v(y,t))}{\abs{x-y}^{n+sp}}   \Bigr) \times \Bigl( F(w(x,t)) - F(w(y,t)) \Bigr) \dd x \dd y \dd t.$$
The arguments will be reminiscent of the ideas in the previous part.

We start with $\Jcal_2$, in the case of a bounded $F$, $\Fcal$ is globally Lipschitz and we have
\[
\begin{aligned}
\absb{\Jcal_2} &\leq  \int_B \absb{\Fcal(w(x,T_1)) - \Fcal(w(x,T_1-\epsilon_0))} \dd x \leq C \int_B \abs{w(x,T_1) - w(x,T_1-\epsilon_0)} \dd x \\
& \leq C \abs{B}^{\frac{1}{2}} \norm{w(\frarg, T_1) - w(\frarg , T_1 -\epsilon_0)}_{L^2(B)}.
\end{aligned}
\]
This converges to $0$ since $w \in C([T_0,T_1];L^2(B))$, in the case of $F(a)=a$, we have 
\[
\begin{aligned}
\absb{\Jcal_2} &\leq  \int_B \absb{\Fcal(w(x,T_1)) - \Fcal(w(x,T_1-\epsilon_0))} \dd x \leq \int_B \absb{w(x,T_1)^2 - w(x,T_1-\epsilon_0)^2} \dd x \\
& \leq  \int_B \absb{w(x,T_1) - w(x,T_1-\epsilon_0)} \absb{w(x,T_1) + w(x,T_1-\epsilon_0)} \dd x \\
&\leq \norm{w(\frarg, T_1) - w(\frarg , T_1 -\epsilon_0)}_{L^2(B)}\norm{w(\frarg, T_1) + w(\frarg , T_1 -\epsilon_0)}_{L^2(B)}.
\end{aligned}
\]
Again since $w \in C([T_0,T_1];L^2(B))$, this term converges to $0$.

 $\Jcal_1$ can be treated in a similar way. 
 For the term $\Jcal_4$, using $\abs{F(a)} \leq C \abs{a}$ we get
 \[
 \absb{\Jcal_4} \leq C \int_{T_1-\epsilon_0}^{T_1 } \int_B \abs{w(x,t)} \abs{f(x,t)} \dd x \dd t.
 \]
 Since $w \in L^{p_s^\star,p}(B \times [T_0,T_1])$ and $f \in L^{(p_s^\star)^\prime,p^\prime}(B \times [T_0,T_1])$, using H\"older's inequality \eqref{eq:Holder}, one can see that 
 $$w(x,t)f(x,t) \in L^1(B\times[T_0,T_1]).$$ 
 Now using the absolute continuity of the integral for integrable functions we can conclude that $\Jcal_4$ converges to $0$. The reasoning for convergence of $\Jcal_3$ is similar.
 
 Now we turn our attention to the nonlocal terms.
 \[
 \begin{aligned}
 \Ncal_2 &= \int_{T_1 - \epsilon_0 }^{T_1} \iint_{B_2 \times B_2} \Bigl( \frac{J_p(u(x,t)- u(y,t)) - J_p(v(x,t)-v(y,t))}{\abs{x-y}^{n+sp}}   \Bigr) \times \Bigl( F(w(x,t)) - F(w(y,t)) \Bigr) \dd x \dd y \dd t\\
  &+ 2 \int_{T_1-\epsilon_0}^{T_1} \iint_{B \times (\R^n \setminus B_2)} \frac{J_p(u(x,t)- u(y,t)) - J_p(v(x,t)-v(y,t))}{\abs{x-y}^{n+sp}} F(w(x,t)) \dd x \dd y \dd t\\
  &\quad:= \Theta_1 + 2\Theta_2
  \end{aligned}
 \]
 First, we treat $\Theta_1$. Notice that since $u,v \in L^p([T_0,T_1];W^{s,p}(B_2))$ we have
$$\frac{J_p(u(x,t)- u(y,t)) - J_p(v(x,t)-v(y,t))}{\abs{x-y}^{\frac{n}{p^\prime}+(p-1)s}} \in L^{p^\prime}([T_0,T_1]; L^{p^\prime}(B_2 \times B_2)]), $$
and using Lipschitz continuity of $F$ and the fact that $w \in L^p([T_0,T_1];W^{s,p}(B_2)) $ we have
$$\frac{F(w(x,t)) - F(w(y,t))}{\abs{x-y}^{\frac{n}{p}+s}} \in L^p([T_0,T_1];L^p(B_2 \times B_2)).$$
This implies that the integrand involved in $\Theta_1$ belongs to $L^1([T_0,T_1];L^1(B_2\times B_2))$. And similar to the treatment of $\Jcal_4$, since the volume of the integration region is shrinking to $0$, $\Theta_1$ converges to $0$.
To deal with $\Theta_2$, notice that
$$F(w(x,t)) \in L^p([T_0,T_1];L^p(B))$$
and define
$$\Gcal(x,t):= \int_{\R^n \setminus B_2} \frac{J_p(u(x,t)- u(y,t)) - J_p(v(x,t)-v(y,t))}{\abs{x-y}^{n+sp}} \dd y.$$
We can estimate this integration in terms of the Tail, that is
$$\absb{\Gcal(x,t)} \leq C(n,s,p)\Bigl( \mathrm{Tail}_{p-1,sp}(u(\frarg,t);0,2) + \abs{u(x,t)}^{p-1} + \abs{v(x,t)}^{p-1} \Bigr),$$
see for example \eqref{eq:gcal-tail}. Therefore, $\Gcal(x,t) \in L^{p^\prime}([T_0,T_1]; L^{p^\prime}(B))$. Hence using H\"older's inequality 
$$\Gcal(x,t)F(x,t) \in L^1([T_0,T_1];L^1(B)).$$
This concludes the result. $\Ncal_1$ can be treated in an exactly similar manner.
\end{proof}

\providecommand{\bysame}{\leavevmode\hbox to3em{\hrulefill}\thinspace}
\providecommand{\MR}{\relax\ifhmode\unskip\space\fi MR }
\providecommand{\MRhref}[2]{%
  \href{http://www.ams.org/mathscinet-getitem?mr=#1}{#2}
}
\providecommand{\href}[2]{#2}

\end{document}